\documentclass[11pt,a4paper]{article}

\usepackage{geometry} % Margins
\usepackage{amsmath,amsfonts,amsthm,amssymb,yhmath, mathtools} % Maths
\usepackage[dvipsnames]{xcolor}  % Colors
\usepackage{enumitem, comment} % Lists and enumerations
\usepackage{graphicx, float, subfig, overpic} %Figures
\usepackage{stmaryrd} %Normes rares
\usepackage{faktor, url} % Quocients
\usepackage{mathrsfs} % Caligrafiques
\usepackage{titlesec} %Titols capitols
\usepackage[none]{hyphenat} % Evitar hyphenation
\usepackage[font=small,labelfont=bf,tableposition=top]{caption}
\usepackage{authblk} % Authors

\usepackage{hyperref} % Hyperlinks
\hypersetup{
	colorlinks=true,
	linkcolor=blue,
	citecolor=green,
}

\usepackage[numbers]{natbib}
\graphicspath{{Figures/}}
\numberwithin{equation}{section}
\newcommand{\st}{\h:\h}

\newcommand{\h}{\hspace{1mm}}

\newcommand{\sprod}[2]{\langle #1, #2 \rangle}

\theoremstyle{plain}
\newtheorem{lemma}{Lemma}[section]
\newtheorem{corollary}[lemma]{Corollary}
\newtheorem{proposition}[lemma]{Proposition}
\newtheorem{remark}[lemma]{Remark}

\newtheorem{definition}[lemma]{Definition}

\newtheorem{theorem}[lemma]{Theorem}
\newtheorem{conjecture}[lemma]{Conjecture}

\newtheoremstyle{named}{}{}{\itshape}{}{\bfseries}{.}{.5em}{\thmname{#1}\thmnumber{ #2}. (\thmnote{#3})}
\theoremstyle{named}
\newtheorem{namedTheorem}{Theorem}
\newtheorem*{namedTheorem*}{Theorem}
\newtheorem{namedCorollary}{Corollary}

\def\be{\begin{eqnarray}}
\def\ee{\end{eqnarray}}
\def\beal{\begin{aligned}}
\def\enal{\end{aligned}}

\newcommand{\norm}[1]{\left\lVert#1\right\rVert}
\newcommand{\norms}[1]{\lVert#1\rVert}

\newcommand{\vabs}[1]{\left| #1 \right|}
\newcommand{\vabss}[1]{| #1 |}
\newcommand{\paren}[1]{\left(#1\right)}
\newcommand{\claus}[1]{\left\{#1\right\}}
\newcommand{\boxClaus}[1]{\left[#1\right]}

\newcommand{\ol}[1]{\overline{#1}}
\newcommand{\conj}[1]{\overline{#1}}
\newcommand{\tl}{\tilde}
\newcommand{\wt}{\widetilde}
\newcommand{\wh}{\widehat}
\renewcommand{\Re}{\mathrm{Re\, }}
\renewcommand{\Im}{\mathrm{Im\,}}
\renewcommand{\arg}{\mathrm{arg\,}}

\newcommand{\Diff}{\mathrm{Diff}}

\newcommand{\reals}{\mathbb{R}}
\newcommand{\naturals}{\mathbb{N}}
\newcommand{\complexs}{\mathbb{C}}
\newcommand{\torus}{\mathbb{T}}

\newcommand{\integers}{\mathbb{Z}}

\newcommand{\al}{\alpha}

\newcommand{\de}{\delta}
\newcommand{\D}{\Delta}

\newcommand{\la}{\lambda}
\newcommand{\La}{\Lambda}
\newcommand{\g}{\gamma}
\newcommand{\G}{\Gamma}
\newcommand{\s}{\sigma}
\newcommand{\Si}{\Sigma}
\newcommand{\tht}{\theta}

\newcommand{\phiA}{\varphi}

\newcommand{\zetaB}{\tilde{\zeta}}

\newcommand{\Dt}[2]{\dfrac{d #1}{d {#2}}}

\newcommand{\AAA}{\mathcal{A}}
\newcommand{\BB}{\mathcal{B}}
\newcommand{\CC}{\mathcal{C}}
\newcommand{\DD}{\mathcal{D}}

\newcommand{\FF}{\mathcal{F}}
\newcommand{\GG}{\mathcal{G}}

\newcommand{\HH}{\mathcal{H}}
\newcommand{\II}{\mathcal{I}}

\newcommand{\LL}{\mathcal{L}}

\newcommand{\OO}{\mathcal{O}}
\newcommand{\PP}{\mathcal{P}}

\newcommand{\RRR}{\mathcal{R}}
\newcommand{\SSS}{\mathcal{S}}

\newcommand{\UU}{\mathcal{U}}
\newcommand{\VV}{\mathcal{V}}
\newcommand{\WW}{\mathcal{W}}
\newcommand{\XX}{\mathcal{X}}

\newcommand{\Lip}{\mathrm{Lip}}
\definecolor{myGreen}{RGB}{0, 200, 0}
\definecolor{myOrange}{RGB}{255, 100, 0}
\definecolor{myYellow}{RGB}{255, 200, 0}
\definecolor{myBlue}{RGB}{0, 200, 255}
\definecolor{myPurple}{RGB}{200, 0, 200}

\newcommand{\Id}{\mathrm{Id}}

\usepackage{stackengine,scalerel}
\stackMath

\newcommand{\unstable}{{\mathrm{u}}}
\newcommand{\stable}{{\mathrm{s}}}

\newcommand{\CInn}{\Theta}
\newcommand{\UReals}{U_{\reals}}
\newcommand{\UComplexs}{U_{\complexs}}

\newcommand{\HInicial}{h}

\newcommand{\Poi}{\mathrm{Poi}}

\newcommand{\scaB}{\mathrm{sca}}
\newcommand{\equi}{\mathrm{eq}}
\newcommand{\osc}{\mathrm{osc}}

\newcommand{\Lyap}{\mathrm{Lya}}

\newcommand{\red}{\mathrm{red}}
\newcommand{\local}{\mathrm{loc}}

\newcommand{\pend}{\mathrm{p}}
\newcommand{\lap}{\la_{\mathrm{p}}}
\newcommand{\Lap}{\La_{\mathrm{p}}}
\newcommand{\ladp}{\dot{\la}_{\mathrm{p}}}
\newcommand{\Ladp}{\dot{\La}_{\mathrm{p}}}
\newcommand{\separatrix}{\sigma_{\mathrm{p}}}

\newcommand{\PiExtA}{\Pi_{A,\betaBow}}

\newcommand{\Ltres}{\mathfrak{L}}
\newcommand{\LtresLa}{\mathfrak{L}_{\Lambda}}
\newcommand{\Ltresx}{\mathfrak{L}_x}
\newcommand{\Ltresy}{\mathfrak{L}_y}

\newcommand{\Law}{J}
\newcommand{\Iw}{I}
\newcommand{\Iwh}{\widehat{I}_{\rho,\de}}
\newcommand{\po}{\mathfrak{P}}
\newcommand{\lapo}{\lambda_{\mathfrak{P}} }
\newcommand{\Lapo}{\Lambda_{\mathfrak{P}} }
\newcommand{\xpo}{x_{\mathfrak{P}} }
\newcommand{\ypo}{y_{\mathfrak{P}} }

\newcommand{\omegar}{\omega_{\rho,\de}}
\newcommand{\torusC}{\torus_{d}}
\newcommand{\banda}{d}

\newcommand{\ZcalR}{\mathcal{Z}}
\newcommand{\normZR}[1]{\lVert #1 \rVert}
\newcommand{\normZRprod}[1]{\lVert #1 \rVert^{\times}}

\newcommand{\zuU}{z^{\mathrm{u}}}

\newcommand{\zsU}{z^{\mathrm{s}}}

\newcommand{\zdU}{z^{\diamond}}

\newcommand{\zuUper}{z_1^{\mathrm{u}}}

\newcommand{\zsUper}{z_1^{\mathrm{s}}}

\newcommand{\zdUper}{z_1^{\diamond}}

\newcommand{\DmigU}{D^{\mathrm{u}}}
\newcommand{\DmigS}{D^{\mathrm{s}}}
\newcommand{\DmigD}{D^{\mathrm{\diamond}}}
\newcommand{\zuD}{Z^{\mathrm{u}}}

\newcommand{\zsD}{Z^{\mathrm{s}}}

\newcommand{\zdD}{Z^{\diamond}}

\newcommand{\zuDper}{Z_1^{\mathrm{u}}}
\newcommand{\zsDper}{Z_1^{\mathrm{s}}}

\newcommand{\zdDper}{Z_1^{\diamond}}
\newcommand{\Ycal}[1]{\mathcal{Y}_{#1}}
\newcommand{\normX}[2]{\lVert#1\rVert_{#2}}

\newcommand{\normY}[2]{\lVert#1\rVert_{#2}}

\newcommand{\normYprod}[2]{\lVert#1\rVert_{#2}^{\times}}

\newcommand{\tauU}{\tau^{\mathrm{u}}}
\newcommand{\tauS}{\tau^{\mathrm{s}}}

\newcommand{\qG}{v}
\newcommand{\pG}{w}

\newcommand{\qU}{v_1}
\newcommand{\wqU}{\widehat{v}_1}
\newcommand{\qUu}{v_1^{\mathrm{u}}}
\newcommand{\qUs}{v_1^{\mathrm{s}}}
\newcommand{\wqUu}{\widehat{v}_1^{\mathrm{u}}}
\newcommand{\wqUs}{\widehat{v}_1^{\mathrm{s}}}
\newcommand{\pU}{w_1}
\newcommand{\wpU}{\widehat{w}_1}
\newcommand{\pUu}{w_1^{\mathrm{u}}}
\newcommand{\pUs}{w_1^{\mathrm{s}}}
\newcommand{\wpUu}{\widehat{w}_1^{\mathrm{u}}}
\newcommand{\wpUs}{\widehat{w}_1^{\mathrm{s}}}
\newcommand{\qD}{v_2}
\newcommand{\wqD}{\widehat{v}_2}
\newcommand{\qDu}{v_2^{\mathrm{u}}}
\newcommand{\qDs}{v_2^{\mathrm{s}}}
\newcommand{\wqDu}{\widehat{v}_2^{\mathrm{u}}}
\newcommand{\wqDs}{\widehat{v}_2^{\mathrm{s}}}
\newcommand{\pD}{w_2}
\newcommand{\wpD}{\widehat{w}_2}
\newcommand{\pDu}{w_2^{\mathrm{u}}}
\newcommand{\pDs}{w_2^{\mathrm{s}}}
\newcommand{\wpDu}{\widehat{w}_2^{\mathrm{u}}}
\newcommand{\wpDs}{\widehat{w}_2^{\mathrm{s}}}

\newcommand{\betaBow}{\beta}
\newcommand{\rhoPeriodicOrbit}{\rho_0}

\newcommand{\cttTheoScalingDomainA}{c_0}
\newcommand{\cttTheoScalingDomainB}{c_1}
\newcommand{\cttTheoScaling}{b_0}
\newcommand{\cttTheoLtres}{b_1}
\newcommand{\cttPeriodicOrbit}{b_2}
\newcommand{\cttExistenciaD}{b_3}
\newcommand{\cttExistenciaU}{b_3} %ULL
\newcommand{\cttExistenciaUD}{b_4}
\newcommand{\cttExistenciaDtech}{b_3} %ULL
\newcommand{\cttPeriodicOrbitTech}{b_6}
\newcommand{\rhoNormalForm}{\varrho_0}

\newcommand{\rhoGlobal}{\widehat{\varrho}_0}
\newcommand{\rhoLocal}{\varrho_0}

\title{Coorbital homoclinic and chaotic dynamics in the
Restricted 3-Body Problem} 
\author[1,4]{Inmaculada Baldom\'a}
\author[2]{Mar Giralt\thanks{Corresponding author.\\
		\emph{E-mail adresses:} \href{mailto:immaculada.baldoma@upc.edu}{immaculada.baldoma@upc.edu} (I. Baldom\'a), 
		\href{mailto:mar.giralt@obspm.fr}{mar.giralt@obspm.fr} (M. Giralt),
		\href{mailto:guardia@ub.edu}{guardia@ub.edu} (M. Guardia).}}
\author[3,4]{Marcel Guardia}
\affil[1]{Departament de Matem\`atiques \& IMTECH, Universitat Polit\`ecnica de Catalunya, Diagonal 647, 08028 Barcelona, Spain}
\affil[2]{IMCCE, CNRS, Observatoire de Paris, Universit\'e PSL, Sorbonne Universit\'e, 77 Avenue Denfert-Rochereau, 75014 Paris, France}
\affil[3]{Departament de Matem\`atiques i Inform\`atica, Universitat de Barcelona, Gran Via, 585, 08007 Barcelona, Spain}
\affil[4]{Centre de Recerca Matem\`atica, Campus de Bellaterra, Edifici C, 08193 Barcelona, Spain}
\date{\today}

\begin{document}

\maketitle 

%% ABSTRACT
\begin{abstract}
%
% \textcolor{red}{Canviar universitats i mails!}
The description of unstable motions in the  Restricted Planar Circular $3$-Body Problem, modeling the dynamics of a Sun-Planet-Asteriod system, is one of the fundamental problems in Celestial Mechanics.
%
%If the primaries perform circular motions and the massless body is coplanar with them, one has the Restricted Planar Circular $3$-Body Problem (RPC$3$BP).
%
The goal of this paper is to analyze  homoclinic and instability phenomena at coorbital motions,  that is when   the  negligible mass Asteroid is at $1:1$ mean motion resonance with the Planet (i.e. nearly equal periods)  and performs close to circular motions. Several bodies in our Solar system belong to such regimes.
%
% When the ratio between the masses of the primaries $\mu$ is small, the modulus of the hyperbolic eigenvalues are weaker, by a factor of order $\sqrt\mu$, than the elliptic ones.
% %
% Due to the rapidly rotating dynamics, the $1$-dimensional unstable and stable manifold of $L_3$ are exponentially close to each other with respect to $\sqrt\mu$.
%
%In \cite{articleInner, articleOuter, BCGG23}, the authors provided an asymptotic formula for the distance between the 1-dimensional invariant manifolds of $L_3$ for small  ratio between the primaries masses.
%
% This result relies on a Stokes constant which we assume that is non zero.
%

In this paper, we obtain the following results.
% and under this assumption,
%we study different chaotic and homoclinic phenomena occurring in a neighborhood of $L_3$ and its invariant manifolds.
First, we prove that, for a  sequence of ratios between the masses of the Planet and the Sun going to 0, there  exist a $2$-round homoclinic orbit to the Lagrange point $L_3$, i.e. homoclinic orbits that approach the critical point twice.
%  result concerns the existence of $2$-round homoclinic connections to $L_3$.
%
%
%Second, we analyze the invariant manifolds of the  family of Lyapunov periodic orbits of $L_3$ with Hamiltonian energy level exponentially close (with respect to the mass ratio) to that of $L_3$.
%
%We show that there exists a set of periodic orbits whose unstable and stable manifolds intersect transversally.
%
%By the Smale-Birkhoff  Theorem, 
Second, we construct chaotic motions (hyperbolic sets with symbolic dynamics) as a consequence of the existence  of transverse homoclinic orbits to Lyapunov periodic orbits associated to $L_3$.
%exponentially close to $L_3$ and its invariant manifolds.
%
Finally, we prove that the RPC3BP possesses Newhouse domains by proving that  the energy level  unfolds generically a quadratic homoclinic tangency to a periodic orbit.
%, which leads to the existence of Newhouse domains.
%
%In this region of the phase space lie the   Lagrange points $L_1,..,L_5$.
%
%On the original coordinates, they correspond to periodic orbits with the same period as the two primaries, i.e they are on a $1:1$ mean motion resonance or on a coorbital motion.
%The  point $L_3$ is a saddle-center  which is collinear with the primaries and beyond the largest one. Its  invariant manifolds play a fundamental role in structuring the global dynamics at coorbital motions.
\end{abstract}

\newpage
\tableofcontents

%% ESQUEMA 

\newpage

\section{Introduction} 
\label{section:introduction}
% Idees per la intro
% \begin{itemize}
% % \item Fer la secci\'o 1 sabent que $\Theta\neq 0$. 
% % \item Aix\`o vol dir moure el teorema A  a una secci\'o despr\'es de Newhouse.
% % \item La secci\'o 1.1 comen\c{c}a amb el corolari que diu que no hi ha one-round. No cal posar l'Ansatz en lloc.
% \item A la secci\'o 1.5 s'introdueix polars, es posa el Teorema A i s'explica com es demostraran els altres teoremes a a partir de l'A.
% \item Caldr\`a veure que aix\`o encaixa amb la 2 i els inicis de la 3 i la 4.
% \item A la 2 s'ha d'afegir m\'es sobre polars.
% \end{itemize}

One of the oldest questions in dynamical systems is to assert whether the Solar System is stable. More precisely, consider the $N$ body problem, that is the motion of $N$ punctual masses under Newtonian gravitational force in the planetary regime (one massive body, the Sun, and $N-1$ light bodies, the planets). For this model, one would like  to determine whether the orbits of the planets stay close to ellipses over long time scales or undergo strong deviations due to the mutual gravitational interaction between planets.

The Arnold-Herman-F\'ejoz  Theorem ensures that there is  a positive measure set of stable motions lying on   quasiperiodic invariant tori, see~\cite{Arnold63,Rob95,Fejoz04,ChPi11}. 
However, the ``gaps'' left by the invariant tori  in the phase space leave room for instability.
M. Herman, in his ICM lecture~\cite{Her98}, referred to this problem as ``the oldest problem in dynamical systems'' and, related to it, he posed the following conjecture. Consider the $N$-body problem in space, with $N\geq 3$ and assume that the center of mass is fixed at the origin and that, on the energy surface of level $e$, the flow is  $\mathcal{C}^\infty$-reparametrized  such that %the flow is complete : we replace $H$ by $H_e = \varphi_e.(H- e)$ so that
the collisions now occur only in infinite time.

\begin{conjecture}
  \label{conj:global}
  Is for every $e$ the non-wandering set of the Hamiltonian flow of $H_e$ on $H_e^{-1} (0)$ nowhere dense in $H_e^{-1} (0)$?
\end{conjecture}
% 
% %
% In particular, in the last century, one of the fundamental questions has been to understand the measure and ``distribution'' of the wandering and non-wandering sets.
% %
% Indeed, in \cite{Her98}, Michael Herman finishes its survey on important open questions in dynamical systems with two questions in the $N$- Body Problem, one in the general regime and the other in the planetary one. 
% %
% Roughly speaking, these questions are: ``Are the non-wandering sets of the $N$-Body Problem nowhere dense?'' and ``In the planetary setting, is it possible to find wandering domains close to the orbits of the planets?''.}
	
% For instance, one could expect the appearance of instabilities close to mean-motion resonances, see~\cite{FGKR16}.
%
Note that this conjecture is not restricted to the planetary regime but is formulated for any value of the masses of the bodies. This conjecture is nowadays wide open. Even results proving unstable motions in Celestial Mechanics models are rather scarce. Most of these results deal with nearly integrable settings, either the planetary regime or the hierarchical regime (when bodies are increasingly separated) and are tipically of two different types: chaotic motions (i.e. existence of Smale horsehoes) or Arnold diffusion (see Section \ref{subsection:introStateArt} for references).
%These unstable motions rely on analyzing the invariant manifolds of invariant objects and their transverse intersections.

One of the main sources of instabilities are resonances where, typically, hyperbolic invariant objects with invariant manifolds appear. These invariant manifolds structure the global dynamics and act as ``highways'' for the unstable motions. Among these resonances, mean motion resonances play a fundamental role in the global dynamics of the Solar System (see, for instance,  \cite{morbidelli2002, FGKR16}). They appear when two (or more) bodies have rationally dependent periods. 

The aim of this article is to study  instability  phenomena   and  how (some) invariant manifolds structure the global dynamics at the $1:1$ mean motion resonance nearly circular orbits. Such region of the phase space is usually called  \emph{coorbital motions}, since two of the bodies, at short time scales, perform approximately the same circular orbit. Many bodies in our Solar System (satellites, asteroids) belong to this region. We focus on the simplest model where such dynamics arise, that is  the  Restricted Planar Circular $3$-Body Problem (RPC$3$BP).
% 
% \textcolor{red}{I potser treure aix\`o:
% The understanding of the planetary motions, and in particular of its stability or instability,  has been a fundamental field of study in the last centuries.
% %
% A significant model to approximate and understand the motions of different celestial bodies is the $3$-Body Problem and its various simplified models.
% %
% Indeed, since Poincar\'e  (see \cite{Poincare1890}), one of the cornerstone problems of dynamical systems has been to understand how the invariant manifolds of the different invariant objects (periodic orbits, invariant tori) structure the global dynamics of the $3$-Body Problem.}
% %

The Restricted Circular $3$-Body Problem 
models the motion of a body of negligible mass under the gravitational influence of two massive bodies, called the primaries, which perform a circular motion.
If one also assumes that the massless body moves on the same plane as the primaries one has the  RPC$3$BP.

Let us name the two primaries $S$ (star) and $P$ (planet) and
normalize their masses so that $m_S=1-\mu$ and $m_P=\mu$, with $\mu \in \left( 0, \frac{1}{2} \right]$. 
Choosing a suitable rotating coordinate system, the positions of the primaries can be fixed at $q_S=(\mu,0)$ and  $q_P=(\mu-1, 0)$ and then, the position and momenta of the third body, $(q,p) \in \reals^2 \times \reals^2$, are governed by the Hamiltonian system associated to the two degrees of freedom Hamiltonian
\begin{equation}\label{def:hamiltonianInitialNotSplit} 
	h(q,p;\mu)=h_0(q,p) + \mu h_1(q;\mu)
%	\begin{split}
%		\HInicial(q,p;\mu) &= \frac{||p||^2}{2} 
%		- q^T \left( \begin{matrix} 0 & 1 \\ -1 & 0 \end{matrix} \right) p 
%		-\frac{(1-\mu)}{||q-(\mu,0)||} 
%		- \frac{\mu}{||q-(\mu-1,0)||}.
%	\end{split}
\end{equation}
where
\begin{equation}\label{def:hamiltonianInitialSplit} 
	\begin{split}
		\HInicial_0(q,p) &= \frac{||p||^2}{2} 
		- q^T \left( \begin{matrix} 0 & 1 \\ -1 & 0 \end{matrix} \right) p 
		-\frac1{||q||},
		\\
		\mu \HInicial_1(q;\mu) &=
		\frac1{||q||} 
		-\frac{(1-\mu)}{||q-(\mu,0)||} 
		- \frac{\mu}{||q-(\mu-1,0)||}.
	\end{split}
\end{equation}
% This Hamiltonian is that the system is singular when the third body collides with one of the primaries at $q=q_S$ or $q=q_J$.
%
This Hamiltonian is autonomous and the conservation of $h$ corresponds to the preservation of the classical Jacobi constant.

We  analyze this model for $\mu>0$ small enough at coorbital motions. That is, when the orbit of the third body is close to the orbit of the Planet. It is a well known fact that \eqref{def:hamiltonianInitialNotSplit}  has five critical points, usually called Lagrange points, which, for $\mu>0$ small enough, lie at the coorbital motions region, (see Figure~\ref{fig:L3perturbed}).
%
%On an inertial (non-rotating) system of coordinates, the Lagrange points correspond to periodic dynamics with the same period as the two primaries, i.e on a 1:1 mean motion resonance.
%
The three collinear Lagrange points, $L_1$, $L_2$ and $L_3$, are of center-saddle type whereas, for small $\mu$, the triangular ones, $L_4$ and $L_5$, are of center-center type  (see, for instance, \cite{Szebehely}).

\begin{figure}
\centering
\begin{overpic}[scale=0.4]{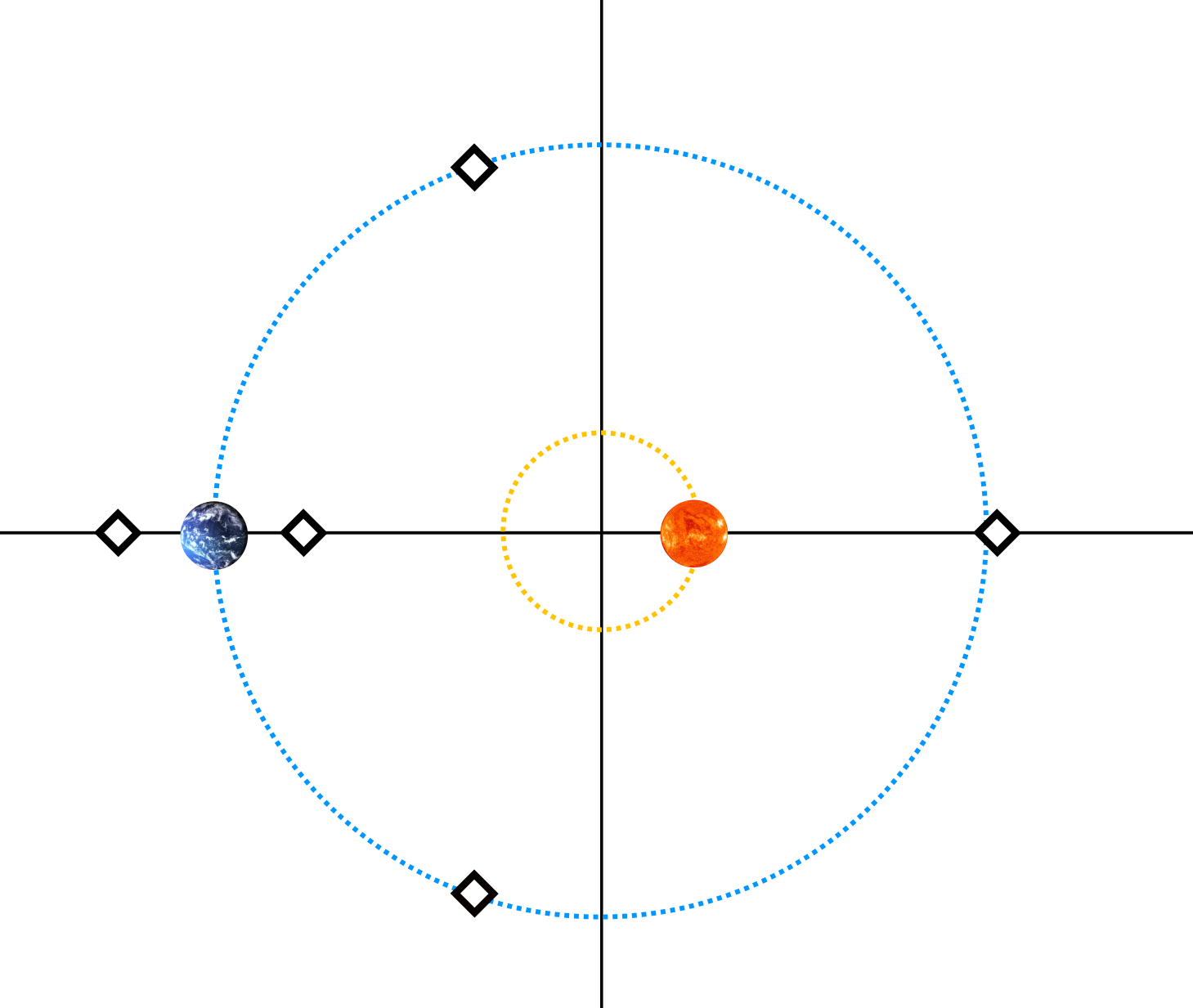}
	\put(60,33){{\color{orange} $S$ }}
	\put(11,31){{\color{blue} $P$ }}
	\put(23,45){{\color{black} $L_1$ }}
	\put(7,45){{\color{black} $L_2$ }}
	\put(87,42){{\color{black} $L_3$ }}
	\put(37,75){{\color{black} $L_5$ }}	
	\put(37,2){{\color{black} $L_4$ }}
	%\put(101,39){$q_1$}
	%\put(49,90){$q_2$}
\end{overpic}
\caption{Projection onto the $q$-plane of the Lagrange equilibrium points for the RPC$3$BP on rotating coordinates.}
	\label{fig:L3perturbed}
\end{figure}
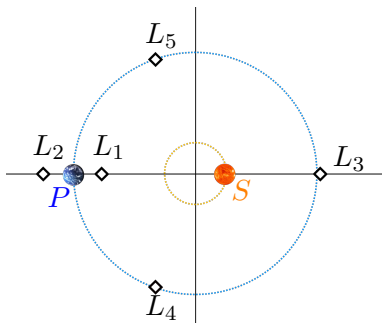
%Due to its interest in astrodynamics, a lot of attention has been paid to the study of the invariant manifolds associated to  the points $L_1$ and $L_2$ (see \cite{KLMR00, GLMS01v1, CGMM04}).
%
%The dynamics around the points $L_4$ and $L_5$ has also been heavily studied since, due to its stability, it is common to find objects orbiting around these points (for instance the Trojan and Greek Asteroids associated to the pair Sun-Jupiter, see \cite{GDFGS89, CelGio90, RobGab06}).
%
There is numerical evidence that the invariant manifolds of  $L_3$ play a fundamental role in structuring the global dynamics at coorbital motions. 
%The purpose of this paper is to analyze those of $L_3$, that is the Lagrange point located ``at the other side'' of the massive primary.
Indeed, its  center-stable and center-unstable invariant manifolds act as boundaries of \emph{effective stability} of the stability domains around $L_4$ and $L_5$ 
(see \cite{GJMS01v4, SSST13}). 
The invariant manifolds of $L_3$ are also relevant in creating transfer orbits from the small primary to $L_3$ in the RPC$3$BP (see \cite{HTL07, TFRPGM10})  or between primaries in the Bicircular 4-Body Problem (see \cite{JorNic20, JorNic21}).

Over the past years, one of the main focus of the study of the  dynamics ``close''  to  $L_3$ and its invariant manifolds has been  the so called ``horseshoe-shaped orbits'', first analyzed in~\cite{Brown1911},
which are quasi-periodic orbits that encompass the critical points $L_4$, $L_3$ and $L_5$.
The interest on these types of orbits arises when modeling the motion of co-orbital satellites, the most famous being Saturn's satellites Janus and Epimetheus and near Earth asteroids.
Recently, in~\cite{NPR20}, the authors have proved the existence of $2$-dimensional elliptic invariant tori on which the trajectories mimic the motions followed by Janus and Epimetheus
(see also  \cite{CorsHall03, BFPC13, CPY19}, and \cite{DerMur81a, DerMur81b, LlibreOlle01, BarrabesMikkola05, BarrabesOlle2006} for numerical studies).

The mentioned results deal with stable coorbital motions. On the contrary, the purpose of this paper is to prove that unstable  motions and homoclinic orbits coexist with them.
In particular,
%relying on the invariant manifolds of $L_3$. In particular, we analyze
%, that is the Lagrange point located ``at the other side'' of the massive primary.
\begin{enumerate}
 \item We construct  homoclinic orbits to $L_3$ for a sequence of mass ratios $\mu$ tending to zero.
The papers~\cite{articleInner, articleOuter,BCGG23} prove that the 1-dimensional stable and unstable manifolds of $L_3$ do not meet the first time they intersect a given transverse section for $\mu$ small enough. However, we prove that, for the sequence values of $\mu$,  they do meet the second time they hit the section (see Theorem \ref{TheoremD} in Section \ref{subsection:introHomoclinicPhenomena} below).
%We prove that for this sequence, they do meet the second time they meet the sect  1-round homoclinic orbits (see Definition \ref{defi:multiround} below) do not exist for small values of $\mu$. In the present paper 
 \item We prove the existence of Smale horseshoes, that is of hyperbolic invariant sets with symbolic dynamics. This is a consequence of the existence of transverse homoclinic orbits to certain Lyapunov periodic orbits which lie on the center manifold of $L_3$. See Theorem \ref{TheoremB} in Section \ref{sec:chaoticcoorbital}.
 %. We analyze the transverse intersections between their invariant manifolds and also show that they may possess quadratic homoclinic tangencies.
%Such analysis, allows us to construct
%\begin{itemize}
% \item[1.] .
% \item[2.] Newhouse domains.
%\end{itemize}
\item We construct Newhouse domains for the RPC3BP by showing that there exist a Lyapunov periodic orbit around $L_3$ with a quadratic homoclinic tangency which unfolds generically with respect to the energy (see Theorem  \ref{TheoremC}). This leads to the existence of hyperbolic sets with Hausdorff dimension arbitrarily close to maximal and to the existence of an infinite number of elliptic islands (see Theorem \ref{thm:Newhousemain} in Section \ref{sec:introNewhouse}).
\end{enumerate}

A key point to obtain these results  is an asymptotic formula for  the distance between the $1$-dimensional stable and unstable invariant manifolds of the point $L_3$ (at a first crossing with a suitable transverse section) for $\mu>0$ small enough which was proven by the authors in \cite{articleInner, articleOuter,BCGG23} (see Theorem \ref{TheoremA} below).

Together with the already mentioned KAM results provided in \cite{NPR20}, the  results presented in this paper show mixed dynamics at coorbital motions. In other words, the coexistence of stable motions (KAM regime) and unstable motions. As far as the authors know this is one of the first papers to build Newhouse domains in Celestial Mechanics (see \cite{GorodetskiK12}). See Section~\ref{subsection:introStateArt}  for a brief discussion on  related  previous results.

\renewcommand\thenamedTheorem{\Alph{namedTheorem}}
\renewcommand\thenamedCorollary{\Alph{namedCorollary}}

\subsection{Homoclinic connections to \texorpdfstring{$L_3$}{L3}}
\label{subsection:introHomoclinicPhenomena}

The critical point $L_3$ and its eigenvalues  satisfy that, as $\mu \to 0$,
\begin{equation}\label{def:pointL3}
	(q_1,q_2,p_1,p_2) = (d_{\mu},0,0,d_{\mu}), 
	\qquad \text{with} \quad
	d_{\mu} = 1 + \frac{5}{12} \mu + \OO(\mu^3)
\end{equation}
and
\begin{equation}\label{eq:specRot}
	\mathrm{Spec}  = \claus{\pm \sqrt{\mu} \, \rho_{\mathrm{eig}}(\mu), \pm i  \, \omega_{\mathrm{eig}}(\mu) },
	\quad
	\text{with} \quad \left\{
	\begin{array}{l}
		\rho_{\mathrm{eig}}(\mu)=\sqrt{\frac{21}{8}}  + \OO(\mu),\\[0.4em]
		\omega_{\mathrm{eig}}(\mu)=1 + \frac{7}{8}\mu + \OO(\mu^2),
	\end{array}\right.
\end{equation}
(see \cite{Szebehely}). 
Therefore, $L_3$ possesses one-dimensional unstable and stable manifolds, which we denote as $W^{\unstable}(L_3)$ and $W^{\stable}(L_3)$.
Notice that, due to the different size in the eigenvalues, the system possesses two time scales which translates to rapidly rotating dynamics coupled with a slow hyperbolic behavior around the critical point $L_3$.

\begin{figure}
	\centering
	\begin{overpic}[scale=0.26]{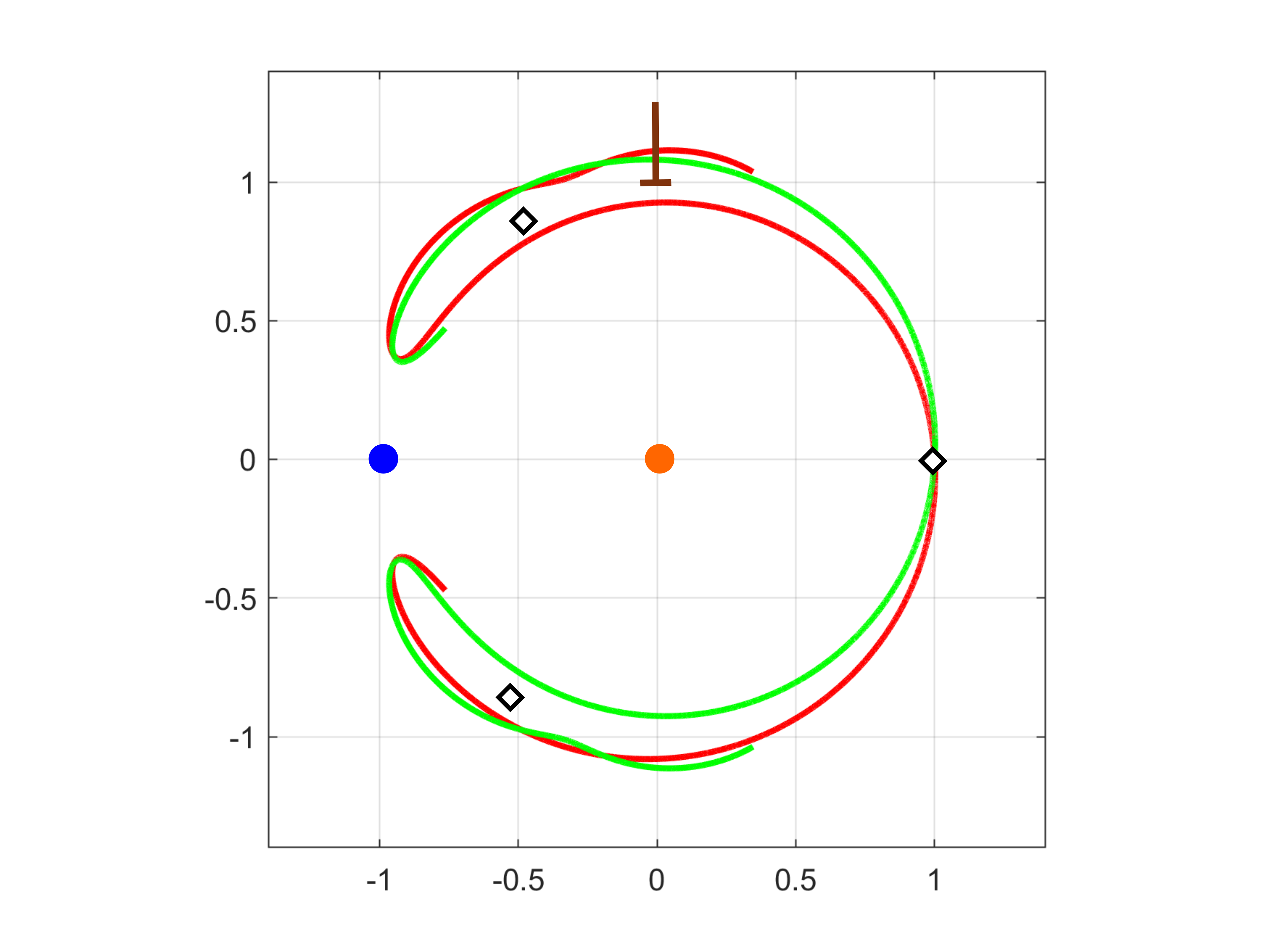}
		\put(51,34){\color{orange} $S$ }
		\put(28,34){\color{blue} $P$ }
		%\put(32,41){{\color{red} $L_1$ }}
		%\put(24,41){{\color{red} $L_2$ }}
		\put(75,38){{\color{black} $L_3$ }}
		\put(38,62){{\color{black} $L_5$ }}	
		\put(38,13){{\color{black} $L_4$ }}
		\put(53,64.5){{\color{brown} $\Sigma$ }}
		\put(65,60){\color{myGreen} $W^{\stable,+}(L_3)$}
		\put(50,50){\color{red} $W^{\unstable,+}(L_3)$}
		\put(50,25){\color{myGreen} $W^{\stable,-}(L_3)$}
		\put(65,13){\color{red} $W^{\unstable,-}(L_3)$}
	\end{overpic}
	\caption{Projection onto the $q$-plane of the unstable (red) and stable (green) manifolds of $L_3$, for ${\mu=0.0028}$. }
	\label{fig:perturbedInvariantManifolds1d}
\end{figure}
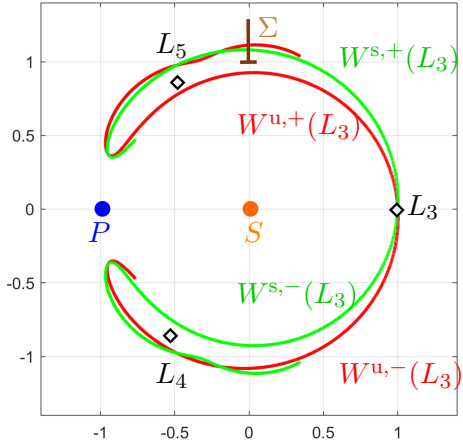

The manifolds $W^{\unstable}(L_3)$ and $W^{\stable}(L_3)$ have two branches each.
One pair, which we denote by $W^{\unstable,+}(L_3)$ and $W^{\stable,+}(L_3)$ circumvents $L_5$ whereas the other circumvents $L_4$ and it is denoted as  $W^{\unstable,-}(L_3)$ and $W^{\stable,-}(L_3)$, see Figure~\ref{fig:perturbedInvariantManifolds1d}.
Notice that the Hamiltonian system associated to  $\HInicial$ in \eqref{def:hamiltonianInitialNotSplit} is reversible with respect to the involution
\begin{equation}
	\label{def:involutionCartesians}
	\Psi(q,p)=(q_1,-q_2,-p_1,p_2).
\end{equation}
Therefore, by~\eqref{def:pointL3}, $L_3$ belongs to the symmetry axis given by $\Psi$ and the $+$ branches of the invariant manifolds of $L_3$ are symmetric to the $-$ ones.

In the papers \cite{articleInner, articleOuter, BCGG23}, the authors provide an asymptotic formula for the distance between the 1-dimensional stable and unstable manifolds of $L_3$ at a transverse section.  To present this  formula, we introduce the classical symplectic polar coordinates
\begin{align}\label{def:changePolars}
	q= 
	%	\begin{pmatrix}
	%		q_1 \\ 
	%		q_2
	%	\end{pmatrix} =
	r \begin{pmatrix}
		\cos \tht \\ 
		\sin \tht
	\end{pmatrix},
	\qquad
	p = 
	%	\begin{pmatrix}
	%		p_1 \\ 
	%		p_2
	%	\end{pmatrix} =
	R
	\begin{pmatrix}
		\cos \tht \\ 
		\sin \tht
	\end{pmatrix} 
	- \frac{G}{r} \begin{pmatrix}
		\sin \tht \\ 
		-\cos \tht
	\end{pmatrix},
\end{align} 
where 
% $r$ is the radius, $\tht$ the argument of $q$, 
$R$ is the radial linear momentum  and $G$ is the angular momentum.
We consider as well the $3$-dimensional section 
\begin{equation}\label{definition:sectionPolars}
\Sigma = \claus{(r,\tht,R,G) \in \reals \times \torus \times \reals^2 
	\st r>1, \, \tht=\frac{\pi}2 \,}
\end{equation}
and denote by $(r^{\unstable}_*,\frac{\pi}2, R^{\unstable}_*,G^{\unstable}_*)$ and $(r^{\stable}_*,\frac{\pi}2,R^{\stable}_*,G^{\stable}_*)$ the first crossing of the invariant manifolds with this section (see Figure \ref{fig:perturbedInvariantManifolds1d}). 
The next theorem measures the distance between these points for $0< \mu\ll 1$.
%
% Its proof is given in \cite{articleInner, articleOuter,BCGG23}.

\begin{namedTheorem}[Distance between the unstable and stable manifolds of $L_3$]
	\label{TheoremA}
	There exists $\mu_0>0$ such that, for $\mu \in (0,\mu_0)$,
	\[
	\norm{(r^{\unstable}_*,R^{\unstable}_*,G^{\unstable}_*)-(r^{\stable}_*,R^{\stable}_*,G^{\stable}_*)}
	=
	\sqrt[3]{4} \,
	\mu^{\frac13} e^{-\frac{A}{\sqrt{\mu}}} 
	\boxClaus{\vabs{\CInn}+\OO\paren{\frac1{\vabs{\log \mu}}}},
	\]
	where the constant $\CInn \in \complexs$ 
satisfies
		\begin{equation}\label{eq:nonvanishingStokes}
		 \Theta\neq 0
		\end{equation}
and the constant $A>0$ is given by the real-valued integral
		\begin{equation}\label{def:integralA}
			A= \int_0^{\frac{\sqrt{2}-1}{2}} \frac{2}{1-x}\sqrt\frac{x}{3(x+1)(1-4x-4x^2)}  dx\approx 0.177744.
		\end{equation}
% 		\item T		% 		by the result of the inner equation in Theorem 2.7 in \cite{articleInner} (see Section \ref{subsection:innerHeuristics} for an overview of the result).	
% 	\end{itemize}
\end{namedTheorem}

%\textcolor{blue}{
The asymptotic formula in the theorem is obtained in the papers \cite{articleInner, articleOuter}. Then, in \cite{BCGG23}, by means of a computer assisted proof, we show that the constant $\Theta$ is not zero. The distance between the  stable and unstable manifolds of $L_3$ is exponentially small with respect to $\sqrt{\mu}$. This is due to the rapidly rotating dynamics of the system (see~\eqref{eq:specRot}) and it is usually known as a \emph{beyond all orders phenomenon}, since the difference between the manifolds cannot be detected by expanding the manifolds in series of powers of $\mu$.
%
%Due to this phenomenon, the classical Melnikov Method (see for instance~\cite{GuckenheimerHolmes}) cannot be applied to obtain Theorem~\ref{TheoremA}.}
%Notice that, in Theorem~\ref{TheoremA}, due to the rapidly rotating dynamics of the system (see~\eqref{eq:specRot}), the distance between the  stable and unstable manifolds of $L_3$ is exponentially small with respect to $\sqrt{\mu}$.
%
Due to the symmetry in~\eqref{def:involutionCartesians}, an analogous result holds for the opposite branches.
%

% The critical point $L_3$ is a saddle-center with one-dimensional stable and unstable invariant manifolds. 
The goal of this section is to analyze the existence of homoclinic orbits to $L_3$. To this end, let us introduce the following definition.

\begin{definition}\label{defi:multiround}
Let $\G(t)$ be an  homoclinic orbit of \eqref{def:hamiltonianInitialNotSplit}  to the critical point $L_3$ and $B_{\mu}$ a ball centered at $L_3$ of radius $\mu$.
Then, we say that  $\G(t)$ is $k$-round if 
\[
\conj{\bigcup_{t\in \reals} \G(t)} \setminus B_{\mu}
\quad
\text{ has } k \text{ connected components}.
\]
\end{definition}
%Let us recall that, by Ansatz~\ref{ansatz}, one has that $\CInn \neq 0$. 
%
%Then, Theorem~\ref{mainTheoremDist} imply that they do not exist  homoclinic connections between the ``$+$''~branches of the stable and unstable manifolds of $\Ltres(\de)$ at its first intersection with the section $\claus{\la=\la_*, \La>0}$.
%
Theorem \ref{TheoremA} implies the following corollary.
% According to this definition, Theorem~\ref{TheoremA} and Ansatz~\ref{ansatz} imply the following. 

\begin{namedCorollary}[$1$-round homoclinic connections]
	\label{corollaryA}%
There exists $\mu_0>0$ such that, for $\mu \in (0,\mu_0)$, the Hamiltonian system associated to \eqref{def:hamiltonianInitialNotSplit} does not have $1$-round homoclinic connections to~$L_3$.
\end{namedCorollary}

%In Section \ref{sec:breakdown1round} below we provide an asymptotic formula for the distance between the stable and unstable manifolds of $L_3$ the first time they hit a given transverse section. The fact that this distance is nonzero imply Theorem \ref{corollaryA}.
% 
% Moreover, due to the symmetry of the system (see \eqref{def:involutionScaling}), an analogous result holds for $\WW^{\unstable,-}(\Ltres)$ 
% and
% $\WW^{\stable,-}(\Ltres)$.
% 
% Theorem~\ref{TheoremA} and Ansatz~\ref{ansatz} imply that the invariant manifolds of $L_3$ do not meet the first time they meet the section $\Si$.
%
%
%The first result of this paper is the existence of such homoclinic connections for certain values of the mass parameter $\mu$.

% To state it, we first classify the types of homoclinic orbits by how many ``rounds'' they take before returning to $L_3$. 
% %
% In particular, we say that an homoclinic connection to $L_3$ is \emph{$k$-round} if, on a  $\de$-neighborhood of this critical point, the closure of the homoclinic orbit has $k$ connected components, (see Figure~\ref{fig:reconnectionsPlot} for examples of $2$-round connections).
% %

\begin{figure}
\centering
\begin{overpic}[scale=0.45]{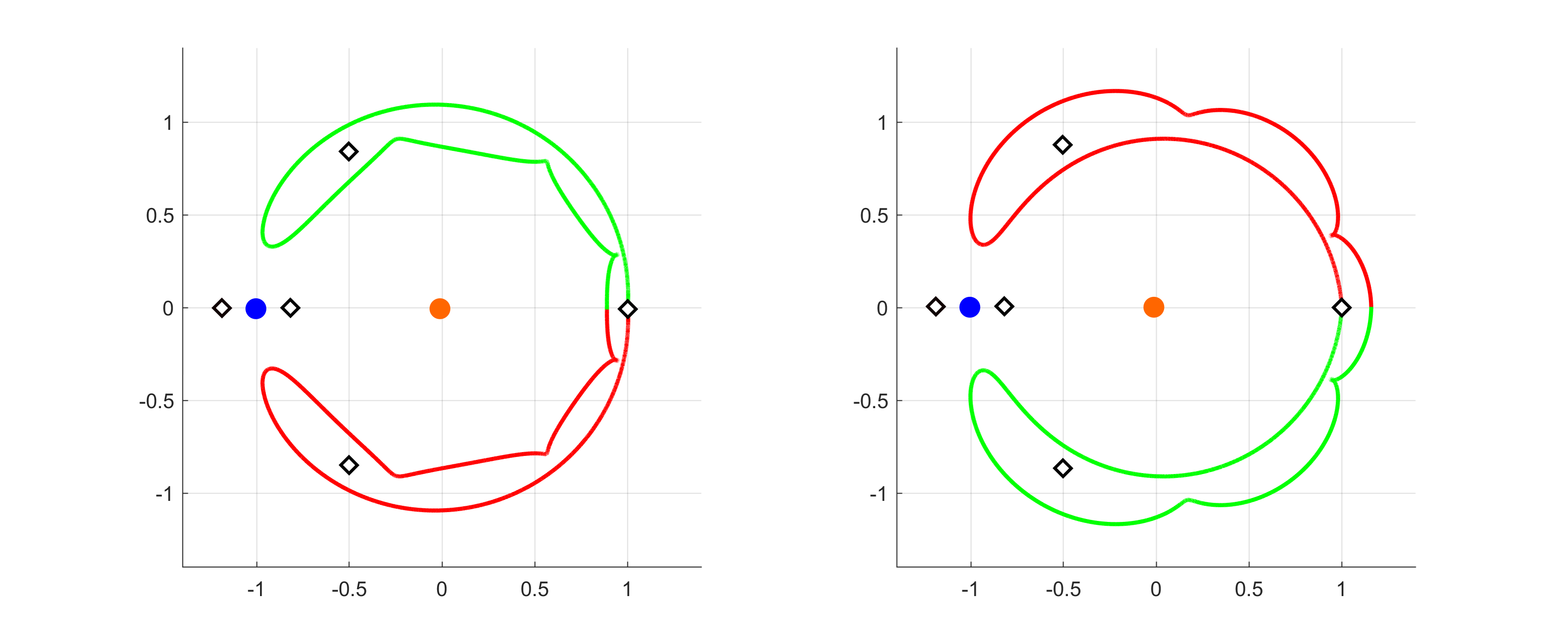}
\put(30,35){\color{myGreen} $W^{\stable,+}(L_3)$}
\put(75,35){\color{red} $W^{\unstable,+}(L_3)$}
\put(75,5.5){\color{myGreen} $W^{\stable,-}(L_3)$}
\put(30,5.5){\color{red} $W^{\unstable,-}(L_3)$}
	\end{overpic}
	\caption{Projection onto the $q$-plane for examples of $2$-round homoclinic connection to $L_3$. (Left) $\mu=0.012144$,  (right) $\mu=0.004192$. }
	\label{fig:reconnectionsPlot}
\end{figure}
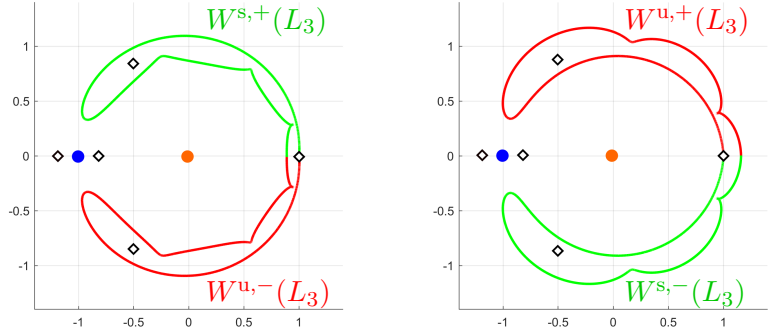
 
This corollary does not prevent the existence of multi-round homoclinic orbits. Indeed, E.~Barrab\'es, J.M.~Mondelo and M.~Oll\'e in \cite{BMO09} analyzed numerically the existence of multi-round homoclinic connections to $L_3$ in the RPC$3$BP and  conjectured the existence of $2$-round homoclinic orbits for a sequence of mass ratios $\claus{\mu_n}_{n \in \naturals}$ satisfying $\mu_n\to 0$ as $n \to \infty$.
%and supported their claim with numeric computations.
%
The first result of this paper proves  this conjecture.
 
\begin{namedTheorem}[$2$-round homoclinic connections]
\label{TheoremD}
%
% Assume Ansatz~\ref{ansatz} and c
%
There exists a sequence $\{{\mu}_n\}_{n \geq N_0}$ with $N_0$ big enough, of the form
\begin{align*}
\mu_n 
=
\frac{A}{n\pi\rho_{\mathrm{eig}}(0)}
\paren{1+ 
	\OO \paren{\frac1{\log n}}},
\qquad \text{for } n\gg 1,
\end{align*}
where  $\rho_{\mathrm{eig}}(\mu)$ is  given in \eqref{eq:specRot} and $A>0$ is the constant introduced in \eqref{def:integralA},
such that the Hamiltonian system~\eqref{def:hamiltonianInitialNotSplit}
has a $2$-round homoclinic connection to the equilibrium point $L_3$. These homoclinic orbits coincide with $W^{\unstable,+}(L_3)$ and $W^{\stable,-}(L_3)$.
\end{namedTheorem}

This theorem is proven in Section~\ref{section:proofReconnections}.
Using the same tools, one can obtain an analogous result for the homoclinic connections between $W^{\unstable,-}(L_3)$ and $W^{\stable,+}(L_3)$  (for a possibly different sequence of mass ratios).

%\renewcommand\thenamedRemark{\Alph{namedRemark}}
%\setcounter{namedRemark}{1}
%\begin{namedRemark}[Multi-round homoclinic connections]%
%In \cite{BMO09}, the authors also conjectured the existence of $k$-round homoclinic connections for $k > 2$ for different sequences of the mass parameter $\mu$.
%
%We believe that our strategy can be applied also for proving the existence of $k$-round homoclinic symmetric connections.
%\end{namedRemark}

\subsection{Coorbital chaotic motions}\label{sec:chaoticcoorbital}
% associated to \texorpdfstring{$L_3$}{L3}}

Next we study the existence of chaotic phenomena associated to $L_3$ and its invariant manifolds.
%
%In particular, we prove the existence of a Smale horseshoe close to the invariant manifolds of $L_3$ by means of the Smale-Birkhoff homoclinic theorem
%
%This classical result 
%(see \cite{Smale67, KatokHasselblatt}).
%states that, by proving the existence of transverse homoclinic orbits to periodic orbits (for flows) one can construct symbolic dynamics.
The Lyapunov Center Theorem (see for instance \cite{MeyerHallOffin}) ensures the existence of a family of periodic orbits emanating from the saddle-center $L_3$ which, close to the equilibrium point, are hyperbolic.
%
%Let us denote by $\Pi_3$ the Lyapunov family of hyperbolic periodic orbits of $L_3$. 
%
This family can be parametrized by the energy level given by the Hamiltonian $h$ in \eqref{def:hamiltonianInitialNotSplit}.

\begin{proposition}[Lyapunov periodic orbits to $L_3$]
\label{PropositionB}
	There exist $\mu_0, \varrho_0>0$ small enough such that, for $\mu \in (0,\mu_0)$, the Hamiltonian system with Hamiltonian~\eqref{def:hamiltonianInitialNotSplit} has a family of hyperbolic periodic orbits 
\begin{align*}
	\Pi_3 = \claus{P_{3,\varrho} \text{ periodic orbit} \st h(P_{3,\varrho})= \varrho^2 + h(L_3), \,
		{\varrho} \in (0,\varrho_0)
	},
\end{align*}
which depend regularly on $\varrho \in (0,\varrho_0)$ and satisfy that $\mathrm{dist}(P_{3,\varrho},L_3) \to 0$ as $\varrho \to 0$ in the sense of Hausdorff distance.
\end{proposition}

In Proposition~\ref{proposition:periodicOrbit} we state this result in a different set of coordinates and provide estimates for the periodic orbits. 
Its proof  can be found in Appendix~\ref{appendix:periodicOrbits}.

We denote by $W^{\unstable}(P_{3,\varrho})$ and $W^{\stable}(P_{3,\varrho})$ the $2$-dimensional unstable and stable invariant manifolds of the Lyapunov periodic orbit $P_{3,\varrho}$.
Analogously to the $L_3$ case, the invariant manifolds have two branches each. We denote by $W^{\unstable,+}(P_{3,\varrho})$ and $W^{\stable,+}(P_{3,\varrho})$ the ones that circumvent $L_5$ and, by $W^{\unstable,-}(P_{3,\varrho})$ and $W^{\stable,-}(P_{3,\varrho})$, the ones that surround $L_4$ (see Figure \ref{fig:perturbedInvariantManifolds2d}).
By the Smale-Birkhoff homoclinic Theorem (see \cite{Smale67, KatokHasselblatt}), proving the existence of transverse intersections between $W^{\unstable,+}(P_{3,\varrho})$ and $W^{\stable,+}(P_{3,\varrho})$ implies the existence of chaotic motions on a neighborhood of $L_3$ and its invariant manifolds.
More specifically, we prove the following result,  whose proof is deferred to Section~\ref{section:proofAB}.

\begin{figure}
	\centering
	\begin{overpic}[scale=0.26]{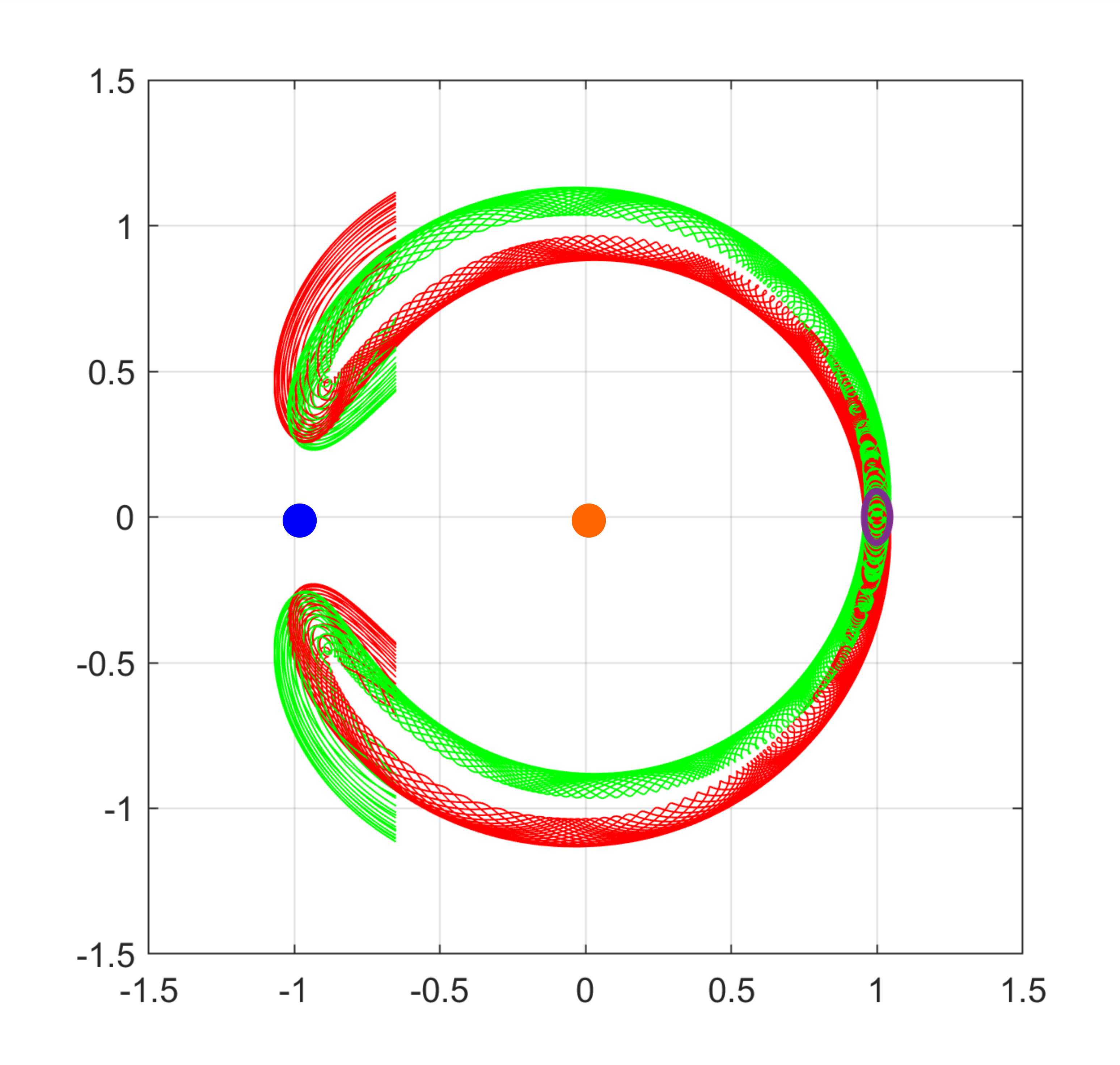}
		\put(55,45){\color{orange} $S$ }
		\put(28,45){\color{blue} $P$ }
		\put(80,46){{\color{purple} $P_{3,\varrho}$ }}
		\put(64,78){\color{myGreen} $W^{\stable,+}(P_{3,\varrho})$}
		\put(38,60){\color{red} $W^{\unstable,+}(P_{3,\varrho})$}
	\end{overpic}
	\caption{Projection onto the $q$-plane of the unstable (red) and stable (green) manifolds of the Lyapunov periodic orbit $P_{3,\varrho}$ (blue), for ${\mu=0.003}$. }
	\label{fig:perturbedInvariantManifolds2d}
\end{figure}
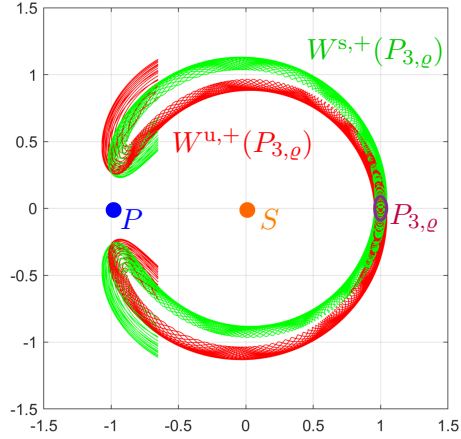

\begin{namedTheorem}[Chaotic motions]
\label{TheoremB}
Let $A>0$ and $\CInn\neq 0$ be the constants  given in Theorem~\ref{TheoremA} and $\varrho_0$ be as in Proposition~\ref{PropositionB}.
%
%Assume Ansatz~\ref{ansatz}.
%
Then, there exist $\mu_0>0$ and two functions
$\varrho_{\min}, \varrho_{\max}:(0,\mu_0) \to [0,\varrho_0]$ of the form
\begin{align*}
	\varrho_{\min}(\mu) &= \frac{\sqrt[6]{2}}2 \vabss{\CInn} \mu^{\frac13} e^{-\frac{A}{\sqrt{\mu}}}
	\boxClaus{1+\OO\paren{\frac1{\vabs{\log \mu}}}},
	\\
	\varrho_{\max}(\mu) &= \frac{\sqrt[6]{2}}2  \vabss{\CInn} \mu^{\frac13} e^{-\frac{A}{\sqrt{\mu}}}
	\boxClaus{2+\OO\paren{\frac1{\vabs{\log \mu}}}},
\end{align*}
such that, for $\mu \in (0,\mu_0)$ and $\varrho\in (\varrho_{\min}(\mu),\varrho_{\max}(\mu)]$, the following statement hold.

\begin{enumerate}
\item The invariant manifolds $W^{\unstable,+}(P_{3,\varrho})$ and $W^{\stable,+}(P_{3,\varrho})$ intersect transversally.

\item Consider the section $\wh{\Si}_{\varrho}=\Si \cap \claus{h=\varrho^2+h(L_3)}$ with $\Si$ as given in \eqref{definition:sectionPolars} and the induced Poincar\'e map $\PP:\wh{\Si}_{\varrho} \to \wh{\Si}_{\varrho}$.
Then, there exists $M>0$ such that $\PP^M$ has an invariant set $\XX$, homeomorphic to $\integers^{\naturals}$, such that $\PP^M|_{\XX}$ is topologically conjugated to the shift. 
\end{enumerate}
\end{namedTheorem}
%the existence of a hyperbolic set whose dynamics is conjugated to that of the Bernoulli shift (in particular, with positive topological entropy) exponentially close to the invariant manifolds of $L_3$

Due to the symmetry in~\eqref{def:involutionCartesians}, an analogous result holds for the transverse intersections of  $W^{\unstable,-}(P_{3,\varrho})$ and $W^{\stable,-}(P_{3,\varrho})$.

The chaotic motions induced by the Smale's horseshoe maps provided in the previous theorem lie in a tubular neighborhood around the invariant manifolds $W^{\unstable}(L_3)$ and $W^{\stable}(L_3)$ with the boundary at an energy level of the form $h=h(L_3)+\OO(\mu^{\frac23}e^{-\frac{2A}{\sqrt{\mu}}})$
%where one can construct a Smale's horseshoe map for a suitable Poincar\'e map induced by the flow of the Hamiltonian $h$ in \eqref{def:hamiltonianInitialNotSplit}.

%

To prove Theorem~\ref{TheoremB}, we rely on the asymptotic formula given in by Theorem~\ref{TheoremA}.
Since $W^{\unstable}(L_3)$ and $W^{\stable}(L_3)$ are exponentially close to each other with respect to $\sqrt{\mu}$, the energy levels where chaotic motions are found are also exponentially close to that of $L_3$.
In addition, by restricting $\mu$ one can take $\varrho_{\max}(\mu)$ bigger (see Theorem~\ref{mainTheoremA} below).

%	
%	Moreover, there exist constants $c_{\max}>c_{\min}>0$ independent of $\mu$ such that
%	\begin{align*}
%		\vabs{k_{\min}(\mu) - h(L_3)} 
%		\leq c_{\min} \,
%		\mu^{\frac23} e^{-\frac{2 A}{\sqrt{\mu}}},
%		\qquad
%		\vabs{k_{\max}(\mu) - h(L_3)} 
%		\leq c_{\max} \,
%		\mu^{\frac23} e^{-\frac{2 A}{\sqrt{\mu}}}.
%	\end{align*}

%In addition, due to the geometry of the invariant manifolds of $\Pi_3$, it can be seen that $W^{\unstable,+}(P_{3,\varrho})$ and $W^{\stable,+}(P_{3,\varrho})$ unfolds a quadratic tangency at $\varrho=\varrho_{\min}$.
%

Moreover, following the same ideas behind Theorem~\ref{TheoremB}, we prove that the  Lyapunov periodic orbit  at the energy level $h=\varrho^2_{\mathrm{min}}(\mu)+h(L_3)$ possesses a quadratic homoclinic tangency.
%${\varrho=\varrho_{\min}(\mu)}$.

\begin{namedTheorem}[Homoclinic tangencies]
\label{TheoremC}
%
% Assume Ansatz~\ref{ansatz} and d
Denote by $f_{\varrho}$ the flow of the Hamiltonian system given in~\eqref{def:hamiltonianInitialNotSplit} restricted to the energy level $h=\varrho^2+h(L_3)$.
Let $\varrho_0, \mu_0>0$ and $\varrho_{\min}(\mu):(0,\mu_0) \to [0,\varrho_0]$ be as given in Theorem~\ref{TheoremB}.
Then, for a fixed $\mu\in (0,\mu_0)$ and $\varrho$ close to $\varrho_{\min}(\mu)$, the flow $f_{\varrho}$ unfolds generically an homoclinic quadratic tangency between $W^{\unstable,+}(P_{3,\varrho_{\min}(\mu)})$ and $W^{\stable,+}(P_{3,\varrho_{\min}(\mu)})$.
\end{namedTheorem}

To prove this result, we use the definition of generic unfolding of a quadratic homoclinic tangency given in \cite{Dua08} (see also Section \ref{sec:introNewhouse} below). 
% 
% \renewcommand\theremarkMain{\Alph{remarkMain}}
% \setcounter{remarkMain}{3}
% \begin{remarkMain}\label{remark:NewhouseDomain}
% We use the definition of generic unfolding given in \cite{Dua08} for area preserving diffeomorphisms (see Theorem~\ref{mainTheoremB} for more details).
% %
% In particular, Theorem~\ref{TheoremC} should lead to prove the existence of a Newhouse domain.
% %	
% \end{remarkMain}
The existence of a quadratic homoclinic quadratic  tangency leads to the existence of Newhouse domains for the RPC3BP. This is explained in the next section. 
As far as the authors know this is one of the first constructions of Newhouse domains in Celestial Mechanics (see also \cite{GorodetskiK12}). 
\subsection{Newhouse domains for the RPC3BP at coorbital motions}\label{sec:introNewhouse}

To describe the dynamics arising from the quadratic homoclinic tangencies provided in Theorem \ref{TheoremC}, we have to introduce several concepts. We follow the approach in \cite{Dua08,Gorodetski12,BergerFP22}.

Consider a symplectic 2-dimensional manifold $M$. A hyperbolic basic set for a $\mathcal{C}^r$  diffeomorphism $f\in  \Diff^r(M)$, $r\geq 4$, is an invariant compact set $\Lambda$ which is transitive, hyperbolic and locally maximal (it is the maximal invariant set in one of its neighborhoods $U$). All the points in $\Lambda$ have  stable and unstable manifolds, which are injectively immersed submanifolds depending continuously on the base point. It is a well known fact that hyperbolic basic sets are robust under $\mathcal{C}^1$ perturbations and the dynamics of the perturbed set is equivalent to that of $\Lambda$. We call the perturbed hyperbolic basic set the hyperbolic continuation of $\Lambda$.

Given two points $x,y\in\Lambda$, an intersection point of $W_x^s\cap W_y^u$ is called a homoclinic tangency if the corresponding intersection is not transverse.

We say that $\Lambda$ is a wild basic set over an open set $\UU\subset \Diff^r(M)$ containing $f$ if, for all maps $g\in\UU$,
\begin{enumerate}
 \item The hyperbolic continuation $\Lambda_g$ is a hyperbolic basic set conjugated to $\Lambda$.
 \item There is at least one orbit of homoclinic tangencies of $\Lambda_g$.
\end{enumerate}
The set $\UU$ is usually called Newhouse region.

Assume that the symplectic diffeomorphism $f$ has a hyperbolic saddle $P$. Its homoclinic class $H(P,f)$ is the closure of the union of the transverse homoclinic points to $P$. It is well known that $H(P,f)$ is a transitive invariant set. Moreover, $H(P,f)$ is the smallest closed invariant set which contains all the basic sets of $f$ containing $P$.

Close to basic sets there will appear plenty of elliptic periodic points with a particular structure. Let us also describe them. 
Consider an elliptic periodic point $P$ of period $N$ of $f$. We say that $P$ is generic if the two eigenvalues $\lambda$, $\lambda^{-1}$ lie in the unit circle and are not resonant up to order 3, that is $|\lambda|= 1$, $\lambda^2\neq 1$, $\lambda^3\neq 1$, and the first coefficient of the Birkhoff Normal Form of $f^N$ at $P$ is not zero. By KAM Theory, around such points there exists an invariant set, with full Lebesgue density at $P$, which
is a union of invariant curves for $f^N$, whose dynamics is conjugated to an irrational rotation of the circle. This structure around $P$ is usually called ``elliptic isle''.

We want to analyze the Newhouse phenomenon and the existence of wild basic sets for one parameter families of symplectic diffeomorphisms, that is a $\mathcal{C}^r$-function $(\eta,x)\to f_\eta(x)$ defined on  $I\times M$ where $I\subset\mathbb{R}$ is an interval, such that $f_\eta\in \Diff^r(M)$  and it is symplectic. We say that the family $f_\eta$ unfolds generically an orbit of homoclinic quadratic tangencies at $(\eta_0,Q_0)\in I\times M$, associated to some hyperbolic periodic point $P$ if, denoting by $P_\eta$ its hyperbolic continuation for the map $f_\eta$,
\begin{enumerate}
 \item The stable and unstable manifolds of $P$ for $f_{\eta_0}$, $W^s(P,f_{\eta_0})$ and $W^u(P,f_{\eta_0})$, have a quadratic tangency at $Q_0$.
 \item If $\ell$ is any smooth curve transverse to $W^s(P,f_{\eta_0})$ and $W^u(P,f_{\eta_0})$ at $Q_0$, then the local intersections of $W^s(P_\eta,f_{\eta})$ and $W^u(P_\eta,f_{\eta})$ with $\ell$ cross each other with relative non-zero velocity at $(\eta_0,Q_0)$.
\end{enumerate}

In \cite{Newhouse70,Dua08,Gorodetski12} it is proven the following.
\begin{theorem}\label{thm:NewhouseMaps}
Fix $0<\nu\ll 1$. Let $f_\eta$ be a $\mathcal{C}^r$ one parameter family of symplectic maps in $\Diff^r(M)$, $r\geq 6$. Let $O$ be a hyperbolic periodic  orbit and $\Gamma$ an orbit of homoclinic quadratic tangencies of $f_0$ which unfolds generically at $\eta=0$. Denote by $O_\eta$ the hyperbolic continuation of $O$ and take any small neighborhood $U$ of $O\cup \Gamma$. Then, there is sequence of Newhouse intervals $\Delta_k$ converging to $\eta=0$. Namely, for each $\eta\in\Delta_k$, $f_\eta$ possesses a wild hyperbolic basic set $\Lambda_{k,\eta}$, which depends continuously on $\eta$ (with respect to the Hausdorff distance), such that $O_\eta\subset\Lambda_{k,\eta}\subset U$.

Moreover, for each $k\geq 1$,
\begin{itemize}
\item For every $\eta\in \Delta_k$, the Hausdorff dimension of $\Lambda_{k,\eta}$ satisfies
\[
 \mathrm{dim}_H\Lambda_{k,\eta}\geq 2-\nu.
\]

 \item Given any periodic point $P_\eta\in\Lambda_{k,\eta}$ (in particular, $O_\eta$), there is a dense subset $D_k\subset\Delta_k$ such that for every $\eta\in D_k$, the periodic point $P_\eta$ has an orbit of homoclinic tangencies.
 \item There is a residual subset $R_k\subset\Delta_k$ such that for every $\eta\in R_k$, 
 \begin{enumerate}
  \item The homoclinic class $H(O_\eta, f_\eta)$ is accumulated by  generic elliptic periodic points of $f_\eta$.
  \item The homoclinic class $H(O_\eta, f_\eta)$ contains hyperbolic sets of Hausdorff dimension arbitrarily close to 2. In particular $\mathrm{dim}_HH(P_\eta, f_\eta)=2$.
 \end{enumerate}
%  the basic set $\Lambda_k$ is contained in the closure of all generic elliptic periodic points of $f_\eta$.
\end{itemize}
\end{theorem}

We apply this result to the quadratic homoclinic tangencies of the invariant manifolds of the Lyapunov periodic orbit around $L_3$ for the RPC3BP  obtained in Theorem \ref{TheoremC}. Since the RPC$3$BP is autonomous, the energy is conserved. Then, it can be seen as a family of 3-dimensional\footnote{The energy level is not  a manifold at the energy value of $L_3$, however it defines  (locally) a manifold for energy levels close enough to that of the Lyapunov  orbit with a quadratic homoclinic tangency.} flows depending on two parameters: the mass ratio $\mu$ and the energy $h$. We denote these flows by $\Phi^t_{\mu,h}$. Doing an abuse of language, in the next theorem, we use the concepts defined above (basic set, homoclinic class,  generic elliptic orbit, etc) referred to flows instead of maps. 

\begin{namedTheorem}[Newhouse phenomenon for the RPC3BP]\label{thm:Newhousemain}
Fix $0<\nu\ll1$ and  $\mu\in (0,\mu_0)$. Let $P=P_{3,\varrho_{\min}(\mu)}$ be the Lyapunov periodic orbit of  $\Phi^t_{\mu,h(\mu)}$ with $h_0(\mu)=\varrho_{\min}^2(\mu)+h(L_3)$  obtained in Theorem \ref{TheoremC}. Let $\Gamma$ be the associated orbit of quadratic homoclinic tangencies obtained in the same theorem. Take any small neighborhood $U$ of $P\cup\Gamma$. Then, 
\begin{itemize}
\item There exist $h^*>h_0(\mu)$ and a sequence of Newhouse intervals $\Delta_k=\Delta_k(\mu)\subset (h_0(\mu), h_*)$ converging to $h_0(\mu)$. That is for each $h\in \Delta_k$, the flow $\Phi^t_{\mu,h}$  possesses a wild hyperbolic basic set $\Lambda_{k,h}$, which depends continuously on $h$ (with respect to the Hausdorff distance) such that $\Lambda_{k,h}\subset U$ and $\Lambda_{k,h}$ contains $P_h$, the hyperbolic continuation of $P$.
\item For every $h\in \Delta_k$, the Hausdorff dimension of $\Lambda_{k,h}$ satisfies
\[
 \mathrm{dim}_H\Lambda_{k,h}\geq 3-\nu.
\]
\item Given any periodic orbit $Q_h\in\Lambda_{k,h}$ (in particular, $P_h$), there is a dense subset $D_k\subset\Delta_k$ such that for every $h\in D_k$,  $Q_h$ has an orbit of homoclinic tangencies.
\item There is a residual subset $R_k\subset\Delta_k$ such that for every $h\in R_k$, 
 \begin{enumerate}
  \item The homoclinic class $H(P_h,\Phi^t_{\mu,h})$ is accumulated by  generic elliptic periodic orbits of $\Phi^t_{\mu,h}$.
  \item The homoclinic class $H(P_h,\Phi^t_{\mu,h})$ contains hyperbolic sets of Hausdorff dimension arbitrarily close to 3. In particular $\mathrm{dim}_HH(P_h, \Phi^t_{\mu,h})=3$.
 \end{enumerate}
\end{itemize}
% Then, there is sequence of Newhouse intervals $\Delta_k$ converging to $\eta=0$. Namely, for each $\eta\in\Delta_k$, $f_\eta$ possesses a wild hyperbolic basic set $\Lambda_{k,\eta}$, which depends continuously on $\eta$ (with respect to the Hausdorff distance) such that $O_\eta\subset\Lambda_{k,\eta}\subset U$.
\end{namedTheorem}

This theorem is a consequence of Theorems \ref{TheoremC} and \ref{thm:NewhouseMaps}. Note that one cannot define a global Poincar\'e map in the energy levels considered. However, the proofs in \cite{Newhouse70,Dua08, Gorodetski12} only rely on an induced map  close to the periodic orbit. To construct it, it is enough to consider a local transverse section to the periodic orbit and therefore their proofs also apply to our setting.
% , defined on a transverse section to it, which we denote by $\mathcal{S}$. We denote this family of two dimensional  maps $F_{\mu,h}$. Note that for $h=h_0(\mu)$ (see Theorem \ref{TheoremC}) in the domain of definition of this map there must exist a quadratic homoclinic tangency to the Lyapunov periodic orbit (which for the map $F_{\mu,h}$ is a hyperbolic fixed point).

\subsection{State of the art}
\label{subsection:introStateArt}

A fundamental problem in dynamical systems is to prove whether a given system has chaotic dynamics.
For many physically relevant models this is usually  remarkably difficult. 
This is the case of many Celestial Mechanics models, where most of the known chaotic motions have been found in nearly integrable regimes where there is an unperturbed problem which already presents some form of ``hyperbolicity''. 
This is the case  in the vicinity of collision orbits (see for example \cite{Moe89, BolMac06, Bol06, Moe07}) or close to parabolic orbits (which allows to construct chaotic/oscillatory motions), see~\cite{Sitnikov1960, Alekseev1976, LlibSim80, Moser2001, GMS16, GMSS17, GPSV21, GuardiaMPS22}. 
There are also several results in regimes far from integrable which rely on computer assisted proofs \cite{Ari02, WilcZgli03, Cap12, GierZgli19}.
%(see also the results for the $N$-center problem in \cite{SoaveT12,BarutelloCT21}). 
%
%\textcolor{red}{Afegir referencies Susann Terraccini i companyia.}

The problem tackled in this paper is radically different. 
Indeed, if one takes the  limit $\mu\to 0$ in \eqref{def:hamiltonianInitialNotSplit} one obtains the classical integrable Kepler problem in the elliptic regime, where no hyperbolicity is present. 
Instead, the (weak) hyperbolicity is created by the $\mathcal{O}(\mu)$ perturbation.
%
% , as it was seen in \cite{articleInner}, for $\mu=0$, the unstable and stable manifolds of $L_3$ collapse into a circle of equilibrium points.
%
The bifurcation scenario we are dealing with is the so called $0^2i\omega$ resonance or Hamiltonian Hopf-Zero bifurcation. 
Indeed, for $\mu>0$ the Hamiltonian system given by $h$ in
\eqref{def:hamiltonianInitialNotSplit} has a saddle-center equilibrium point at $L_3$.
However, for $\mu=0$, the equilibrium point degenerates and the spectrum of its linear part consists in a pair of purely imaginary  and a double $0$ eigenvalues, (see \eqref{eq:specRot}).
%
%Then, if $V_{h,\mu}$ is the vector field associated to $h$, we say that the family admits a resonance $0^2i\omega$ at the origin or has a Hamiltonian–Hopf bifurcation.

Most of the studies in homoclinc phenomena around a saddle-center equilibrium are focused on the non-degenerate case where all the eigenvalues have comparable size, see \cite{Ler91, MHO92, Rag97a, Rag97b, BRS03}.
However, for close to resonance $0^2i\omega$ cases, to the best of authors knowledge, the results are more rare.
%
%In \cite{GaiGel10}, the authors study this singularity combining numerical and analytic techniques.
%
The generic unfolding of the  reversible  $0^2i\omega$ resonance is considered in \cite{Lom99, Lom00} where the author proves the existence of transverse homoclinic connections for every periodic orbit exponentially close to the origin and the breakdown of homoclinic orbits to the origin itself.
In \cite{JBL16}, the authors show the existence of homoclinic connections with several loops for every periodic orbit close to the equilibrium point for a generic unfolding of a Hamiltonian  $0^2i\omega$ resonance. Note that the unfolding of the  $0^2i\omega$ resonance in the RPC3BP is highly non-generic due to the strong degeneracies of its Keplerian approximation.

The work here presented shows the existence of homoclinic connections for both the equilibrium point and periodic orbits (exponentially) close to the equilibrium point. 
In the case of the (non-Hamiltonian) Hopf-zero singularity, we remark the strongly related work~\cite{BIS20}.
Also, in~\cite{GGSZ21}, the authors use similar techniques to analyze breather solutions for the nonlinear Klein-Gordon partial differential equation.

\section*{Acknowledgements}

I Baldom\'a has been supported by the grant PID-2021-
122954NB-100 funded by the Spanish State Research Agency through the programs
MCIN/AEI/10.13039/501100011033 and “ERDF A way of making Europe”.

M. Giralt has been supported by the European Union’s Horizon 2020 research and innovation programme under the Marie Sk\l odowska-Curie grant agreement No 101034255.
M. Giralt has also been supported by the research project PRIN 2020XB3EFL ``Hamiltonian and dispersive PDEs".

M. Guardia has been supported by the European Research Council (ERC) under the European Union's Horizon 2020 research and innovation programme (grant agreement No. 757802). 
M. Guardia is also supported by the Catalan Institution for Research and Advanced Studies via an ICREA Academia Prize 2019. 

This work is supported by the Spanish State Research Agency, through the Severo Ochoa and Mar\'ia de Maeztu Program for Centers and Units of Excellence in R\&D (CEX2020-001084-M).

\section{Scaled Poincar\'e variables and previous results}
\label{section:reformulation}

Let us notice that, for the unperturbed problem $h$ in \eqref{def:hamiltonianInitialNotSplit} with $\mu=0$, the five Lagrange point disappear into a circle of degenerate critical points.
For this reason, in~\cite{articleInner}, we introduced a singular change of coordinates to obtain a new first order Hamiltonian which has a  saddle-center equilibrium point (close to $L_3$) with stable and unstable manifolds that coincide along a separatrix.

First, in Section~\ref{subsection:scalingCoordinates}, we introduce the main features of this change of coordinates and its relation to $L_3$.
Then, in Section~\ref{subsection:reconnexionsPoincare}, we state Theorem \ref{mainTheoremDist}, which is a reformulation of Theorem \ref{TheoremA}
% and \ref{TheoremD} 
in the new set of coordinates. 
%
% Finally, in Section~\ref{subsection:interseccionsPoincare}, we introduce the Lyapunov family of periodic orbits surrounding $L_3$ in this set of variables and state Theorem~\ref{mainTheoremA}, which is a reformulation of both Theorem~\ref{TheoremB} and \ref{TheoremC}.

\subsection{A singular perturbation formulation of the problem}
\label{subsection:scalingCoordinates}

%The Lagrange point $L_3$ is a centre-saddle equilibrium point of the Hamiltonian $h$ whose eigenvalues satisfy
%%
%\begin{equation}
%\label{eq:specRot}
%	\mathrm{Spec}  = \claus{\pm \sqrt{\mu} \, \rho_{\mathrm{eig}}(\mu), \pm i  \, \omega_{\mathrm{eig}}(\mu) },
%	\quad
%	\text{with} \quad \left\{
%	\begin{array}{l}
%		\rho_{\mathrm{eig}}(\mu)=\sqrt{\frac{21}{8}}  + \OO(\mu),\\[0.4em]
%		\omega_{\mathrm{eig}}(\mu)=1 + \frac{7}{8}\mu + \OO(\mu^2),
%	\end{array}\right.
%\end{equation}
%as $\mu \to 0$.
%%
%The center and saddle eigenvalues are found at different time-scales. Moreover, when $\mu=0$, the unstable and stable manifolds of $L_3$ ``collapse'' to a circle of critical points.
%
%The singular character of the problem obstructs the application of methods of splitting of separatrices (see \cite{BFGS12} for example).
%, since it does not exist an homoclinic connection for the unperturbed case.
%
Applying a suitable singular change of coordinates, the Hamiltonian $h$ can be written as a perturbation of a pendulum-like Hamiltonian weakly coupled with a fast oscillator.
We summarize the most important properties of this set of coordinates, which was studied in detail in \cite{articleInner}.
% and we rewrite  Theorem \ref{theorem:mainTheorem}.

%When studying close to integrable system around resonances it is usual to ``blow-up'' the singularity to capture the slow-fast time scales.
%%
%Indeed, applying an adequate singular change of coordinates, the Hamiltonian $h$ in \eqref{eq:RShamiltonianInitial} becomes a pendulum-like Hamiltonian weakly coupled with a fast oscillator.

%The first step is to consider the Poincar\'e change of coordinates,
%$
%\phi_{\Poi}: (\la,L,\eta,\xi) \to (q,p).
%$
%
The Hamiltonian $h$ expressed in the classical (rotating) Poincar\'e coordinates,
$
\phi^{\Poi}: (\la,L,\eta,\xi) \to (q,p),
$ 
defines a Hamiltonian system with respect to the symplectic form $d\la \wedge dL + i \h d\eta \wedge d\xi$ and the Hamiltonian
\begin{equation} \label{def:hamiltonianPoincare}
H^{\Poi} = H_0^{\Poi}	+ \mu H_1^{\Poi},
\end{equation}
with
\begin{equation}\label{def:hamiltonianPoincare01}
\begin{split}
H_0^{\Poi}(L,\eta,\xi) &= -\frac{1}{2L^2} - L +  \eta \xi 
\qquad \text{and} \qquad
H_1^{\Poi} = h_1 \circ \phi_{\Poi}.
\end{split}
\end{equation}
Moreover, the critical point $L_3$ satisfies
\begin{equation}\label{eq:L3Poincare}
\la=0, \qquad \quad(L,\eta,\xi) = (1,0,0) + \OO(\mu)
\end{equation}
and the linearization of the vector field at this point has, at first order, an uncoupled nilpotent and center blocks.
%\begin{equation*}
%\begin{pmatrix}
%0 & -3 & 0 & 0 \\
%0 & 0 & 0 & 0 \\
%0 & 0 & i & 0 \\
%0 & 0 & 0 & -i 
%\end{pmatrix}
%+ \OO(\mu).
%\end{equation*}
Since $\phi^{\Poi}$ is an implicit change of coordinates, there is no explicit expression for $H_1^{\Poi}$.
However, since $H_1^{\Poi}$ is analytic for $\vabs{(L-1,\eta,\xi)}\ll 1$, it is possible to obtain series expansion in powers of $(L-1,\eta,\xi)$ (see \cite[Lemma 4.1]{articleInner})

In addition, since the original Hamiltonian $h$ is reversible with respect to the involution ${\Psi}$ in~\eqref{def:involutionCartesians}, the Hamiltonian $H^{\Poi}$ is reversible with respect to the involution 
\begin{equation}\label{def:involutionPoincare}
	{\Phi}_{\Poi}(\la,L,\eta,\xi) = (-\la,L,\xi,\eta).
\end{equation}

%\begin{remark}
%	In particular, one can see that the critical point $L_3$ satisfies
%	\begin{align*}
%		L &= \sqrt{\frac{d_{\mu}}{2-d^3_{\mu}}}
%		= 1 + \frac56 \mu + \OO(\mu^2),
%		\qquad
%		\eta = \xi \, = 
%		\sqrt{\sqrt{\frac{d_{\mu}}{2-d^3_{\mu}}}
%			- d_{\mu}^2}
%		= \frac{5\sqrt2}8 \mu + \OO(\mu^2),
%	\end{align*}
%	where $d_{\mu}$ satisfies that $d_0=1$ and is a solution of equation
%	\[
%	d_{\mu}^5 + (2+\mu)d_{\mu}^4 + (1 + 2\mu)d_{\mu}^3 
%	- (1 - \mu)d_{\mu}^2 - 2(1 - \mu)d_{\mu} - (1 -\mu)=0.
%	\]
%\end{remark}
To capture the slow-fast dynamics of the system,
we perform the singular symplectic scaling 
\begin{equation}\label{def:changeScaling}
	\phi_{\de}:
	(\la, \La, x,  y) 
	\mapsto 
	(\la,L,\eta,\xi),
	\qquad
	L = 1 + \de^2 \La , \quad
	\eta = \de x , \quad
	\xi = \de y , \quad
\end{equation}
and the time {reparametrization ${t} = \de^{-2} t'$},
where
\begin{equation}\label{def:delta}
\de= \mu^{\frac{1}{4}}.
\end{equation}
Defining the potential
\begin{equation}\label{def:potentialV}
\begin{split}
V(\la) 
&=  H_1^{\Poi}(\la,1,0,0;0)
=1 - \cos \la - \frac{1}{\sqrt{2+2\cos \la}},
\end{split}
\end{equation}
the Hamiltonian system associated to $H^{\Poi}$, expressed in scaled coordinates,
defines a Hamiltonian system with respect to the symplectic form $d\la \wedge d\La + i dx \wedge dy$ 
and the Hamiltonian
\begin{equation} \label{def:hamiltonianScaling}
	\begin{split}
		{H} = {H}_{\pend} + {H}_{\osc} + H_1,
	\end{split}
\end{equation}
where
\begin{align} 
	{H}_{\pend}(\la,\La) &= -\frac{3}{2} \La^2  + V(\la), \qquad
	{H}_{\osc}(x,y; \de) = \frac{x y}{\de^2},
	\label{def:HpendHosc} \\
	H_1(\la,\La,x,y;\de) &=
	H_1^{\Poi}(\la,1+\de^2\La,\de x,\de y;\de^4)  - V(\la) +
	\frac{1}{\de^4} F_{\pend}(\de^2\La),
	\label{def:hamiltonianScalingH1}
\end{align}
and
\begin{equation}\label{def:Fpend}
	F_{\pend}(z) = \paren{-\frac{1}{2(1+z)^2}-(1+z)}+\frac{3}{2} + \frac{3}{2}z^2 = \OO(z^3).
\end{equation}

We introduce a suitable neighborhood where the coordinates $(\la,\La,x,y)$ are defined. 
For $\cttTheoScalingDomainA>0$ we define the domain
\begin{align}\label{def:dominiUReals}
\UReals(\cttTheoScalingDomainA,\cttTheoScalingDomainB) = 
\claus{(\la,\La,x,\conj{x}) \in \reals/2\pi\integers \times \reals \times \complexs^2 : 
\vabs{\pi-\la}>\cttTheoScalingDomainA, 
	\vabs{(\La,x)}< \cttTheoScalingDomainB}.
\end{align} 
%
%Notice that, the domain is still $4$-dimensional.
%
For technical reasons, we consider some of the objects of the system in an analytical extension of the domain $\UReals$. 
In particular we use the domain
\begin{align}\label{def:dominiUComplexs}
	\UComplexs(\cttTheoScalingDomainA,\cttTheoScalingDomainB) = 
	\claus{(\la,\La,x,y) \in \complexs/2\pi\integers \times\complexs^3 : 
		\vabs{\pi-\Re \la}>\cttTheoScalingDomainA,
		\vabs{(\Im \la,\La, x, y)} <\cttTheoScalingDomainB}.
\end{align} 
The next proposition states some properties of the Hamiltonian $H$.
%and is a direct consequence of the results in \cite{articleInner} and \cite{articleOuter}.

\begin{proposition}\label{proposition:HamiltonianScaling}
Fix $\cttTheoScalingDomainA,\cttTheoScalingDomainB>0$.	
Then, there exists  $\de_0=\de_0(\cttTheoScalingDomainA,\cttTheoScalingDomainB)>0$ such that, for $\de \in (0,\de_0)$, one has that
\begin{itemize}
	\item The Hamiltonian $H$  in \eqref{def:hamiltonianScaling} is real-analytic in the sense of $\conj{H(\la,\La,x,y;\de)}= H(\conj{\la},\conj{\La},y,x;\de)$ in the domain $\UComplexs(\cttTheoScalingDomainA,\cttTheoScalingDomainB)$.
	\item There exists $\cttTheoScaling>0$ independent of $\de$ such that, for $(\la,\La,x,y) \in \UComplexs(\cttTheoScalingDomainA,\cttTheoScalingDomainB)$, 
	the second derivatives of the Hamiltonian $H_1$ given in \eqref{def:hamiltonianScalingH1} satisfy
	\begin{align*}
		\vabs{ \partial^2_{\la} H_1},
		\vabs{ \partial_{\la x} H_1},
		\vabs{ \partial_{\la y} H_1} &\leq \cttTheoScaling\de,
		&
		\vabs{\partial_{\la\La} H_1},
		\vabs{ \partial^2_{\La} H_1} &\leq
		\cttTheoScaling\de^2,
		\\
		\vabs{ \partial^2_{x} H_1},
		\vabs{ \partial_{x y} H_1},
		\vabs{\partial^2_{y} H_1}
		&\leq \cttTheoScaling\de^2,
		&
		\vabs{\partial_{\La x} H_1}
		\vabs{ \partial_{\La y} H_1}
		&\leq \cttTheoScaling\de^3.
	\end{align*}
	Moreover\footnote{One can obtain more precise estimates for the third derivatives of $H_1$. However, these rough estimates are sufficient for the proofs of this paper.}, 
	\begin{align*}
		\vabs{\partial_{\al_1, \al_2, \al_3} H_1} \leq \cttTheoScaling\de,
		\qquad 	\text{with} \qquad
		\al_1, \al_2, \al_3 \in \claus{\la,\La,x,y}. 
	\end{align*}
\end{itemize}
\end{proposition}

\begin{proof}
The first statement follows from \cite[Theorem 2.1]{articleInner}.
The second statement is a consequence of \cite[Lemma A.3]{articleOuter}.
\end{proof}

%The Hamiltonian $H$ (see \eqref{def:hamiltonianScaling}),
%away from collision with the primaries,
%is real-analytic  in the sense of $\conj{H(\la,\La,x,y;\de)}= H(\conj{\la},\conj{\La},y,x;\conj{\de}).$ 
%Moreover, for $\de > 0$ small enough, the critical point $L_3$ expressed in  coordinates $(\la,\La,x,y)$ satisfies
%\begin{equation*}
%	\la=0, 
%	\qquad
%	(\La,x,y)= (0,0,0) + \OO(\de),
%\end{equation*}
%and its linearization is
%\begin{equation*}
%	\begin{pmatrix}
%		0 & -3 & 0 & 0 \\
%		-\frac{7}{8} & 0 & 0 & 0 \\
%		0 & 0 & \frac{i}{\de^2} & 0 \\
%		0 & 0 & 0 & -\frac{i}{\de^2}
%	\end{pmatrix} + \OO(\de).
%\end{equation*}
%\end{proposition}	
%
%\begin{lemma}\label{lemma:expressionHamiltonianInfty}
%	Fix constants $\varrho_0>0$ and $\la_0 \in (0,\pi)$.
%	%
%	Then, there exists $\de_0>0$ such that, for $\de \in (0,\de_0)$,
%	$\vabs{\la}<\la_0$ and
%	$\vabs{(\La,x,y)}<\varrho_0$,
%	the second derivatives of the Hamiltonian $H_1(\la,\La,x,y;\de)$ (see \eqref{def:hamiltonianScalingH1}) satisfy
%	\begin{align*}
%		\vabs{ \partial^2_{\la} H_1}
%		&\leq C\de, &
%		\vabs{\partial_{\la\La} H_1}
%		&\leq C\de^2, &
%		\vabs{ \partial_{\la x} H_1}
%		&\leq C\de, &
%		\vabs{ \partial_{\la y} H_1}
%		&\leq C\de, 
%		\\
%		\vabs{ \partial^2_{\La} H_1}
%		&\leq C\textcolor{red}{\de^2}, &
%		\vabs{\partial_{\La x} H_1}
%		&\leq C\de^3, &
%		\vabs{ \partial_{\La y} H_1}
%		&\leq C\de^3, & &
%		\\
%		\vabs{ \partial^2_{x} H_1}
%		&\leq C\de^2, &
%		\vabs{ \partial_{x y} H_1}
%		&\leq C\de^2, &
%		\vabs{\partial^2_{y} H_1}
%		&\leq C\de^2. & &
%	\end{align*}
%\end{lemma}

\begin{remark}\label{remark:realanalytic}
Consider  $M \subseteq \complexs^4$ a symmetric subset with respect to $\reals^4$.
We say that a function $\zeta=(\zeta_{\la},\zeta_{\La},\zeta_x,\zeta_y): M \to \UComplexs(\cttTheoScalingDomainA,\cttTheoScalingDomainB)$ is real-analytic if, for $m \in M$,
$\zeta_{\la}\paren{\conj{m}} = \conj{\zeta_{\la}(m)}$,
$\zeta_{\La}\paren{\conj{m}} = \conj{\zeta_{\La}(m)}$,
$\zeta_x\paren{\conj{m}} = {\zeta_y(m)}$
and
$\zeta_y\paren{\conj{m}} = {\zeta_x(m)}$.
%\begin{align*}
%	\zeta_{\la}\paren{\conj{m}} = \conj{\zeta_{\la}(m)},
%	\qquad
%	\zeta_{\La}\paren{\conj{m}} = \conj{\zeta_{\La}(m)},
%	\qquad
%	\zeta_x\paren{\conj{m}} = {\zeta_y(m)},
%	\qquad
%	\zeta_y\paren{\conj{m}} = {\zeta_x(m)}.
%\end{align*}
%
Notice that, as a consequence,
$\zeta(m) \in \UReals(\cttTheoScalingDomainA)$, for $m \in M \cap \reals^4$.
%\begin{align*}
%	\zeta(m) \in \UReals(\cttTheoScalingDomainA), 
%	\quad \text{for} \quad m \in M \cap \reals^k.
%\end{align*}
\end{remark}

%Then, the critical point $L_3$ (see \eqref{eq:L3Poincare}) expressed in coordinates $(\la,\La,x,y)$ satisfies
%\begin{equation*}
%	\la=0, 
%	\qquad
%	(\La,x,y)= (0,0,0) + \OO(\de).
%\end{equation*}

Notice that, by \eqref{def:involutionPoincare}, the Hamiltonian $H$ is reversible with respect to the involution 
\begin{equation}\label{def:involutionScaling}
	{\Phi}(\la,\La,x,y) = (-\la,\La,y,x),
\end{equation}
which has  symmetry axis
\begin{equation}\label{def:symmetryAxisScaling}
	\SSS = \claus{\la=0, \, x=y}.
\end{equation}
In the next proposition, proven in \cite[Theorem 2.1]{articleInner}, we obtain an expression and suitable estimates for the equilibrium point $L_3$.

\begin{proposition}\label{proposition:existenceFixedPoint}
There exist $\de_0> 0$ and $\cttTheoLtres>0$ such that, for $\de \in (0,\de_0)$, the critical point $L_3$ expressed in  coordinates $(\la,\La,x,y)$ is of the form
\begin{equation}\label{def:pointL3sca}
	\Ltres(\de)=\paren{0,
		\de^2 \LtresLa(\de),
		\de^3 \Ltresx(\de),
		\de^3 \Ltresy(\de)}^T \in \SSS,
\end{equation}
with
$
	\vabs{\LtresLa(\de)}, 
	\vabs{\Ltresx(\de)}, \vabs{\Ltresy(\de)} \leq \cttTheoLtres$
and $\SSS$ as given in~\eqref{def:symmetryAxisScaling}.	
\end{proposition}

The linearization of $\Ltres(\de)$ is given by
\begin{equation*}
	\begin{pmatrix}
		0 & -3 & 0 & 0 \\
		-\frac{7}{8} & 0 & 0 & 0 \\
		0 & 0 & \frac{i}{\de^2} & 0 \\
		0 & 0 & 0 & -\frac{i}{\de^2}
	\end{pmatrix} + \OO(\de).
\end{equation*}
This analysis leads us to define a ``new'' first order for the Hamiltonian $H$ in \eqref{def:hamiltonianScaling} as
\begin{equation}\label{def:hamiltonianScalingH0}
	H_0(\la,\La,x,y;\de) = H_{\pend}(\la,\La) + H_{\osc}(x,y;\de),
\end{equation}
and we refer to $H_0$ as the unperturbed Hamiltonian and to $H_1$ (see \eqref{def:hamiltonianScalingH1}) as the perturbation.

Notice that the unperturbed Hamiltonian is uncoupled.
In the $(x,y)$-plane, it displays a fast oscillator of velocity $\frac1{\de^2}$
whereas, in the $(\la,\La)$-plane, it has a saddle at $(0,0)$ with two homoclinic connections or separatrices at the energy level $H_{\pend}(\la,\La)=-\frac12$, (see Figure~\ref{fig:separatrix}).
We define 
\begin{equation}\label{def:la0}
	\la_0 = \arccos\paren{\frac12-\sqrt2},
\end{equation}
which satisfies $H_{\pend}(\la_0,0)=H_{\pend}(0,0) =-\frac12$ and corresponds with the crossing point of the right separatrix with the axis $\claus{\La=0}$.
%
% (See Section~\ref{subsection:unperturbedSeparatrix} for more details).

\begin{figure}
	\centering
	\begin{overpic}[scale=0.6]{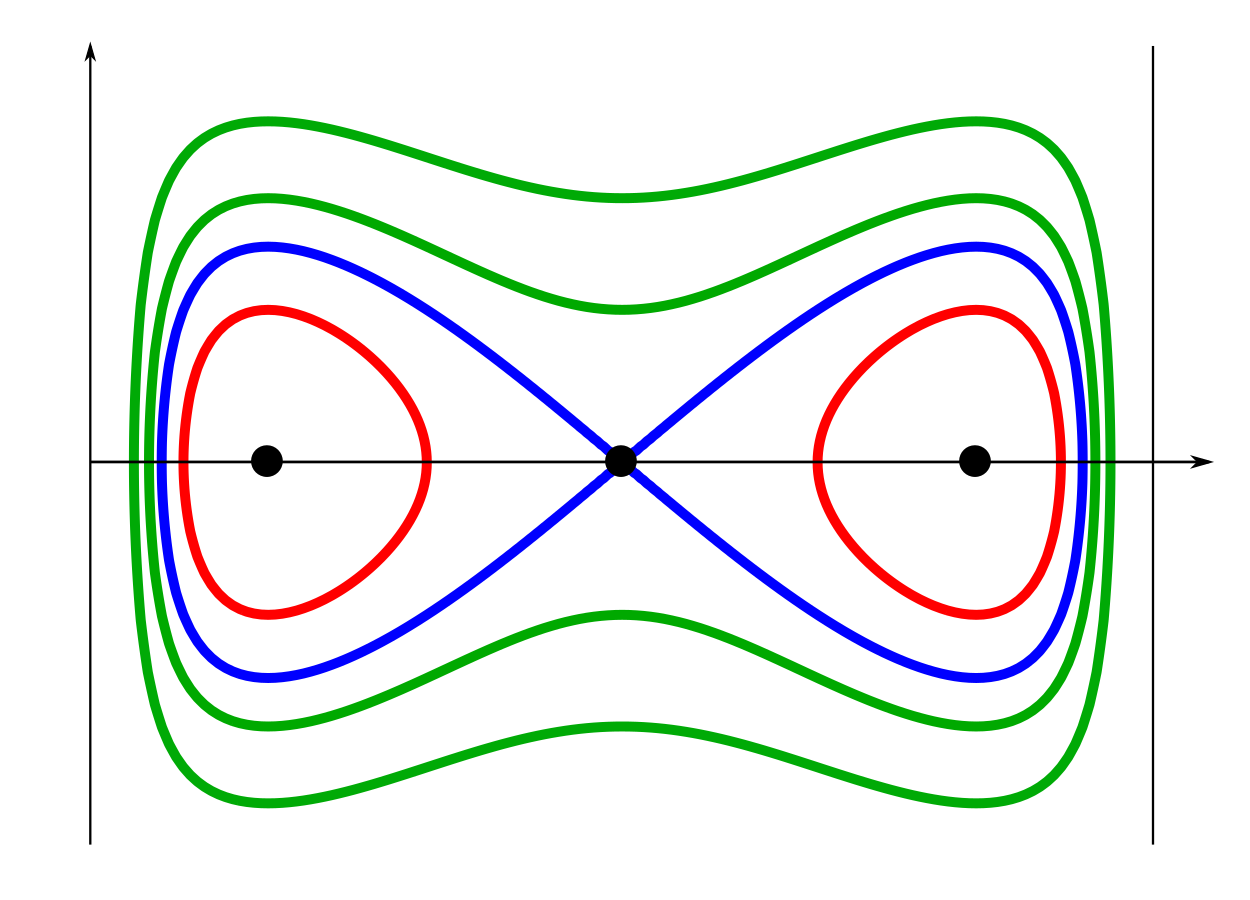}
		\put(93,31){{$\pi$}}
		\put(77,28.5){{$\frac{2}{3}\pi$}}
		\put(49,29){{$0$}}
		\put(17,28.5){{$-\frac{2}{3}\pi$}}
		\put(-4,31){{ $-\pi$}}
		\put(98,34){{$\la$}}
		\put(5,70){{$\La$}}
	\end{overpic}
	\caption{Phase portrait of the system given by Hamiltonian ${H}_{\pend}(\la,\La)$ on~\eqref{def:HpendHosc}. On blue the two separatrices.}
	\label{fig:separatrix}
\end{figure}
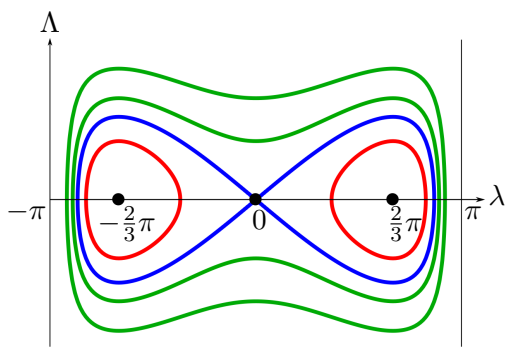

\subsection{The invariant manifolds of \texorpdfstring{$L_3$}{L3}}
\label{subsection:reconnexionsPoincare}

% To prove the results in Theorem \ref{TheoremD},  we analyze 
The unstable and stable manifolds of  the critical point $\Ltres(\de)$  for small values of $\de$, 
have  two  branches, which are symmetric with respect to the involution  \eqref{def:involutionScaling} (see Figure~\ref{fig:perturbedInvariantManifolds1d}).

For $\de>0$, we denote by $\WW^{\unstable}(\Ltres)$ and $\WW^{\stable}(\Ltres)$ the $1$-dimensional unstable and stable manifolds of $\Ltres(\de)$.
In addition, as done in Section~\ref{section:introduction}, we consider each branch independently.
Let $\psi_t$ be the flow given by the Hamiltonian $H$ and $\mathbf{e_1}=(1,0,0,0)^T$. 
We denote
\begin{align*}
	\WW^{\unstable,+}(\Ltres) &= \claus{z \in \WW^{\unstable}(\Ltres) \st 
	\lim_{t \to -\infty}\sprod{\psi_t(z)}{\mathbf{e_1}} = 0^+},
	&
	\WW^{\stable,-}(\Ltres) &= {\Phi}\paren{\WW^{\unstable, +}},
	\\
	\WW^{\stable,+}(\Ltres) &= \claus{z \in \WW^{\stable}(\Ltres) \st 
	\lim_{t \to +\infty}\sprod{\psi_t(z)}{\mathbf{e_1}} = 0^+},
	&
	\qquad
	\WW^{\unstable,-}(\Ltres) &= {\Phi}\paren{\WW^{\stable, +}},
\end{align*}
the branches of $\WW^{\diamond}(\Ltres)$, for $\diamond=\unstable,\stable$.
%
%Then, by symmetry,
%$
%	W^{\diamond}(\Ltres)=W^{\diamond,+}(\Ltres) \cup W^{\diamond,-}(\Ltres).
%$	

Next result,  proven in \cite[Theorem 2.2]{articleOuter}, gives an asymptotic formula for the distance between the first intersection of the one dimensional manifolds $\WW^{\unstable,+}(\Ltres)$ and $\WW^{\stable,+}(\Ltres)$ on a suitable section.
In particular, 
Theorem~\ref{TheoremA} is a consequence of this result.

\begin{theorem} \label{mainTheoremDist}
Fix an interval $[\la_1,\la_2]\subset (0,\la_0)$ with $\la_0$ as given in \eqref{def:la0}.
There exists $\de_0>0$ such that, for $\de\in(0,\de_0)$ and $\la_*\in[\la_1,\la_2]$, the invariant manifolds $\WW^{\unstable,+}(\Ltres)$ and $\WW^{\stable,+}(\Ltres)$ intersect the section $\claus{\la=\la_*, \La >0}$.
Denote by $(\la_*,\La^{\unstable}_{\de},x^{\unstable}_{\de},y^{\unstable}_{\de})$
and
$(\la_*,\La^{\stable}_{\de},x^{\stable}_{\de},y^{\stable}_{\de})$ the first intersection points of the unstable and stable manifolds with this section, respectively. 
They satisfy
\begin{align*}
	y^{\unstable}_{\de}-y^{\stable}_{\de}
	&= 
	\sqrt[6]{2} \,
	\de^{\frac{1}{3}}  e^{-\frac{A}{\de^2}} 
	\boxClaus{{\CInn}
		+ \OO\paren{\frac{1}{\vabs{\log \de}}}},
	\qquad
	x^{\unstable}_{\de}-x^{\stable}_{\de} = \conj{y^{\unstable}_{\de}-y^{\stable}_{\de}},
	\\
	\La^{\unstable}_{\de}-\La^{\stable}_{\de} &=
	\OO \paren{\de^{\frac43} e^{-\frac{A}{\de^2}}},
\end{align*}
where $A>0$ and
$\CInn \in \complexs \setminus\{0\}$ are the constants described in Theorem \ref{TheoremA}.
%
%$\CInn \in \complexs$ is the Stokes constant described in Ansatz~\ref{ansatz} and
%$A>0$ is given by
%\begin{equation}\label{def:integralA}
%	A= \int_0^{\frac{\sqrt{2}-1}{2}} \frac{2}{1-x}\sqrt\frac{x}{3(x+1)(1-4x-4x^2)}  dx\approx 0.177744.
%\end{equation}
\end{theorem}

To prove Theorem \ref{TheoremB} and \ref{TheoremC}, it will be  convenient to analyze the distance between the invariant manifolds in the ``horizontal section''
\begin{align}\label{def:section0}
	\Si_{0} = \claus{\paren{\la,\La,x,y}\in \UReals(\cttTheoScalingDomainA, \cttTheoScalingDomainB) : 
	\La = \de^2\LtresLa(\de), \,
	 H(\la,\La,x,y)=H(\Ltres(\de))},
\end{align}
within the energy level of $\Ltres(\de)$, where $(x,y)$ define a system of coordinates.
% 
% we focus on the study of the ``$+$'' invariant manifolds.
% %
% By symmetry, there exist analogous results for the 
% ``$-$'' invariant manifolds.
% %
% In particular, for  $\cttTheoScalingDomainA, \cttTheoScalingDomainB>0$, we look for intersections between $\WW^{\unstable,+}(\po_{\rho})$ and $\WW^{\stable,+}(\po_{\rho})$ in a suitable $2$-dimensional section of the domain $\UReals(\cttTheoScalingDomainA, \cttTheoScalingDomainB)$ (see \eqref{def:dominiUReals}) within a fixed energy level.
% %
% In particular, for technical reasons, we choose such section as
% \begin{align}\label{def:section}
% 	\Si_{\rho} = \claus{\paren{\la,\La,x,y}\in \UReals(\cttTheoScalingDomainA, \cttTheoScalingDomainB) : 
% 	\La = \de^2\LtresLa(\de), \,
% 	 H(\la,\La,x,y)=\frac{\rho^2}{\de^2} + H(\Ltres(\de))},
% \end{align}
% where $\Ltres=(0,\de^2\LtresLa,\de^3\Ltresx,\de^3\Ltresy)^T$as given in Proposition~\ref{proposition:existenceFixedPoint}.
% 
% We notice that, by Proposition~\ref{proposition:periodicOrbit}, the periodic orbit $\po_{\rho}$ belongs to the energy level $H=\frac{\rho^2}{\de^2}+H(\Ltres(\de))$ where $\Sigma_{\rho}$ is included.
% %
% In addition, for $\rho=0$, one has that $\po_{0}=\Ltres$. 
%
The following corollary is a consequence of Theorem~\ref{mainTheoremDist}. It is proven in  Appendix \ref{appendix:changeSectionL3}.

\begin{corollary} \label{mainTheoremDistCorollary}
There exists $\de_0>0$ such that, for every $\de\in(0,\de_0)$, the invariant manifolds $\WW^{\unstable,+}(\Ltres)$ and $\WW^{\stable,+}(\Ltres)$ intersect the section $\Si_{0}$.
Denote by $(\la^{\unstable}_{\de},\de^2\LtresLa,x^{\unstable}_{\de},y^{\unstable}_{\de})$
and
$(\la^{\stable}_{\de},\de^2\LtresLa,x^{\stable}_{\de},y^{\stable}_{\de})$ the first intersection points of the unstable and stable manifolds, respectively, with the section.
Then, they satisfy
\begin{align*}
	\vabs{x^{\unstable}_{\de}-x^{\stable}_{\de}}
	=
	\vabs{y^{\unstable}_{\de}-y^{\stable}_{\de}}
	= 
	\sqrt[6]{2} \,
	\de^{\frac{1}{3}}  e^{-\frac{A}{\de^2}} 
	\boxClaus{\vabs{\CInn}
		+ \OO\paren{\frac{1}{\vabs{\log \de}}}}.
\end{align*}
% where $A>0$ and
% $\CInn \in \complexs$ are the constants described in Theorem \ref{mainTheoremDist}.
\end{corollary}

\section{2-round homoclinic orbits to \texorpdfstring{$L_3$}{L3}: Proof of Theorem~\ref{TheoremD}}
\label{section:proofReconnections}

In this section we study  the existence of $2$-round homoclinic connections to the $\Ltres(\de)$ (see \eqref{def:pointL3sca})  for certain values of the parameter $\de$ and we prove  Theorem~\ref{TheoremD}. We first restate it referred to the Hamiltonian \eqref{def:hamiltonianScaling} (recall that $\de= \mu^{\frac{1}{4}}$, see \eqref{def:delta}).

% In the next result, we study  the existence of $2$-round homoclinic connections to $\Ltres(\de)$ for certain values of the parameter $\de$.
%

\begin{theorem} \label{mainTheoremB}
% Assume Ansatz~\ref{ansatz}.
%
% Then, t
There exist $N_0>0$ and a sequence $\{{\de}_n\}_{n \geq N_0}$ satisfying
\begin{align*}
	\de_n 
	=
	\sqrt[8]{\frac8{21}}\sqrt[4]{\frac{A}{n\pi}}
	\paren{1+ 
		\OO \paren{\frac1{\log n}}},
	\qquad \text{for } n \geq N_0,
\end{align*}
such that, for each $n\geq N_0$, there exist a $2$-round homoclinic connection to the equilibrium point $\Ltres(\de_n)$ between $\WW^{\unstable,+}(\Ltres)$ and $\WW^{\stable,-}(\Ltres)$.
\end{theorem}

The rest of this section is devoted to prove this theorem. 
% Theorem~\ref{TheoremD} is a direct consequence of it. 

%
%There exists $\de_0>0$ such that, for $\de\in(0,\de_0)$, the invariant manifold $\WW^{\unstable,+}(\Ltres)$ intersects with the section $\claus{\la=0}$.
%%
%Let us denote by ${z}^{\unstable,+}_0\paren{\de}$ to the first intersection point of the $+$~branch of the unstable manifolds with the section.
%%
%Then, there exist $N_0>0$ and a sequence $\{{\de}_n\}_{n \geq N_0} \subset (0,\de_0)$ such that
%\[
%{z}^{\unstable,+}_0\paren{\de_n} \in \SSS,
%\qquad \text{for } n \geq N_0.
%\]
%%
%Moreover, 
%\begin{align*}
%	\de_n 
%	=
%	\sqrt[8]{\frac8{21}}\sqrt[4]{\frac{A}{n\pi}}
%	\paren{1+ 
%		\OO \paren{\frac1{\log n}}},
%	\qquad \text{for } n \geq N_0.
%\end{align*}	

%To prove Theorem~\ref{mainTheoremB} we see that, for certain values of the parameter $\de$, there exist connections between $\WW^{\unstable,+}(\Ltres)$ and $\WW^{\stable,-}(\Ltres)$ that only approach $\Ltres(\de)$ once.
%
%\textcolor{red}{An analogous result holds for connections between the branches $\WW^{\stable,+}(\Ltres)$ and $\WW^{\unstable,-}(\Ltres)$).}
%
%To prove Theorem \ref{TheoremD} we prove that, for a sequence of parameter $\de$,
%there exists an homoclinic connection between the branches $\WW^{\unstable,+}(\Ltres)$ and $\WW^{\stable,-}(\Ltres)$.
%%
%An analogous result holds for connections between the branches $\WW^{\stable,+}(\Ltres)$ and $\WW^{\unstable,-}(\Ltres)$.
%
To  prove Theorem~\ref{mainTheoremB}, we take advantage of the fact that the Hamiltonian $H$ is reversible with respect to the axis
$
\SSS = \claus{\la=0, \, x=y}
$
(see~\eqref{def:symmetryAxisScaling}).
Therefore, by symmetry, it is only necessary to see that there exists a sequence of $\de$ such that $\WW^{\unstable,+}(\Ltres)$ intersects the symmetry axis $\SSS$,
see Figure~\ref{fig:reconnectionsMain}.
\begin{figure}
	\centering
	\begin{overpic}[scale=1]{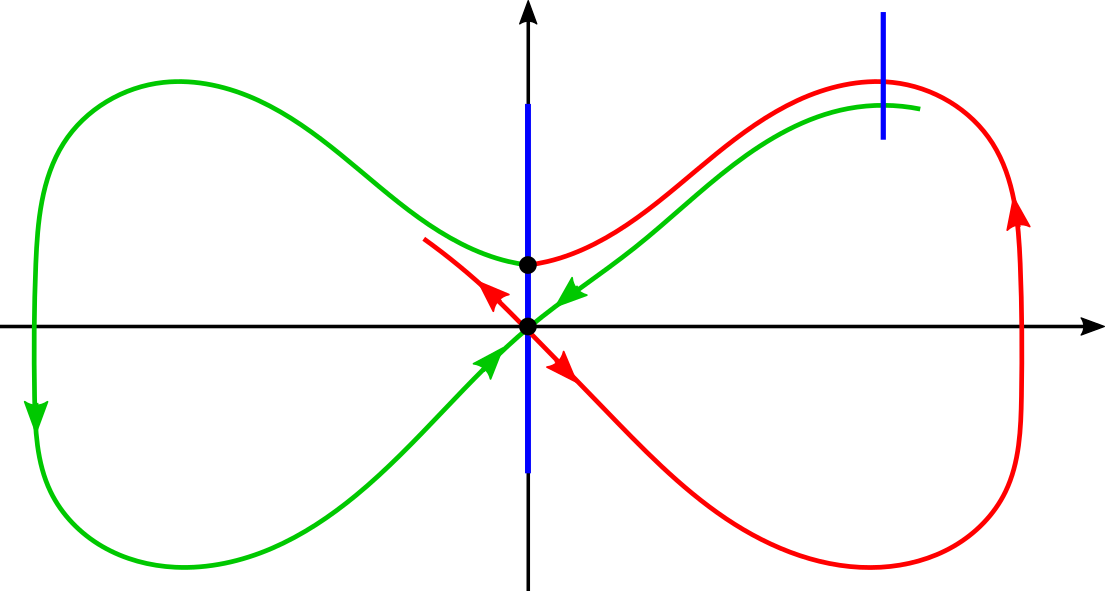}
		\put(101,22.5){$\la$}
		\put(46.5,54){$\La$}
		\put(13,6){\color{myGreen} $\WW^{\stable,-}(\Ltres)$}
		\put(28,25.5){\color{red}$\WW^{\unstable,-}(\Ltres)$}
		\put(56,25.5){\color{myGreen}$\WW^{\stable,+}(\Ltres)$}
		\put(70,6){\color{red} $\WW^{\unstable,+}(\Ltres)$}
		\put(81,49){\color{blue}$\claus{\la=\la_*, \La>0}$}
		\put(44,40){\color{blue}$\SSS$}
		%\put(49,34){$\Ltres$}
	\end{overpic}
	\caption{Projection into the $(\la,\La)$-plane of the unstable and stable manifolds and its intersections with the symmetry axis and section $\claus{\la=\la_*, \La>0}$.}
	\label{fig:reconnectionsMain}
\end{figure}
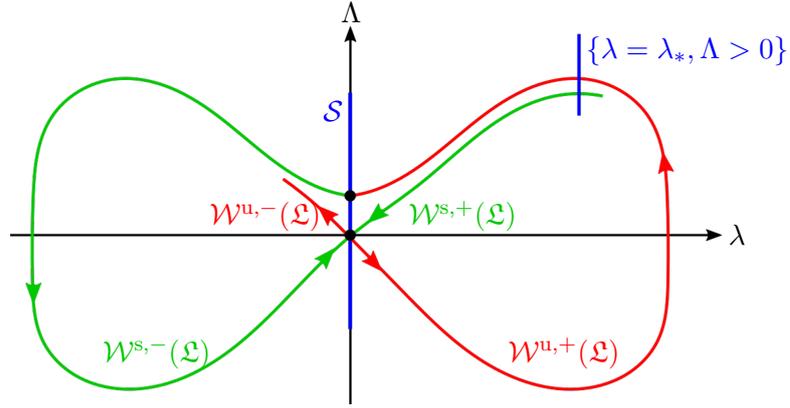
To this end, we extend the manifold $\WW^{\unstable,+}(\Ltres)$ from the section $\claus{\la=\la_*, \La>0}$, studied in Theorem~\ref{mainTheoremDist}, to a neighborhood of the critical point $\Ltres(\de)$ and look for intersections with $\SSS$. 
To study the invariant manifolds near $\Ltres(\de)$, we use a normal form result for Hamiltonian systems in a neighborhood of a saddle-center critical point.
Note that, the classical normal form result by Moser in~\cite{Moser58} is not enough for our purposes.
Indeed, we need to control that the radius of convergence of the normal form does not goes to zero when $\de\to 0$.
For that reason, we apply a more quantitative normal form obtained by T.~J\'ez\'equel, P.~Bernard and E.~Lombardi in~\cite{JBL16}.
% that ensures that the normalization does not blow up as $\de \to 0$.
%
% The complete proof of Theorem~\ref{mainTheoremB} is postponed to Section~\ref{section:proofReconnections}.

% \textcolor{blue}{
% 	The strategy to prove Theorem~\ref{TheoremD} consists on taking advantage of the fact that the Hamiltonian $\HInicial$ is reversible with respect to the involution \eqref{def:involutionCartesians}.
% 	%
% 	Therefore, $W^{\unstable,+}(L_3)$ and $W^{\stable,-}(L_3)$ are symmetric with respect to the symmetry axis $\claus{q_2=0, p_1=0}$ and, by symmetry, it is only necessary to prove that there exists a sequence of the parameters $\mu$ for which the invariant manifold $W^{\unstable,+}(L_3)$ intersects the symmetry axis.
% 	%
% 	To obtain this result, we need to extend  $W^{\unstable,+}(L_3)$ from section $\Si$ (as given in Theorem~\ref{TheoremA}) to a neighborhood of the equilibrium point $L_3$.}
% 	
% 	\textcolor{blue}{
% 	To study the invariant manifolds near $L_3$, we use a normal form result for Hamiltonian systems in a neighborhood of saddle-center critical points.
% 	%
% 	Note that, the classical normal form result by J. Moser in~\cite{Moser58} is not enough for our purposes.
% 	%
% 	Indeed, we need to control that the radius of convergence of the normal form does not goes to zero when $\mu\to 0$.
% 	%
% 	For that reason, we rely on a more quantitative normal form obtained by T.~J\'ez\'equel, P.~Bernard and E.~Lombardi in~\cite{JBL16} that ensures that the normalization does not blow up when $\mu \to 0$. }
% 

\subsection{Proof of Theorem~\ref{mainTheoremB}}

To prove Theorem~\ref{mainTheoremB}, we first perform a detailed local analysis of the Hamiltonian $H$ in \eqref{def:hamiltonianScaling} close to the equilibrium point~$\Ltres(\de)$.
%
%To do so, we apply the normal form obtained in~\cite{JBL16} that ensures that the normalization does not blow up when $\de \to 0$.
%
In the next proposition we introduce the normal form result given by T.~J\'ez\'equel, P.~Bernard and E.~Lombardi
in~\cite{JBL16} adapted to the Hamiltonian $H$.
Then, in Proposition~\ref{proposition:formaNormalConsequence}, we translate the results in Theorem~\ref{mainTheoremDist} and the symmetry axis $\SSS$ in \eqref{def:symmetryAxisScaling} into the new set of coordinates provided by the normal form.
%
% The proof of both results is postponed to Appendix~\ref{appendix:normalForm}.

\begin{proposition}\label{proposition:formaNormalTot}
There exist $\de_0, \rhoNormalForm,\cttTheoScalingDomainA,\cttTheoScalingDomainB>0$ and a family of analytic  changes of coordinates
\begin{align*}
	\FF_{\de}: B(\rhoNormalForm)=\claus{\mathbf{z} \in \reals^4 : \vabs{\mathbf{z}}<\varrho_0} &\to \UReals(\cttTheoScalingDomainA,\cttTheoScalingDomainB) 
	\\
	(\qU,\pU,\qD,\pD) &\mapsto 
	(\la,\La,x,y),
\end{align*}
defined for $\de \in (0,\de_0)$, with the following properties:
\begin{enumerate}
\item It is canonical with respect to the symplectic form $d\qU \wedge d\pU + d\qD \wedge d\pD$.
\item $\FF_{\de}(0)=\Ltres(\de)$.
\item The Hamiltonian $H$ (see \eqref{def:hamiltonianScaling}) in the new coordinates reads
\begin{align*}
	\HH(\qU,\pU,\qD,\pD;\de)
	&=
	H(\FF_{\de}(\qU,\pU,\qD,\pD);\de) - H(\Ltres(\de);\de)
	\\
	&= 
	\qU \pU 
	+ \frac{\al(\de)}{2 \de^2}\paren{\qD^2+\pD^2}
	+ \RRR (\qU \pU, \qD^2+\pD^2;\de),
\end{align*}
where $\al(\de)$ is a $\CC^1$-function satisfying that
$
\al(\de) = \sqrt{\frac8{21}} + \OO(\de^4)
$
and $\RRR$ satisfies 
\begin{align*}
	\vabs{\RRR(\qU \pU, \qD^2+\pD^2;\de)} \leq C \vabss{(\qU \pU, \qD^2+\pD^2)}^2,
\end{align*}
for $(\qU,\pU,\qD,\pD) \in B(\rhoNormalForm)$ and $C>0$ a constant independent of $\de$.
\end{enumerate}
\end{proposition}
The proof of this proposition, which is a consequence of the results  in \cite{JBL16}, is explained in Section \ref{app:proofNF}. Observe that the equations associated to the Hamiltonian $\HH$ are of the form
\begin{equation}\label{proof:equationTime}
\begin{aligned}
	\dot{\qU} &= \qU \paren{1 + \partial_1 \RRR(\qU\pU,\qD^2+\pD^2;\de)}
	\\
	\dot{\pU} &= -\pU \paren{1 + \partial_1 \RRR(\qU\pU,\qD^2+\pD^2;\de)}
	\\
	\dot{\qD} &= \pD \paren{\frac{\al(\de)}{\de^2} + 2\partial_2 \RRR(\qU\pU,\qD^2+\pD^2;\de)}
	\\
	\dot{\pD} &= -\qD \paren{\frac{\al(\de)}{\de^2} + 2\partial_2 \RRR(\qU\pU,\qD^2+\pD^2;\de)}.
\end{aligned}
\end{equation}
Since this system has two conserved quantities, $\qU \pU$ and $\qD^2+\pD^2$, its solutions are
\begin{align}\label{proof:localCoordinatesSolution}
	\begin{aligned}
		\qU(t) &= \qU(0) e^{\nu_1t}, 
		\\
		\pU(t) &= \pU(0) e^{-\nu_1t}, 
		\\
		\begin{pmatrix}
			\qD(t) \\
			\pD(t)
		\end{pmatrix}
		&=
		\begin{pmatrix}
			\cos{\nu_2t}
			&
			\sin{\nu_2t}
			\\
			-\sin{\nu_2t}
			&
			\cos{\nu_2t}
		\end{pmatrix}
% 		\Rot(t \nu_2(\de))
		%\begin{pmatrix}
		%	\cos{k_2(\de,\nu)t}
		%	&
		%	\sin{k_2(\de,\nu)t}
		%	\\
		%	-\sin{k_2(\de,\nu)t}
		%	&
		%	\cos{k_2(\de,\nu)t}
		%\end{pmatrix}
		\begin{pmatrix}
			\qD(0) \\
			\pD(0)
		\end{pmatrix},
% 		\qquad
% 		\Rot(\tht) = 
% 		\begin{pmatrix}
% 			\cos{\tht}
% 			&
% 			\sin{\tht}
% 			\\
% 			-\sin{\tht}
% 			&
% 			\cos{\tht}
% 		\end{pmatrix},
	\end{aligned}
\end{align}
where, for $(\qU(0),\pU(0),\qD(0),\pD(0)) \in B(\rhoNormalForm)$,
\begin{equation}\label{proof:localCoordinatesSolutionK}
	\begin{aligned}
		\nu_1=\nu_1(\de) &= 1+
		\partial_1 \RRR\paren{\qU(0)\pU(0), \qD^2(0)+\pD^2(0);\de}>0,
		\\
		\nu_2=\nu_2(\de) &= \frac{\al(\de)}{\de^2}
		+ 2
		\partial_2 \RRR\paren{\qU(0)\pU(0), \qD^2(0)+\pD^2(0);\de}>0.
	\end{aligned}
\end{equation}
%Let us denote $w^{\unstable}_{\de}(0)$ and $w^{\stable}_{\de}(0)$ to the expression in coordinates $(\qU,\pU,\qD,\pD)$ of a neighborhood of the unstable and stable manifolds, $\WW^{\unstable}(\Ltres(\de))$ and $\WW^{\stable}(\Ltres(\de))$, close to the equilibrium point. 
%%
%That is the subsets of $B(0;\varrho_0)$ such that
%\begin{equation*}
%	\FF_{\de}(w^{\unstable}_{\de}(0)) 
%	\subset
%	\WW^{\unstable}(\Ltres(\de)),
%	\qquad
%	\FF_{\de}(w^{\stable}_{\de}(0)) 
%	\subset
%	\WW^{\stable}(\Ltres(\de)).
%\end{equation*}
%%
%By Proposition~\ref{proposition:formaNormalTot}, one has that $(\qUu,\pUu,\qDu,\pDu) \in w^{\unstable}_{\de}(0)$ and $(\qUs,\pUs,\qDs,\pDs) \in w^{\stable}_{\de}(0)$, (see Figure \ref{fig:reconnections}).
%
Notice that the local unstable and stable manifolds are given by 
$\claus{\pU=\qD=\pD=0}$ and 
$\claus{\qU=\qD=\pD=0}$, respectively.
%\[
%\claus{\pU=0, \, \qD=\pD=0} \subset w^{\unstable}_{\de}(0),
%\qquad
%\claus{\qU=0, \, \qD=\pD=0} \subset w^{\stable}_{\de}(0).
%\]
%

\begin{proposition}\label{proposition:formaNormalConsequence}
Consider the constants $\varrho_0,\delta_0$ given by Proposition~\ref{proposition:formaNormalTot}.
Then,
\begin{enumerate}
	\item There exists $\la_* \in(0,\la_0)$ and $\delta_1\in (0,\delta_0)$ such that, for any $\delta\in (0,\delta_1)$, the first intersections $\mathbf{z}^{\unstable}_{\de}(\la_*)$ and $\mathbf{z}^{\stable}_{\de}(\la_*)$ of the invariant manifolds $\WW^{\unstable,+}(\Ltres)$ and $\WW^{\stable,+}(\Ltres)$ with the section $\claus{\la=\la_*, \La>0}$  respectively (see Theorem~\ref{mainTheoremDist}), satisfy that
	\begin{align}\label{eq:changeGlobalToLocal}
		(\qUu,\pUu,\qDu,\pDu) 
		=
		\FF_{\de}\paren{
			\mathbf{z}^{\unstable}_{\de}(\la_*(\varrho))}, 
		\qquad
		(\qUs,\pUs,\qDs,\pDs) 
		=
		\FF_{\de}\paren{
			\mathbf{z}^{\stable}_{\de}(\la_*(\varrho))}
	\end{align}
% satisfy $(\qUu,\pUu,\qDu,\pDu), (\qUs,\pUs,\qDs,\pDs) \in
belong to the ball $B(\rhoNormalForm)$.
% 	Fix an interval $[\la_1,\la_2]\subset (0,\la_0)$ with $\la_0$ as given in \eqref{def:la0} and denote $\mathbf{z}^{\unstable}_{\de}(\la_*)$ and $\mathbf{z}^{\stable}_{\de}(\la_*)$ to the first intersection of the invariant manifolds $\WW^{\unstable,+}(\Ltres)$ and $\WW^{\stable,+}(\Ltres)$ with the section $\claus{\la=\la_*, \La>0}$ with $\la_*\in[\la_1,\la_2]$, respectively (see Theorem~\ref{mainTheoremDist}).
	
% 	\textcolor{red}{Parlar-ho per al paper. Per la tesis es queda aix\'i.} 
% 	
% 	There exist $0<\rhoLimitA<\rhoLimitB<\rhoNormalForm$ such that, 
% 	for $\varrho \in [\rhoLimitA,\rhoLimitB]$ and $\de \in (0,\de_0)$, 
% 	%
% 	there exist
% 	$\la_*(\varrho) \in [\la_1,\la_2]$ 
% 	and $(\qUu,\pUu,\qDu,\pDu), (\qUs,\pUs,\qDs,\pDs) \in B(\rhoNormalForm)$ such that
Moreover, there exists $\varrho\in (0,\varrho_0)$ such that, for $\de\in (0,\de_1)$, these points can be written as
% 	\begin{align}\label{eq:changeGlobalToLocal}
% 		(\qUu,\pUu,\qDu,\pDu) 
% 		=
% 		\FF_{\de}\paren{
% 			\mathbf{z}^{\unstable}_{\de}(\la_*(\varrho))}, 
% 		\qquad
% 		(\qUs,\pUs,\qDs,\pDs) 
% 		=
% 		\FF_{\de}\paren{
% 			\mathbf{z}^{\stable}_{\de}(\la_*(\varrho))},
% 	\end{align}
% 	and
	\begin{align*}
		\qUu &= -
		\frac{\sqrt[3]{2}}{\varrho} 
		\de^{-\frac{4}{3}}
		e^{-\frac{2A}{\de^2}} 
		\boxClaus{\vabs{\CInn}^2
			+ \OO\paren{\frac{1}{\vabs{\log \de}}}},
		& 
		\qUs &= 0,
		\\
		\pUu &= \varrho + \OO\paren{\de^{\frac43} e^{-\frac{A}{\de^2}}},
		&
		\pUs & = \varrho,
		\\
		\qDu &= \sqrt[3]{4}\sqrt[4]{\frac{21}8} \,
		\de^{\frac{1}{3}}  e^{-\frac{A}{\de^2}} 
		\boxClaus{\Re{\CInn}
			+ \OO\paren{\frac{1}{\vabs{\log \de}}}},
		&
		\qDs &= 0,
		\\
		\pDu &= \sqrt[3]{4}\sqrt[4]{\frac{21}8} \,
		\de^{\frac{1}{3}}  e^{-\frac{A}{\de^2}} 
		\boxClaus{-\Im{\CInn}
			+ \OO\paren{\frac{1}{\vabs{\log \de}}}},
		&
		\pDs &= 0.
		\end{align*}
	\item Let $\SSS=\claus{\la=0, x=y}$ be the symmetry axis \eqref{def:symmetryAxisScaling} of the Hamiltonian $H$.
	There exist real-analytic functions $\Psi_1,\Psi_2: B(\rhoNormalForm) \times (0,\de_0) \to \reals$ and a constant $C>0$ such that the curve
	\begin{align}\label{def:symmetryAxisLocal}
		\SSS_{\local} = \big\{
		\qU+\pU = \Psi_1(\qU,\pU,\qD,\pD;\de), \,
		\pD = \Psi_2(\qU,\pU,\qD,\pD;\de)
		\big\}
	\end{align}
	satisfies that $\FF_{\de}(\SSS_{\local}) \subset \SSS$ and, for $(\qU,\pU,\qD,\pD;\de)\in B(\rhoNormalForm)\times (0,\de_0)$,
	\begin{enumerate}
		\item $\vabs{\Psi_1(\qU,\pU,\qD,\pD;\de)} \leq C\de\vabs{(\qU,\pU,\qD,\pD)} + 
		C \vabs{(\qU,\pU)}^2$,
		\item $\vabs{\Psi_2(\qU,\pU,\qD,\pD;\de)} \leq C\de\vabs{(\qU,\pU,\qD,\pD)}$.
	\end{enumerate}
\end{enumerate}
\end{proposition}

This proposition is proven in Section \ref{app:translationrecon}.

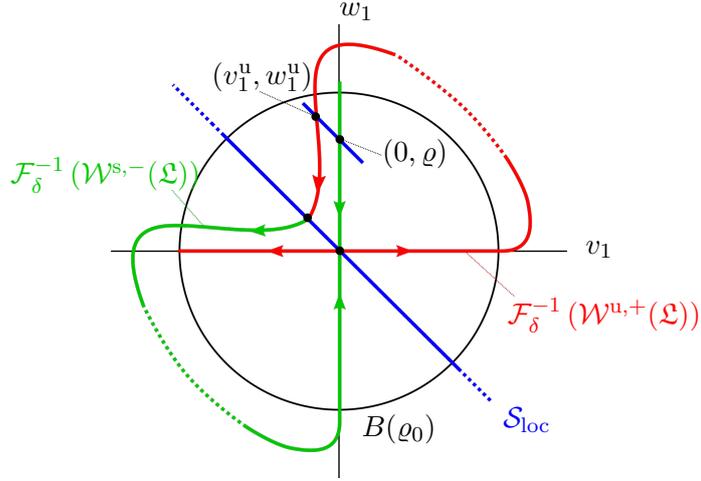
\begin{figure}
	\centering
	\begin{overpic}[height=6cm]{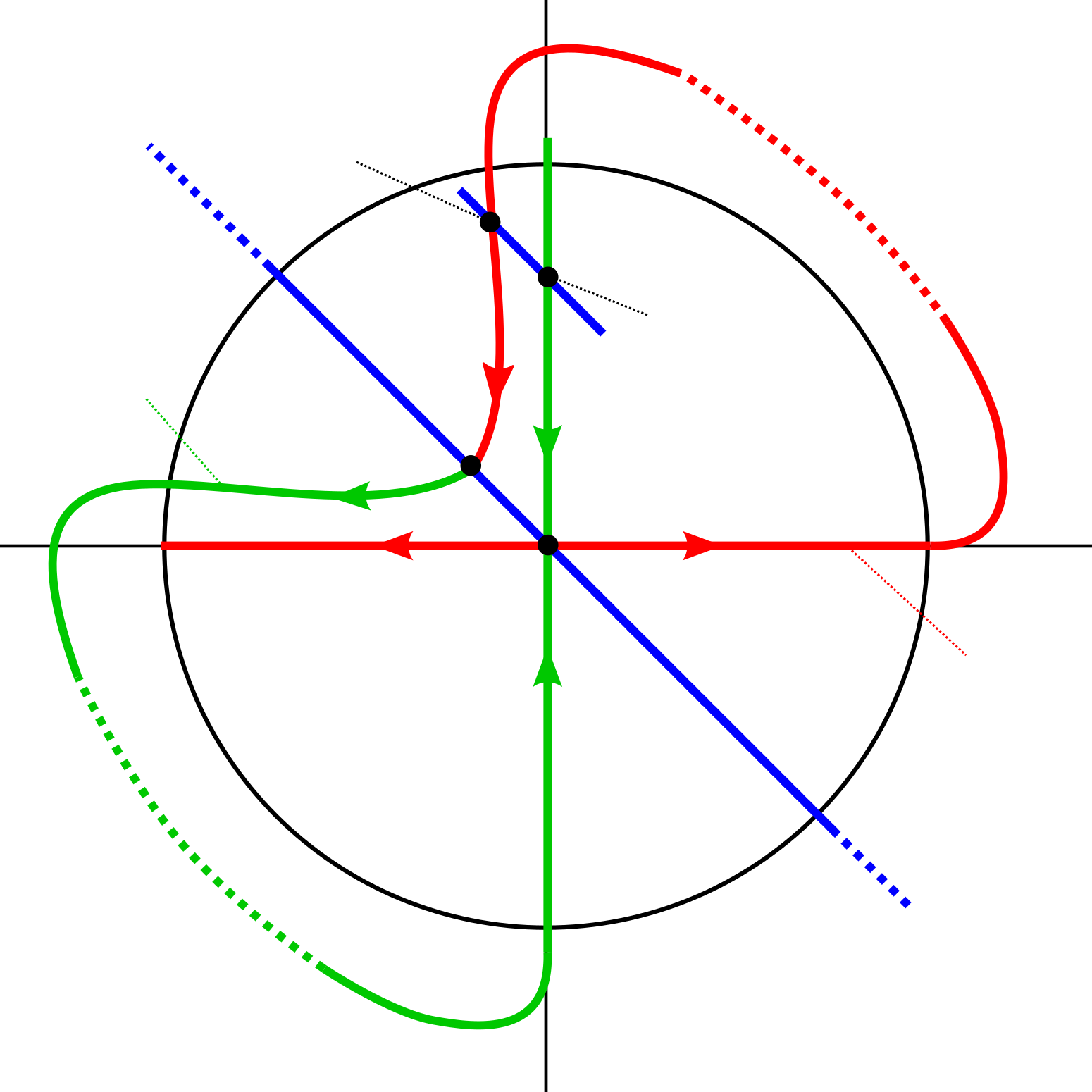}
		\put(87,34){\textcolor{red}{
				$\FF_{\de}^{-1}\paren{ \WW^{\unstable,+}(\Ltres)}$}}
		\put(-22,65){\textcolor{myGreen}{
				$\FF_{\de}^{-1}\paren{ \WW^{\stable,-}(\Ltres)}$}}
		%\put(62,53){\textcolor{blue}{$\Sigma_{\de}$}}
		\put(86,11){\textcolor{blue}{$\SSS_{\local}$}}
		\put(55,9){$B(\rhoNormalForm)$}
		\put(59.5,70){$(0,\varrho)$}
		\put(20,87){ $(\qUu,\pUu)$}
		\put(102,49.5){ $\qU$}
		\put(48,102){ $\pU$}
	\end{overpic}
	\caption{Representation of the unstable and stable manifolds in local coordinates $(\qU,\pU,\qD,\pD)$ given in Propositions~\ref{proposition:formaNormalTot} and~\ref{proposition:formaNormalConsequence}. }
	\label{fig:reconnections}
\end{figure}

From now on, we work in the set of local coordinates $(\qU,\pU,\qD,\pD) \in B(\rhoNormalForm)$ given in Proposition~\ref{proposition:formaNormalTot}.
Then, to prove Theorem~\ref{mainTheoremB}, it remains to extend the unstable manifold from the point $(\qUu,\pUu,\qDu,\pDu)$ given in~\eqref{eq:changeGlobalToLocal} and to analyze for which values of  $\de>0$ it intersects with the symmetry curve $\SSS_{\local}$ given in~\eqref{def:symmetryAxisLocal}, (see Figure \ref{fig:reconnections}).
%
%By Proposition~\ref{proposition:formaNormalTot}, one can see that the linear and independent of $\de$ part of $\SSS_{\local}$ is
%\begin{align*}%\label{def:symmetryAxisFake}
%	\claus{\qU + \pU = 0, \, \pD=0}.
%\end{align*}

To give an intuition of the proof of Theorem \ref{mainTheoremB}, in the next lemma, we consider the intersection of the unstable manifold with a convenient ``first order'' of the symmetry axis $\SSS_{\local}$.
From now on, we denote by $C$ any positive constant independent of $\de$.

\begin{lemma}\label{lemma:reconnexionsPartial}
% Assume Ansatz~\ref{ansatz}.
%
Let $\Phi^{\unstable}(t;\de)$ be the trajectory of the Hamiltonian system given by  $\HH$ in Proposition~\ref{proposition:formaNormalTot} with initial condition $(\qUu,\pUu,\qDu,\pDu)$ as given in Proposition~\ref{proposition:formaNormalConsequence}.
Then, there exist $N_0>0$ and sequences
$\{\wh{T}_n\}_{n \geq N_0}$ and $\{\wh{\de}_n\}_{n \geq N_0}$ such that, for $n\geq N_0$,
\[
\Phi^{\unstable}(\wh{T}_n;\wh{\de}_n) \in \claus{\qU+\pU=0,\, \pD=0}.
%\qquad
%\text{for } n\geq N_0.
\]
Moreover,
\begin{align*}
\wh{\de}_n 
=
\sqrt[8]{\frac8{21}}\sqrt[4]{\frac{A}{n\pi}}
\paren{1+ 
\OO \paren{\frac1{\log n}}},
\qquad \text{for } n \geq N_0,
\end{align*}
where $A>0$ is the constant introduced in Theorem~\ref{TheoremA}.
\end{lemma}

\begin{proof}
Let $(\qU(t),\pU(t),\qD(t),\pD(t))$	be a trajectory of the Hamiltonian system given by $\HH$.
We want to find $\de>0$ such that there exists $T_{\de}^0 >0$ satisfying
\begin{align*}
\paren{\qU(0),\pU(0),\qD(0),\pD(0)}
&=
(\qUu,\pUu,\qDu,\pDu), 
\\
\paren{\qU(T^0_{\de}),\pU(T^0_{\de}),
	\qD(T^0_{\de}),\pD(T^0_{\de})}
&\in \claus{\qU+\pU=0,\, \pD=0}.
\end{align*}
In other words, using~\eqref{proof:localCoordinatesSolution},
\begin{align}
	\label{proof:eqReconnexionsA}
	\qUu e^{\nu_1 T^0_{\de}} 
	+ \pUu 	e^{-\nu_1 T^0_{\de}} = 0, \\
	\label{proof:eqReconnexionsB}
	\cos(\nu_2T^0_{\de}) \pDu - \sin(\nu_2T^0_{\de}) \qDu = 0,
\end{align}
where, by its definition in \eqref{proof:localCoordinatesSolutionK} and Proposition~\ref{proposition:formaNormalConsequence}, one has that
\begin{equation}\label{proof:functionsKasymptotics}
	\nu_1=\nu_1(\de) = 1 + \OO\paren{\de^{-\frac83} e^{-\frac{4A}{\de^2}}}, 
	\qquad
	\nu_2=\nu_2(\de) = \frac1{\de^2}\sqrt{\frac8{21}}+ \OO\paren{\de^2}.
\end{equation}
For any $\delta$ small enough, equation~\eqref{proof:eqReconnexionsA} has the solution
\begin{equation}\label{proof:Tdelta}
\begin{aligned}
	T^0_{\de}
	= -\frac1{2 \nu_1}\ln\paren{-\frac{\qUu}{\pUu}}
%	&= 
%	-\frac{A}{\de^2} - \frac23 \log \de 
%	+ \log \paren{\frac{\sqrt[3]{2}\vabs{\CInn}^2}{\qUs}}
%	+ \OO \paren{\frac1{\vabs{\log \de}}} \\
&= 
\frac{A}{\de^2} + \frac23\log\de	
- \log \paren{\sqrt[6]{2}\vabs{\CInn}\varrho^{-1}}
+ \OO\paren{\frac1{\vabs{\log\de}}}
\\
&= 
\frac{A}{\de^2}\paren{1+ \OO(\de^2\vabs{\log \de})}.
\end{aligned}
\end{equation}
Next, we study equation \eqref{proof:eqReconnexionsB}.
Let us denote $\tht=\arg \CInn$.
From Proposition~\ref{proposition:formaNormalConsequence},
\begin{align*}
\qDu &= \sqrt[3]{4}\sqrt[4]{\frac{21}8} \,
\de^{\frac{1}{3}}  e^{-\frac{A}{\de^2}} 
\boxClaus{\vabs{\CInn}\cos\tht
	+ \OO\paren{\frac{1}{\vabs{\log \de}}}}
\\
\pDu &= -\sqrt[3]{4}\sqrt[4]{\frac{21}8} \,
\de^{\frac{1}{3}}  e^{-\frac{A}{\de^2}} 
\boxClaus{\vabs{\CInn}\sin \tht
	+ \OO\paren{\frac{1}{\vabs{\log \de}}}}.
\end{align*}
By Theorem \ref{TheoremA}, one has that ${\CInn} \neq 0$.
Then, \eqref{proof:eqReconnexionsB} is equivalent to
\begin{align*}
\cos(\nu_2T^0_{\de})\sin\tht + 
\sin(\nu_2T^0_{\de})\cos\tht = 
\sin\paren{\tht+\nu_2T^0_{\de}} 
=g_0(\de),
\end{align*}
where $g_0(\de)$ contains the higher order terms,
%\begin{equation*}%\label{proof:functiong0}
%\begin{aligned}	
%g_0(\de) =& \,
%-\cos(\nu_2 T^0_{\de})\paren{
%\sqrt[4]{\frac8{21}}
%\frac{e^{\frac{A}{\de^2}}\de^{-\frac13}}{
%	\sqrt[3]{4}\vabs{\CInn}}
%\pDu + \sin \tht}
%\\
%&- \sin(\nu_2T^0_{\de}) 
%\paren{
%\sqrt[4]{\frac8{21}}
%\frac{e^{\frac{A}{\de^2}}\de^{-\frac13}}{
%	\sqrt[3]{4}\vabs{\CInn}}
%	\qDu - \cos\tht}.
%\end{aligned}
%\end{equation*} 
\begin{equation*}%\label{proof:functiong0}
g_0(\de) =
-\cos(\nu_2 T^0_{\de})\paren{
\sqrt[4]{\frac8{21}}
\frac{e^{\frac{A}{\de^2}}\de^{-\frac13}}{
	\sqrt[3]{4}\vabs{\CInn}}
\pDu + \sin \tht}
- \sin(\nu_2T^0_{\de}) 
\paren{
\sqrt[4]{\frac8{21}}
\frac{e^{\frac{A}{\de^2}}\de^{-\frac13}}{
	\sqrt[3]{4}\vabs{\CInn}}
	\qDu - \cos\tht}.
\end{equation*} 
and satisfies  $g_0(\de)=\OO(\vabs{\log \de}^{-1})$.
We deduce then that, for $n \in \integers$,
\begin{align*}
\nu_2 T_{\de}^0 + \tht = n\pi - \arcsin g_0(\de).
\end{align*}
Using the asymptotic expressions of $\nu_2=\nu_2(\de)$ and $T_{\de}^0$ in \eqref{proof:functionsKasymptotics} and \eqref{proof:Tdelta}, we have that $\de$ has to satisfy 
\begin{align*}
\frac{A}{\de^4}\sqrt{\frac8{21}}\paren{1+ g_1(\de)}
 = \pi n ,
\end{align*}
where $g_1(\de)=\OO(\de^2 \vabs{\log \de}).$
%\begin{align*}
%g_1(\de) &= \frac{\de^4}{A}\sqrt{\frac{21}8}
%\paren{\tht - \arcsin \paren{\frac{g_0(\de)}{\vabs{\CInn}} }
%- k_2(\de,\de) T_{\de}^0} - 1 = \OO(\de^2 \vabs{\log \de}).
%\end{align*}
%
%Lastly, we consider the equation $f(\de)=\e$ where
%\begin{align*}
%	f(\de) = \frac{\de}{\sqrt[4]{1+g_1(\de)}}, 
%	\qquad
%	\e^4 = \sqrt{\frac8{21}}\frac{A}{\pi n}.
%\end{align*}
Therefore, there exists $N_0>0$ and a sequence $\{\wh{\de}_n\}_{n\geq N_0} \subset (0,\de_1)$
satisfying the previous equation and the asymptotic expression of the lemma.
Finally, one has that $\wh{T}_n = T^0_{\de}$ for $\de=\wh{\de}_n$.
\end{proof}

%\begin{proposition}\label{proposition:reconnexionsComplete}
%Assume Ansatz~\ref{ansatz}.
%%	
%There exist $N_0>0$ and a succession $\{{\de}_n\}_{n \geq N_0}$ such that the unstable manifold $w^{\unstable}_{\de}(0)$ intersects the symmetry axis $\SSS$ given in \eqref{def:symmetryAxis} for $\de=\de_n$ and $n\geq N_0$.
%%
%Moreover, 
%\begin{align*}
%	\de_n 
%	=
%	\sqrt[8]{\frac8{21}}\sqrt[4]{\frac{A}{n\pi}}
%	\paren{1+ 
%		\OO \paren{\frac1{\log n}}}.
%\end{align*}	
%\end{proposition}

\begin{proof}[End of the proof of Theorem~\ref{mainTheoremB}]
We proceed analogously to the proof of Lemma \ref{lemma:reconnexionsPartial}.
Let us consider the expressions of $(\qU(t),\pU(t),\qD(t),\pD(t))$ given in~\eqref{proof:localCoordinatesSolution} and $T_{\de}>0$, such that
\begin{align*}
\paren{\qU(0),\pU(0),\qD(0),\pD(0)}
&=
(\qUu,\pUu,\qDu,\pDu), 
\\
\paren{\qU(T_{\de}),\pU(T_{\de}),
	\qD(T_{\de}),\pD(T_{\de})}
&\in \SSS_{\local},
\end{align*}
with $\SSS_{\local}= \claus{\qU+\pU=\Psi_1, \, \pD=\Psi_2}$ as given in Proposition~\ref{proposition:formaNormalConsequence}.

First, we deal with the equation $\qU+\pU = \Psi_1$.
Then, $T_{\de}$ must satisfy
\begin{equation}\label{proof:equation1complete}
\qU(T_{\de}) + \pU(T_{\de})
=	
\Psi_1 \paren{\qU(T_{\de}),\pU(T_{\de}),\qD(T_{\de}),\pD(T_{\de})}.
\end{equation}
%In particular, by the definition of $\Psi_1$ in \eqref{proof:compositionChanges}, Lemma \ref{lemma:normalForm} and Proposition \ref{proposition:normalForm},
%one sees that there exists a $\CC^1$ function $\gamma_1:[0,\de_0] \to \reals^4$ such that, for $(q,p) \in B(0;\rho_0)$,
%\begin{equation}\label{proof:estimatesPsi1}
%\vabss{\wh{\Psi}_1(q,p) - \de^4 \sprod{\gamma_1(\de)}{(q,p)^T}}  
%\leq  C \vabs{(q,p)}^2.
%\end{equation}
Let us denote $\tau=\tau(\delta)=T_{\de}-T_{\de}^0$, with $T_{\de}^0$ satisfying $\qU(T_{\de}^0)+\pU(T_{\de}^0)=0$ (see equations~\eqref{proof:eqReconnexionsA} and~\eqref{proof:Tdelta}).
Then, by \eqref{proof:localCoordinatesSolution}, $\tau$ has to satisfy 
\begin{align*}
\qU(T_{\de}^0)e^{\nu_1 \tau} + 
\pU(T_{\de}^0)e^{-\nu_1 \tau}
&=
\pU(T_{\de}^0)(e^{-\nu_1 \tau}-e^{\nu_1 \tau}) \\
&=
{\Psi}_1(\qU(T_{\de}^0+\tau),\pU(T_{\de}^0+\tau),
\qD(T_{\de}^0+\tau),\pD(T_{\de}^0+\tau)).
\end{align*}
Namely, $\tau(\de)=F[\tau](\de)$ with
\begin{align*}
F[\tau](\de)
= 
\frac{e^{-\nu_1\tau}-e^{\nu_1\tau}+
	2 \tau \nu_1}{2 \nu_1}
-
\frac{{\Psi}_1(\qU(t),\pU(t),\qD(t),\pD(t))|_{t=T_{\de}^0+\tau}}{2\nu_1\pU(T_{\de}^0)}.
%
%\frac{
%\qU(T_{\de}^0)(1+ k_1 \tau- e^{k_1 \tau})
%+
%\pU(T_{\de}^0)(1 - k_1\tau- e^{-k_1 \tau })
%- 
%\wh{\Psi}_1(q(T_{\de}^0+\tau),p(T_{\de}^0+\tau))
%}{k_1(\qU(T_{\de}^0)+\pU(T_{\de}^0))}.
\end{align*}
First we obtain estimates for $F[0](\de)$. By Proposition~\ref{proposition:formaNormalConsequence} and \eqref{proof:functionsKasymptotics},
\begin{align*}
\vabs{F[0](\de)} 
&\leq \frac{\vabs{{\Psi}_1(\qU(T_{\de}^0),\pU(T_{\de}^0),\qD(T_{\de}^0),\pD(T_{\de}^0))}}{2\vabs{\nu_1 \pU(T_{\de}^0)}} 
\\
&\leq C
\frac{\de\vabs{(\qG(T_{\de}^0),\pG(T_{\de}^0))}
+ \vabs{(\qU(T_{\de}^0),\pU(T_{\de}^0))}^2
}{\vabs{\pU(T_{\de}^0)}}.
\end{align*}
Let us recall that, by \eqref{proof:Tdelta}, we have an asymptotic expression for $T_{\de}^0$.
Then, by~\eqref{proof:localCoordinatesSolution}, \eqref{proof:functionsKasymptotics} and Proposition~\ref{proposition:formaNormalConsequence},
\begin{equation}\label{proof:assymptotisInitial}
\begin{aligned}
\pU(T_{\de}^0) 
&= \pUu e^{-\nu_1 T_{\de}^0}
= \sqrt[6]{2} \vabs{\CInn} \de^{-\frac23} e^{-\frac{A}{\de^2}} 
\boxClaus{1 + \OO\paren{\frac1{\vabs{\log\de}}}},
\\
\qD(T_{\de}^0) 
&=\cos(\nu_2 T_{\de}^0) \qDu + 
\sin(\nu_2 T_{\de}^0) \pDu 
= \OO\paren{\de^{\frac13}e^{-\frac{A}{\de^2}}},
%&= \sqrt[3]{4}\sqrt[4]{\frac{21}8} \de^{\frac13} e^{-\frac{A}{\de^2}} 
%\vabs{\CInn} 
%\paren{\cos(\tht-k_2(\de,\de) T_{\de}^0)+ \OO\paren{\frac1{\vabs{\log\de}}}},
\\
\pD(T_{\de}^0) 
&=-\sin(\nu_2 T_{\de}^0) \qDu 
+\cos(\nu_2 T_{\de}^0) \pDu 
= \OO\paren{\de^{\frac13}e^{-\frac{A}{\de^2}}}.
%&= \sqrt[3]{4}\sqrt[4]{\frac{21}8} \de^{\frac13} e^{-\frac{A}{\de^2}} 
%\vabs{\CInn} 
%\paren{\sin(\tht-k_2(\de,\de) T_{\de}^0)+ \OO\paren{\frac1{\vabs{\log\de}}}}.
\end{aligned}
\end{equation}
Since $\qU(T_{\de}^0)=-\pU(T_{\de}^0)$, one has that
$
\vabs{F[0](\de)} \leq C\de.
$
Next, we study the Lipschitz constant of the operator $F$. 
Let us consider continuous functions $\tau_0,\tau_1: (0,\de_0) \to \reals$ such that $\vabss{\tau_0(\de)}, \vabss{\tau_1(\de)} \leq C\de$
and the function $\tau_{\sigma}=\sigma \tau_1 + (1-\sigma)\tau_0$.
Then, by the mean value theorem, 
\begin{align*}
\vabs{F[\tau_1](\de)-F[\tau_0](\de)} 
&\leq \,
C \vabs{\tau_1(\de)-\tau_0(\de)} \cdot \\
&\sup_{\sigma \in [0,1]}
\Big\{
 \vabs{\tau_{\sigma}(\de)}^2
+
\de^{\frac23} e^{\frac{A}{\de^2}} 
\vabs{
D{\Psi}_1(\qU,\pU,\qD,\pD)\cdot
(\dot{\qU},\dot{\pU},\dot{\qD},\dot{\pD})^T}_{t=T_{\de}^0+\tau_{\sigma}(\de)}
\Big\}.
%\\
%&{\Psi}_1(q(T_{\de}^0+\tau_{\sigma}),p(T_{\de}^0+\tau_{\sigma})) \cdot (\partial_t {q}(T_{\de}^0+\tau_{\sigma}),\partial_t {p}(T_{\de}^0+\tau_{\sigma}))^T}
%\Big\}.
\end{align*}
Since $\Psi_1$ is a real-analytic function, by Proposition~\ref{proposition:formaNormalConsequence}, one has $\vabs{D\Psi_1} \leq C\de + C\vabs{(\qU,\pU)}$.
Moreover, using \eqref{proof:equationTime}, one can obtain estimates for the derivatives $(\dot{\qU},\dot{\pU},\dot{\qD},\dot{\pD})$. 
Then,
\begin{align*}
\vabs{F[\tau_1](\de)-F[\tau_0](\de)} \leq C \de \vabs{\tau_1(\de)-\tau_0(\de)}.	
\end{align*} 
This implies that, taking $\de>0$ small enough,
% $
% \vabs{F[\tau_1]-F[\tau_0]} \leq \frac12 \vabs{\tau_1-\tau_0}
% $
% and, as a result,
$F$ is well defined and contractive.
Hence, $F$ has a fixed point $\tau(\de)$ such that $\vabs{\tau(\de)}\leq C\de$.
Therefore, there exists $T_{\de}$ satisfying equation~\eqref{proof:equation1complete} such that
\begin{align}\label{proof:TdeltaSencer}
	T_{\de} = T_{\de}^0 + \tau(\de) = 
	\frac{A}{\de^2}(1 + \OO(\de^2 \vabs{\log \de})).
\end{align}
Next, we study the equation $\pD=\Psi_2$.
One has that $\de>0$ must satisfy
\begin{equation}\label{proof:equation2complete}
	\pD(T_{\de})
	=	
	{\Psi}_2\paren{\qG(T_{\de}),\pG(T_{\de})}.
\end{equation}
Theorem~\ref{TheoremA} implies that ${\CInn} \neq 0$.
Then, by \eqref{proof:localCoordinatesSolution}, $\de$ has to satisfy
\begin{align*}
% \cos(\nu_2(\de)T_{\de})\sin\tht +
% \sin(\nu_2(\de)T_{\de})\cos\tht =
\sin\paren{\tht+\nu_2T_{\de}} 
=\wh{g_0}(\de),
\end{align*}
where
\begin{equation*}%\label{proof:functiong0}
	\begin{aligned}	
		\wh{g_0}(\de) =& \,
		\Psi_2\paren{\qG(T_{\de}),\pG(T_{\de})}
		-\cos(\nu_2 T_{\de})\paren{
			\sqrt[4]{\frac8{21}}
			\frac{e^{\frac{A}{\de^2}}\de^{-\frac13}}{
				\sqrt[3]{4}\vabs{\CInn}}
			\pDu + \sin \tht}
		\\
		&- \sin(\nu_2T_{\de}) 
		\paren{
			\sqrt[4]{\frac8{21}}
			\frac{e^{\frac{A}{\de^2}}\de^{-\frac13}}{
				\sqrt[3]{4}\vabs{\CInn}}
			\qDu - \cos\tht}.
	\end{aligned}
\end{equation*}
Then, we deduce that, for $n \in \integers$,
\begin{align*}
	\nu_2 T_{\de} + \tht = n\pi - \arcsin \left(\wh{g_0}(\de)\right).
\end{align*}
By Proposition~\ref{proposition:formaNormalConsequence} and using the asymptotic expressions in  \eqref{proof:assymptotisInitial} and \eqref{proof:TdeltaSencer}, 
\begin{align*}
	\vabs{\wh{g_0}(\de)} &\leq C \de \vabs{(\qG(T_{\de}),\pG(T_{\de}))}
	+ \frac{C}{\vabs{\log \de}} 
	\leq C \de \vabs{(\qG(T^0_{\de}),\pG(T^0_{\de}))}
	+ \frac{C}{\vabs{\log \de}}
	\leq \frac{C}{\vabs{\log \de}}.
\end{align*}
Therefore, $\de$ has to satisfy 
\begin{align*}
	\frac{A}{\de^4}\sqrt{\frac8{21}}\paren{1+ \wh{g_1}(\de)}
	= \pi n ,
\end{align*}
where $\wh{g_1}(\de)=\OO(\de^2 \vabs{\log \de}).$
Then, there exists $N_0>0$ and a sequence $\{{\de}_n\}_{n\geq N_0} \subset (0,\de_0)$
satisfying the statement of the Theorem and that
\begin{align*}
	{\de}_n 
	=
	\sqrt[8]{\frac8{21}}\sqrt[4]{\frac{A}{n\pi}}
	\paren{1+ 
		\OO \paren{\frac1{\log n}}},
	\qquad \text{for } n \geq N_0.
\end{align*}

%Then, following the proof of Lemma \ref{lemma:reconnexionsPartial}, we obtain the statement of the proposition.
%
%Recall that the symmetry axis in Poincar\'e coordinates is given by $\claus{\la=0, \eta=\xi}$.
%%	
%Since the equilibrium point $L_3$ is at the symmetry axis, applying the translation $\phi_{\equi}$ conserves it.
%%
%Applying the corresponding changes of variables, one can see that $H^{\linear}$ is reversible with respect to the involution
%\begin{equation}\label{def:involutionLinear}
%	{\Phi}_{\linear}(x_{\la},x_L,x_{\eta},x_{\xi}) = (-\la,L,\xi,\eta).
%\end{equation}
\end{proof}

% We devote this section to prove Propositions~\ref{proposition:formaNormalTot} and~\ref{proposition:formaNormalConsequence}.

\subsection{A quantitative Moser normal form}\label{app:proofNF}
To prove Proposition~\ref{proposition:formaNormalTot}, we first  introduce a series of affine changes of coordinates in order to put the Hamiltonian $H(\la,\La,x,y;\de)$ in~\eqref{def:hamiltonianScaling} in the form considered in~\cite{JBL16} (see \eqref{def:hamiltonianGlobalChange} below).

\begin{lemma}\label{lemma:globalChanges}
Fix $\cttTheoScalingDomainA, \cttTheoScalingDomainB>0$.
There exists $\de_0, \rhoGlobal>0$ and a family of  affine transformations
\begin{align*}
	\wh{\phi}_{\de}: B(\rhoGlobal)=\claus{\mathbf{z} \in \reals^4 : \vabs{\mathbf{z} }<\rhoGlobal} &\to \UReals(\cttTheoScalingDomainA,\cttTheoScalingDomainB) 
	\\
	(\wqU,\wpU,\wqD,\wpD) &\mapsto 
	(\la,\La,x,y),
\end{align*}
defined for $\de \in (0,\de_0)$, with $\CC^1$-functions of $\de$ as coefficients such that  the Hamiltonian system given by $H$ (see \eqref{def:hamiltonianScaling}) in the new coordinates and after a scaling in time is Hamiltonian
with respect to the canonical form and 
\begin{equation} \label{def:hamiltonianGlobalChange}
\begin{split}
	\wh{H}(\wqU,\wpU,\wqD,\wpD;\de) =& \,
	H\paren{\wh{\phi}_{\de}(\wqU,\wpU,\wqD,\wpD);
		\de}
	- H(\Ltres(\de);\de) \\
	=& \,
	\wqU \wpU 
	+ \frac{\al(\de)}{2\de^2}\paren{\wqD^2 + \wpD^2}
	+ \wh{K}(\wqU,\wpU) \\
	&+\de \wh{H}_1(\wqU,\wpU,\wqD,\wpD;\de),
\end{split}
\end{equation}
where $\al(\de)$ is a $\CC^1$-function in $\de$ satisfying
$
\al(\de) = \sqrt{\frac8{21}} + \OO(\de^4)
$
and, for $(\wqU,\wpU,\wqD,\wpD) \in B(\rhoGlobal)$, there exists a constant $C>0$ independent of $\delta$ such that
%\begin{align*}
%\wh{K}(\wqU,\wpU) &= \RRR_0 \paren{\vabs{\wqU+\wpU}^3},
%\qquad
%\wh{H}_1(\wqU,\wpU,\wqD,\wpD;\de) = \OO\paren{\vabs{(\wqU,\wpU,\wqD,\wpD)}^3}.
%\end{align*}
%
\begin{align*}
\vabss{\wh{K}(\wqU,\wpU)} \leq C \vabs{\wqU+\wpU}^3,
\qquad
\vabss{\wh{H}_1(\wqU,\wpU,\wqD,\wpD;\de)}
\leq C \vabs{(\wqU,\wpU,\wqD,\wpD)}^3.
\end{align*}
% \end{lemma}
Moreover, the change of coordinates satisfies that $\wh{\phi}_{\de}(0)=\Ltres(\de)$ and
\begin{align*}
D \wh{\phi}_0 = \begin{pmatrix}
	\frac2{\sqrt{7}} & \frac2{\sqrt{7}} & 0 & 0 \\
	-\frac1{\sqrt6} & \frac1{\sqrt6} & 0 & 0 \\
	0 & 0 & \sqrt[4]{\frac2{21}} & i\sqrt[4]{\frac2{21}} \\
	0 & 0 & \sqrt[4]{\frac2{21}} & -i\sqrt[4]{\frac2{21}},
\end{pmatrix},
\quad
D \wh{\phi}_0^{-1} = \begin{pmatrix}
	\frac{\sqrt7}4 & -\sqrt3 & 0 & 0 \\
	\frac{\sqrt7}4 & \sqrt3 & 0 & 0 \\
	0 & 0 & \sqrt[4]{\frac{21}{32}} & \sqrt[4]{\frac{21}{32}} \\
	0 & 0 & -{i}\sqrt[4]{\frac{21}{32}} & {i}\sqrt[4]{\frac{21}{32}}
\end{pmatrix}.
\end{align*}
%
% Moreover, $\wh{\phi}_{\de}$ is a symplectic scaling with respect to the form ${d\wqU\wedge d\wpU + d\wqD\wedge\wpD}$.
%
% Then, the Hamiltonian system given by $H$ (see \eqref{def:hamiltonianScaling}) in the new coordinates and after a scaling in time is Hamiltonian
% with respect to
% \begin{equation} \label{def:hamiltonianGlobalChange}
% \begin{split}
% 	\wh{H}(\wqU,\wpU,\wqD,\wpD;\de) =& \,
% 	H\paren{\wh{\phi}_{\de}(\wqU,\wpU,\wqD,\wpD);
% 		\de}
% 	- H(\Ltres(\de);\de) \\
% 	=& \,
% 	\wqU \wpU 
% 	+ \frac{\al(\de)}{2\de^2}\paren{\wqD^2 + \wpD^2}
% 	+ \wh{K}(\wqU,\wpU) \\
% 	&+\de \wh{H}_1(\wqU,\wpU,\wqD,\wpD;\de),
% \end{split}
% \end{equation}
% where $\al(\de)$ is a $\CC^1$-function in $\de$ satisfying
% $
% \al(\de) = \sqrt{\frac8{21}} + \OO(\de^4)
% $
% and, for $(\wqU,\wpU,\wqD,\wpD) \in B(\rhoGlobal)$, there exists a constant $C>0$ independent of $\delta$ such that
% %\begin{align*}
% %\wh{K}(\wqU,\wpU) &= \RRR_0 \paren{\vabs{\wqU+\wpU}^3},
% %\qquad
% %\wh{H}_1(\wqU,\wpU,\wqD,\wpD;\de) = \OO\paren{\vabs{(\wqU,\wpU,\wqD,\wpD)}^3}.
% %\end{align*}
% %
% \begin{align*}
% \vabss{\wh{K}(\wqU,\wpU)} \leq C \vabs{\wqU+\wpU}^3,
% \qquad
% \vabss{\wh{H}_1(\wqU,\wpU,\wqD,\wpD;\de)}
% \leq C \vabs{(\wqU,\wpU,\wqD,\wpD)}^3.
% \end{align*}
\end{lemma}

\begin{proof}
The proof of this lemma relies on the approach and techniques of \cite{JBL16}. 
For technical reasons and to be consistent with  \cite{JBL16}, we consider  the Poincar\'e Hamiltonian $H^{\Poi}(\la,L,\eta,\xi;\mu)$ introduced in \eqref{def:hamiltonianPoincare} instead of the scaled version $H$ defined in \eqref{def:hamiltonianScaling}.
Let us denote the point $L_3$ in Poincar\'e coordinates $(\la,L,\eta,\xi)$ as 
$
L_3^{\Poi} = (\phi^{\Poi})^{-1}(L_3).
$
Therefore, $L_3^{\Poi}$ is a saddle-center equilibrium point of the system given by $H^{\Poi}$ and, by~\eqref{eq:L3Poincare}, it satisfies that
\[
\la=0, \qquad (L,\eta,\xi) = (1,0,0) + \OO(\mu)= (1,0,0) + \OO(\de^4).
\]
%In addition, since $H^{\Poi}$ is reversible with respect to the involution ${\Phi}_{\Poi}(\la,L,\eta,\xi) = (-\la,L,\xi,\eta)$ (see \eqref{def:involutionPoincare}),
%%
%the phase space has a symmetry axis given by
%\begin{align}\label{def:symmetryAxisPoincare}
%	\SSS^{\Poi}=\claus{\la=0, \, \eta=\xi}
%	\qquad \text{and} \qquad L_3^{\Poi} \in \SSS^{\Poi}. 
%\end{align}
%
We perform several changes of coordinates.
\begin{enumerate}
%\paragraph{Translation of the equilibrium point:}
\item {Translation of the equilibrium point}. Let   $\phi^{\equi}:(\la,\wt{L},\wt{\eta},\wt{\xi}) \to (\la,L,\eta,\xi)$ be the translation such that $\phi^{\equi}(0)=L_3^{\Poi}$.
Then, the Hamiltonian system associated to $H^{\Poi}$ in the new coordinates defines a Hamiltonian system with respect to the symplectic form ${d\la\wedge d\wt{L}} + i d\wt{\eta}\wedge d\wt{\xi}$ and the Hamiltonian
\begin{align*} 
	H^{\equi}
	=
	H^{\Poi} \circ \phi^{\equi} - H^{\Poi}(L_3^{\Poi};\mu).
\end{align*}
Denoting 
 $\wt{\mathbf{z}}=({\la},\wt{L},\wt{\eta},\wt{\xi})$, $H^{\equi}(\wt{\mathbf{z}};\mu)$ can be written as
 \[
 H^{\equi}(\wt{\mathbf{z}};\mu) = H_0^{\equi}(\wt{\mathbf{z}})
 + R_2^{\equi}(\wt{\mathbf{z}};\mu)
 + R_3^{\equi}(\wt{\mathbf{z}};\mu),
 \\
 \]
 with
\begin{equation}\label{def:hamiltonianEquiSplit}
\begin{aligned}
	H_0^{\equi}(\wt{\mathbf{z}}) &= 
	\frac12
	D^2 H^{\Poi}(L_3^{\Poi};0) [\wt{\mathbf{z}},\wt{\mathbf{z}}]
	=
	-\frac32 \wt{L}^2 + \wt{\eta}\wt{\xi},
	\\ 
	R_2^{\equi}(\wt{\mathbf{z}};\mu) &=  
	\frac12
	D^2 H^{\Poi}(L_3^{\Poi};\mu) [\wt{\mathbf{z}},\wt{\mathbf{z}}] 
	-H_0^{\equi}(\wt{\mathbf{z}}) 
	= \OO(\mu \vabs{\wt{\mathbf{z}}}^2), \\
	R_3^{\equi}(\wt{\mathbf{z}};\mu) &= 
	(H^{\Poi}\circ \phi^{\equi})(\wt{\mathbf{z}};\mu)
	- H_0^{\equi}(\wt{\mathbf{z}})  - R_2^{\equi}(\wt{\mathbf{z}};\mu)
	- H^{\Poi}(L_3^{\Poi};\mu)
	\\
	&= \OO(\wt{L}^3) + \OO(\mu \vabs{\wt{\mathbf{z}}}^3),
\end{aligned}		
\end{equation}
where we have used that $\wt{L}=L-1+\OO(\mu)$, 
$\wt{\eta}=\eta+\OO(\mu)$ and $\wt{\xi}=\xi+\OO(\mu)$.
Notice that as a result, for $\mu > 0$, $\wt{\mathbf{z}}=0$ is a saddle-center point of the system given by the Hamiltonian  $H^{\equi}(\wt{\mathbf{z}};\mu)$.
%
%Lastly, by~\eqref{def:symmetryAxisPoincare}, one has that $H^{\equi}$ is reversible with respect to the symmetry axis
%\begin{align}\label{def:symmetryAxisEqui}
%	\SSS^{\equi} = \claus{{\la}=0, \, \wh{\eta}=\wh{\xi}\,}.
%\end{align}

%\paragraph{Reduction terms order 2:}

\item {Reduction of the terms of order 2}. Following the strategy of the proof of \cite[Theorem 1.3]{JBL16} in our setting, %
for $\mu \geq 0$, there exists a family $\phi^{\red}_{\mu}:\mathbf{x}=(x_{\la},x_L,x_{\eta},x_{\xi})\mapsto\wt{\mathbf{z}}=({\la},\wt{L},\wt{\eta},\wt{\xi})$ of real-analytic linear diffeomorphisms satisfying that 
	$D\phi^{\red}_{0}(0) = \mathbf{Id}$ and that
	\begin{align*}
		H^{\red}(\mathbf{x};\mu) =
		(H^{\equi} \circ \phi^{\red}_{\mu})(\mathbf{x};\mu) = H^{\equi}_{0}(\mathbf{x})
		+ R^{\red}_{2}(\mathbf{x};\mu) 
		+ R^{\red}_{3}(\mathbf{x};\mu),
	\end{align*}
	where $R^{\red}_{2}(\mathbf{x};\mu)$ is a real polynomial of degree $2$ in $\mathbf{x}$ with $\CC^1$-functions of $\mu$ as coefficients and 
	\begin{align*}
		R^{\red}_{2}(\mathbf{x};\mu) &= \OO(\mu\vabs{\mathbf{x}}^2),
		&
		\claus{H_0^{\equi} \circ \mathbf{J},R_2^{\red}}&=0,
		&
		R_3^{\red}(\mathbf{x};\mu)&= \OO(\vabs{\mathbf{x}}^3),
		%=\OO(z_{L}^3)+\OO(\mu \vabs{z}^3),
	\end{align*}
	where $\mathbf{J}$ is the matrix associated to the symplectic form $dx_{\la} \wedge d x_L + i dx_{\eta}\wedge dx_{\xi}$.
%\end{lemma}

The fact that $\claus{H_0^{\equi} \circ \mathbf{J},R_2^{\red}}=0$
and that $R^{\red}_2$ is a homogeneous polynomial of degree $2$ and $\OO(\mu|\mathbf{x}|^2)$ imply that there exist $\CC^1$-functions $\sigma_1(\mu), \sigma_2(\mu) = \OO(1)$ such that
\begin{align*}
	R^{\red}_2(x_{\la},x_L,x_{\eta},x_{\xi};\mu)
	=
	\mu \sigma_1 (\mu) \frac{x_{\la}^2}2
	+ 
	\mu \sigma_2(\mu) x_{\eta} x_{\xi}.
\end{align*}
Since $\phi^{\red}_{\mu}$ is linear and taking into account that  $D\phi^{\red}_{0}(0) = \mathbf{Id}$ and the definition of the potential $V(\la)$ in \eqref{def:potentialV}, one has that
\begin{align}\label{proof:defc1}
	\sigma_1(0) = \frac1{\mu} \partial^2_{\la} H^{\Poi}(L_3^{\Poi};\mu) \Big|_{\mu=0}
	=
	\partial^2_{\la} H_1^{\Poi}(0,1,0,0;0)
	= V''(0)  
	=\frac78.
\end{align}
Therefore, by \eqref{def:hamiltonianEquiSplit}, one has that
\begin{align*}
	H^{\red}(\mathbf{x};\mu) = 
	-\frac32 x_L^2 +  \mu \sigma_1 (\mu) \frac{x_{\la}^2}2 + (1+\mu\sigma_2(\mu))x_{\eta}x_{\xi} 
	+ R_3^{\red}(\mathbf{x};\mu).
\end{align*}
In addition, since the terms of order $3$ and higher of $H^{\equi}$ are of the form $\OO(\wt{L}^3)+ \OO(\mu \vabs{\wt{\mathbf{x}}}^3)$ (see \eqref{def:hamiltonianEquiSplit}), one has that
\begin{align*}
	R_3^{\red}(\mathbf{x};\mu) = \OO(x_L^3) + \OO(\mu \vabs{\mathbf{x}}^3).
\end{align*}
%Lastly, by \eqref{def:symmetryAxisEqui}, there exist $\CC^1$ functions $\gamma^{\equi}_1,\gamma^{\equi}_2:[0,\de_0] \to \complexs^4$ such that $H^{\red}$ is reversible with respect to the symmetry axis
%\begin{equation}\label{def:symmetryAxisRed}
%	\SSS^{\red} = \claus{z_{\la} = \de^4 \sprod{\gamma^{\equi}_1(\de)}{z}, \,
%	z_{\eta}-z_{\xi} = \de^4 
%	\sprod{\gamma^{\equi}_2(\de)}{z}}.
%\end{equation}

%\paragraph{Scaling:}
\item {Symplectic scaling}. We rename the  parameter $\de=\mu^{\frac14}$ (see \eqref{def:delta})
and, similarly to \eqref{def:changeScaling}, we consider ${\phi}^{\scaB}:\mathbf{y}=(y_{\la},y_{L},y_{\eta},y_{\xi}) \mapsto \mathbf{x}=(x_{\la},x_L,x_{\eta},x_{\xi})$ such that
\begin{align*}
	x_{\la} = \frac1{\sqrt{\sigma_1(\de^4)}} \, y_{\la},
	\quad
	x_L = \frac{\de^2}{\sqrt3} y_L,
	\quad
	x_{\eta} = \frac{\de}{\sqrt[4]{3 \sigma_1(\de^4)}}  y_{\eta} ,
	\quad
	x_{\xi} = \frac{\de}{\sqrt[4]{3 \sigma_1(\de^4)}} y_{\xi},
\end{align*}
and a scaling in time by a factor of $\de^2\sqrt{3 \sigma_1(\mu)}$.
The Hamiltonian system of $H^{\red}$ expressed in these coordinates
defines a system  associated with the form ${d y_{\la}\wedge d y_L} + {i d y_{\eta}\wedge d y_{\xi}}$ and the Hamiltonian
\begin{equation} \label{def:hamiltonianScaProof}
	\begin{split}
		{H}^{\scaB}(y;\de) = \frac12 \paren{y_{\la}^2 - y_L^2} + \al(\de)\frac{y_{\eta} y_{\xi} }{\de^2} + K^{\scaB}(y_{\la}) +
		\de {H}^{\scaB}_1(\mathbf{y};\de),
	\end{split}
\end{equation}
where
\begin{align*}
	\al(\de) 
	=& \,
	\frac{1+\de^4 \sigma_2(\de^4)}{\sqrt{3 \sigma_1(\de^4)}}
	= \sqrt{\frac8{21}}+\OO(\de^4),
	\\
	K^{\scaB}(y_{\la})
	=& \,
	\frac1{\de^4}
	R_3^{\red}\paren{\frac{y_{\la}}{\sqrt{\sigma_1(0)}},0,0,0;0} = \OO(y_{\la}^3),
	\\
	\de {H}^{\scaB}_1(\mathbf{y};\de) 
	=& \,
	\frac1{\de^4} R_3^{\red}\paren{\phi^{\scaB}(\mathbf{y});\de^4}
	- K^{\scaB}(y_{\la})
	= \OO(\de \vabs{\mathbf{y}}^3),
\end{align*}
where we have used Cauchy estimates to bound $D R_3^{\red}$.
%
%Then, by \eqref{def:symmetryAxisRed}, there exist $\CC^1$ functions $\gamma^{\sca}_1,\gamma^{\sca}_2:[0,\de_0] \to \complexs^4$ such that $H^{\red}$ is reversible with respect to the symmetry axis
%\begin{equation}\label{def:symmetryAxisSca}
%	\SSS^{\sca} = \claus{y_{\la} = \de^4 \sprod{\gamma^{\sca}_1(\de)}{y}, \,
%		y_{\eta}-y_{\xi} = \de^3 
%		\sprod{\gamma^{\equi}_2(\de)}{y}}.
%\end{equation}

%\paragraph{Linearization:}
\item {Diagonalization}. Consider the symplectic change of coordinates
$\phi^{\mathrm{diag}}:(\wqU,\wpU,\wqD,\wpD)\mapsto  \mathbf{y}=(y_{\la},y_{L},y_{\eta},y_{\xi})$ defined by
\begin{align*}
	\begin{pmatrix}
		y_{\la} \\
		y_{L}
	\end{pmatrix}
	=
	\frac1{\sqrt2} \begin{pmatrix}
		1 & 1 \\
		-1 & 1 \\
	\end{pmatrix}
	\begin{pmatrix}
		\wqU \\
		\wpU
	\end{pmatrix},
	\qquad
	\begin{pmatrix}
		y_{\eta} \\
		y_{\xi}
	\end{pmatrix}
	=
	\frac1{\sqrt2} \begin{pmatrix}
		1 & i \\
		1 & -i \\
	\end{pmatrix}
	\begin{pmatrix}
		\wqD \\
		\wpD
	\end{pmatrix}.
\end{align*}
Then, the Hamiltonian system associated to \eqref{def:hamiltonianScaProof} expressed in these coordinates
defines a Hamiltonian system with respect to the form $d \wqU\wedge d \wpU + d \wqD\wedge d\wpD$
and the Hamiltonian
\begin{equation} \label{def:hamiltonianLinear}
	\begin{split}
		\wh{H}(\wqU,\wpU,\wqD,\wpD;\de) =& 
		\wqU \wpU
		+ \frac{\al(\de)}{2\de^2}\paren{\wqD^2 + \wpD^2} 
		+ \wh{K}(\wqU,\wpU) \\
		&+\de \wh{H}_1(\wqU,\wpU,\wqD,\wpD;\de),
	\end{split}
\end{equation}
where
\begin{align*}
	\wh{K}(\wqU,\wpU) &= 
	K^{\scaB}\paren{\frac{\wqU+\wpU}{\sqrt2}} 
	= \OO\paren{\vabs{\wqU+\wpU}^3},
	\qquad
	\wh{H}_1 = H^{\scaB}_1 \circ \phi^{\mathrm{diag}}.
\end{align*} 
%
%Then, by \eqref{def:symmetryAxisSca}, there exist $\CC^1$ functions $\gamma^{\linear}_1,\gamma^{\linear}_2:[0,\de_0] \to \reals^4$ such that $H^{\red}$ is reversible with respect to the symmetry axis
%\begin{equation}\label{def:symmetryAxisLin}
%	\SSS^{\linear} = \claus{x_{\la} + x_L = \de^4 \sprod{\gamma^{\linear}_1(\de)}{x}, \,
%		x_{\xi} = \de^3 
%		\sprod{\gamma^{\linear}_2(\de)}{x}}.
%\end{equation}
\end{enumerate}
\end{proof}

Next proposition provides a normal form  in a neighborhood of the saddle-center equilibrium point. 
It is a direct consequence of \cite[Proposition C.1]{JBL16}.
In order to use this result, we introduce the artificial parameter $\nu>0$ and rewrite $\wh{H}$ in \eqref{def:hamiltonianGlobalChange} as
\begin{equation} \label{def:hamiltonianGlobalChangeParameter}
	\begin{split}
		\wh{H}(\wqU,\wpU,\wqD,\wpD;\de,\nu) =&
		\wqU \wpU 
		+ \frac{\al(\de)}{2\de^2}\paren{\wqD^2 + \wpD^2} 
		+ K(\wqU,\wpU) \\
		&+\nu \wh{H}_1(\wqU,\wpU,\wqD,\wpD;\de).
	\end{split}
\end{equation}
Note that we are interested in the case $\nu=\de$.
% but, in order to apply \cite{JBL16}, we  use this artificial parameter.

\begin{proposition}\label{proposition:normalForm}
	There exist $\de_0, \rhoLocal>0$ and a family of analytic canonical changes of coordinates,  defined for $\nu\in[0,\de_0)$ and $\de\in(0,\de_0)$,
	\begin{align*}
		\wh{\FF}_{\de,\nu} = \paren{\phiA_{1,\nu}, \psi_{1,\nu}, \phiA_{2,\nu}, \psi_{2,\nu}}: 
		B(\rhoLocal)  &\to 
		B(\rhoGlobal) \subset \reals^4
		\\
		(\qU,\pU,\qD,\pD) &\mapsto 
		(\wqU,\wpU,\wqD,\wpD),
	\end{align*}
 such that the Hamiltonian $\wh{H}$ in \eqref{def:hamiltonianGlobalChangeParameter} in the new coordinates reads
	\begin{align*}
		\HH(\qU,\pU,\qD,\pD;\de,\nu)
		&=
		\wh{H}\paren{\wh{\FF}_{\de,\nu}(\qU,\pU,\qD,\pD);\de,\nu}
		\\
		&= 
		\qU\pU 
		+ \frac{\al(\de)}{2 \de^2}\paren{\qD^2+\pD^2}
		+ \RRR (\qD \pD, \qD^2+\pD^2)
	\end{align*}
where $\RRR$ satisfies that, for $(\qU,\pU,\qD,\pD) \in B(\rhoNormalForm)$,
\begin{align*}
	\vabs{\RRR(\qU \pU, \qD^2+\pD^2;\de)} \leq C \vabss{(\qU \pU, \qD^2+\pD^2)}^2,
\end{align*}
for some $C>0$ independent of $\de$ and $\nu$.

In addition, for all $(\qU,\pU,\qD,\pD) \in B(\rhoLocal)$ 
and all
$(\wqU,\wpU,\wqD,\wpD)\in B(\rhoGlobal)$, the individual components of the change of coordinates satisfy
\begin{enumerate}[label=(\arabic*)]
	\item $\vabs{\phiA_{1,\nu}(\qU,\pU,\qD,\pD)-\qU}
	\leq C \claus{\vabs{(\qU,\pU)}^2+
		\nu \vabs{(\qU,\pU,\qD,\pD)}^2}$,
	\\
	$\vabss{\phiA^{-1}_{1,\nu}(\wqU,\wpU,\wqD,\wpD)-\wqU}
	\leq C \claus{\vabss{(\wqU,\wpU)}^2 + \nu\vabs{(\wqU,\wpU,\wqD,\wpD)}^2}$,
	\item $\vabs{\psi_{1,\nu}(\qU,\pU,\qD,\pD)-\pU}
	\leq C \claus{\vabs{(\qU,\pU)}^2+
		\nu \vabs{(\qU,\pU,\qD,\pD)}^2}$,
	\\
	$\vabss{\psi^{-1}_{1,\nu}(\wqU,\wpU,\wqD,\wpD)-\wpU}
	\leq C \claus{\vabss{(\wqU,\wpU)}^2 + \nu\vabs{(\wqU,\wpU,\wqD,\wpD)}^2}$,
	\item $\vabss{\phiA_{2,\nu}(\qU,\pU,\qD,\pD)-\qD} \leq 
	C \nu \vabs{(\qU,\pU,\qD,\pD)}^2$,
	\\
	$\vabss{\phiA^{-1}_{2,\nu}(\wqU,\wpU,\wqD,\wpD)-\wqD} \leq C\nu\vabs{(\wqU,\wpU,\wqD,\wpD)}^2$,
	\item $\vabss{\psi_{2,\nu}(\qU,\pU,\qD,\pD)-\pD} \leq 
	C \nu \vabs{(\qU,\pU,\qD,\pD)}^2$,
	\\
	$\vabss{\psi^{-1}_{2,\nu}(\wqU,\wpU,\wqD,\wpD)-\wpD} \leq C\nu\vabs{(\wqU,\wpU,\wqD,\wpD)}^2$.
\end{enumerate}	

% \color{red}{O b\'e aquesta opci\'o:
% In addition, for all $(\qU,\pU,\qD,\pD) \in B(\rhoLocal)$ 
% and all
% $(\wqU,\wpU,\wqD,\wpD)\in B(\rhoGlobal)$, the change of coordinates satisfy
% \begin{enumerate}[label=(\arabic*)]
% 	\item 
% 	$\vabss{\wh{\FF}_{\de,\nu}(\qU,\pU,\qD,\pD)-
% 		\wh{\FF}_{\de,0}(\qU,\pU,\qD,\pD)}\leq C\nu \vabs{(\qU,\pU,\qD,\pD)}^2$,
% 	\item $\vabss{\wh{\FF}^{-1}_{\de,\nu}(\wqU,\wpU,\wqD,\wpD)-\wh{\FF}_{\de,0}^{-1}(\wqU,\wpU,\wqD,\wpD)} \leq C\nu\vabs{(\wqU,\wpU,\wqD,\wpD)}^2$.
% \end{enumerate}	
% Moreover, the individual components satisfy that $\phiA_{1,0}=\phiA_{1,0}(\qU,\pU)$, $\psi_{1,0}=\psi_{1,0}(\qU,\pU)$, $\phiA_{2,0} = \psi_{2,0} = \Id$ and 
% \begin{enumerate}[label=(\arabic*)]
% 	\setcounter{enumi}{2}
% 	\item $\vabs{\phiA_{1,0}(\qU,\pU)-\qU}, \vabs{\psi_{1,0}(\qU,\pU)-\pU}
% 	\leq C \vabs{(\qU,\pU)}^2$,
% 	\item $\vabss{\phiA^{-1}_{1,0}(\wqU,\wpU)-\wqU}, \vabss{\psi_{1,0}^{-1}(\wqU,\wpU)-\wpU}
% 	\leq C \vabss{(\wqU,\wpU)}^2$.
% \end{enumerate}	
% }
\end{proposition}

Proposition~\ref{proposition:formaNormalTot} is a direct consequence of Lemma~\ref{lemma:globalChanges} and Proposition~\ref{proposition:normalForm}.

\subsection{The invariant manifolds in normal form variables}\label{app:translationrecon}

To prove Proposition~\ref{proposition:formaNormalConsequence}, we translate the results in Theorem~\ref{mainTheoremDist} (Statement 1) and the axis of symmetry $\SSS$ (Statement 2) into the set of coordinates $(\qU,\pU,\qD,\pD)$ given in Proposition~\ref{proposition:formaNormalTot}.
Recall that in the proof of Proposition~\ref{proposition:formaNormalTot}, we have used the ``intermediate'' system of coordinates $(\wqU,\wpU,\wqD,\wpD)$.
We translate first the results via the change of coordinates $\wh{\phi}_{\de}:(\wqU,\wpU,\wqD,\wpD) \to (\la,\La,x,y)$, given by Lemma~\ref{lemma:globalChanges}.
Then, we apply the second change of coordinates $\wh{\FF}_{\de,\nu}:(\qU,\pU,\qD,\pD)\to(\wqU,\wpU,\wqD,\wpD)$ with $\nu=\de$, given by Proposition~\ref{proposition:normalForm}.

\paragraph{Statement 1:} 

Let $\la_* \in [\la_1,\la_2]\subset(0,\la_0)$ to be chosen later and consider the section
$
{\Sigma}(\la_*) = \claus{\la=\la_*, \La>0}.
$
Let $\mathbf{z}^{\unstable}_{\de}(\la_*)$ and $\mathbf{z}^{\stable}_{\de}(\la_*)$ be the first intersections of the invariant manifolds $\WW^{\unstable,+}(\Ltres)$ and $\WW^{\stable,+}(\Ltres)$ with the section $\Sigma(\la_*)$, respectively.

Let us recall that, by Proposition~\ref{proposition:existenceFixedPoint}, the critical point $\Ltres(\de)$ in $(\la,\La,x,y)$ coordinates is of the form $\Ltres(\de)=\paren{0,
	\de^2 \LtresLa(\de),
	\de^3 \Ltresx(\de),
	\de^3 \Ltresy(\de)}^T$,
with $\LtresLa,\Ltresx,\Ltresy = \OO(1)$.
Then, applying the change of coordinates $\wh{\phi}_{\de}$ given in Lemma~\ref{lemma:globalChanges}, there exist $\CC^1$ functions $\gamma_1,\gamma_2:(0,\de_0) \to \reals^4$ satisfying $\gamma_1,\gamma_2= \OO(1)$ such that
\begin{equation}\label{proof:sectionGlobalCoordinates}
\begin{aligned}
\wh{\Sigma}(\la_*,\de) = 
\wh{\phi}_{\de} \paren{{\Sigma}(\la_*)} =
\Big\{ &\wqU+\wpU + \de \sprod{\gamma_1(\de)}{(\wqU,\wpU,\wqD,\wpD)}=  {\textstyle\frac{\sqrt7}2}\la_*, 
\\
 &\wpU-\wqU + \de \sprod{\gamma_2(\de)}{(\wqU,\wpU,\wqD,\wpD)}
 + \de^2 \sqrt6 \LtresLa(\de)>0
\Big\}.
\end{aligned}
\end{equation}
Notice that
% , at first order, it corresponds to 
$
\wh{\Sigma}(\la_*,0) = \{\wqU+\wpU = {\textstyle\frac{\sqrt7}2}\la_*, \, \wpU>\wqU\}.
$
Moreover, we denote
\begin{align*}%\label{proof:changeGlobal}
(\wqUu,\wpUu,\wqDu,\wpDu) 
&=
\wh{\phi}_{\de}^{-1}\paren{\mathbf{z}^{\unstable}_{\de}(\la_*)
	} 
\in \wh{\Sigma}(\la_*,\de), 
\\
(\wqUs,\wpUs,\wqDs,\wpDs) 
&=
\wh{\phi}_{\de}^{-1}\paren{
	\mathbf{z}^{\stable}_{\de}(\la_*)}
\in \wh{\Sigma}(\la_*,\de).
\end{align*}	
Since $\wh{\phi}_{\de}$ is an affine transformation, by Theorem~\ref{mainTheoremDist} and Lemma~\ref{lemma:globalChanges}, one has that
\begin{equation}\label{proof:diferenciaGlobalCoordinates}
\begin{aligned}
\wqUu-\wqUs
&=
\boxClaus{\paren{\frac{\sqrt7}4,-\sqrt3,0,0} + \OO(\de)} \cdot \paren{\mathbf{z}^{\unstable}_{\de}(\la_*)-\mathbf{z}^{\stable}_{\de}(\la_*)}
=
\OO \paren{\de^{\frac43} e^{-\frac{A}{\de^2}}},
\\
\wpUu-\wpUs
&=
\boxClaus{\paren{\frac{\sqrt7}4,\sqrt3,0,0} + \OO(\de)} \cdot \paren{\mathbf{z}^{\unstable}_{\de}(\la_*)-\mathbf{z}^{\stable}_{\de}(\la_*)}
=
\OO \paren{\de^{\frac43} e^{-\frac{A}{\de^2}}},
\\
\wqDu-\wqDs 
&= 
\boxClaus{\paren{0,0,\sqrt[4]{\frac{21}{32}},\sqrt[4]{\frac{21}{32}}} + \OO(\de)} \cdot \paren{\mathbf{z}^{\unstable}_{\de}(\la_*)-\mathbf{z}^{\stable}_{\de}(\la_*)}
\\
&=
\sqrt[3]{4}\sqrt[4]{\frac{21}8}
\de^{\frac{1}{3}}  e^{-\frac{A}{\de^2}} 
\boxClaus{\Re{\CInn}
	+ \OO\paren{\frac{1}{\vabs{\log \de}}}},
\\
\wpDu-\wpDs 
&= 
\boxClaus{\paren{0,0,-i\sqrt[4]{\frac{21}{32}},i \sqrt[4]{\frac{21}{32}}} + \OO(\de)} \cdot \paren{\mathbf{z}^{\unstable}_{\de}(\la_*)-\mathbf{z}^{\stable}_{\de}(\la_*)}
\\
&= -
\sqrt[3]{4}\sqrt[4]{\frac{21}8}
\de^{\frac{1}{3}}  e^{-\frac{A}{\de^2}} 
\boxClaus{\Im{\CInn}
	+ \OO\paren{\frac{1}{\vabs{\log \de}}}}.
\end{aligned}
\end{equation}
%where $A>0$ and
%$\CInn \in \complexs$ are the constants described in Theorem \ref{TheoremA}.

Next, we consider the change of coordinates $\wh{\FF}_{\de,\nu}$ with $\nu=\de$ given in Proposition~\ref{proposition:normalForm}.
Let us denote
\begin{align}\label{proof:changeGlobalToLocal}
	(\qUu,\pUu,\qDu,\pDu) 
	=
	\wh{\FF}_{\de,\de}^{-1}
	(\wqUu,\wpUu,\wqDu,\wpDu),
	\quad
	(\qUs,\pUs,\qDs,\pDs) 
	=
	\wh{\FF}^{-1}_{\de,\de}
	(\wqUs,\wpUs,\wqDs,\wpDs) .
\end{align}
Since the local stable manifold is given by $\claus{\qU=\qD=\pD=0}$ (see \eqref{proof:localCoordinatesSolution}), one has that $\qUs=\qDs=\pDs=0$ and we call $\varrho=\pUs$, (see Figure~\ref{fig:intersectionsTechnical}).
Taking into account that $\wh{\FF}_{\de,\de}(0,\varrho,0,0)=(\wqUs,\wpUs,\wqDs,\wpDs)\in \wh{\Sigma}(\la_*,\de)$ for $\la \in [\la_1,\la_2]$, by \eqref{proof:sectionGlobalCoordinates}, the value $\la_*$ must satisfy
\[
 \la^*=\sqrt{\frac74}\boxClaus{\wqUs+\wpUs + \de^4 \sprod{\g_1(\de)}{\wh{\FF}_{\de,\de}(0,\varrho,0,0)}}.
\]
% \begin{align*}
% 	\la_* = \sqrt{\frac74}\boxClaus{\phiA_{1,\de}(0,\varrho,0,0)
% 		+ \psi_{1,\de}(0,\varrho,0,0) + \de^4 \sprod{\g_1(\de)}{\wh{\FF}_{\de,\de}(0,\varrho,0,0)}},
% \end{align*}
Then, using the notation $\wh{\FF}_{\de,\de} = \paren{\phiA_{1,\de}, \psi_{1,\de}, \phiA_{2,\de}, \psi_{2,\de}}$,
 by Proposition \ref{proposition:normalForm}, one has that for $\varrho\in (0,\varrho_0)$ and $\de>0$ small enough, 
\[
\la_* = \sqrt{\frac74}\varrho \paren{1+ \OO(\varrho,\de)}.
\]
Then, it is clear that taking, for instance, $\varrho=\varrho_0/2$, the corresponding $\la_*$ belongs to a closed interval in $(0,\la_0)$ independent of $\delta$.
% and there exists $[\rhoLimitA,\rhoLimitB]\subset(0,\rhoLocal)$ such that, for $\varrho \in [\rhoLimitA,\rhoLimitB]$ and $\de>0$ small enough, one has that ${\la_* \in [\la_1,\la_2]}$.

Next, we consider the difference between $(\qUu,\pUu,\qDu,\pDu)$
and  $(\qUs,\pUs,\qDs,\pDs)=(0,\varrho,0,0)$. By  \eqref{proof:changeGlobalToLocal}, one has that
\begin{align*}
	\pUu - \varrho &= 
	\psi^{-1}_{1,\de}(\wqUu,\wpUu,\wqDu,\wpDu) 
	-\psi^{-1}_{1,\de}(\wqUs,\wpUs,\wqDs,\wpDs),
	\\
	\qDu  &= 
	\phiA^{-1}_{2,\de}(\wqUu,\wpUu,\wqDu,\wpDu) 
	-\phiA^{-1}_{2,\de}(\wqUs,\wpUs,\wqDs,\wpDs),
	\\
	\pDu &= 
	\psi^{-1}_{2,\de}(\wqUu,\wpUu,\wqDu,\wpDu) 
	-\psi^{-1}_{2,\de}(\wqUs,\wpUs,\wqDs,\wpDs).
\end{align*}
For $\pUu-\varrho$, by the mean value theorem, Proposition \ref{proposition:normalForm} and \eqref{proof:diferenciaGlobalCoordinates}, one obtains
\begin{align*}
\vabs{\pUu - \varrho} 
&\leq 
C \vabs{\wqUu-\wqUs} + 
C \vabs{\wpUu-\wpUs} + 
C \de \vabs{\wqDu-\wqDs} + 
C \de \vabs{\wpDu-\wpDs} \leq
C \de^{\frac43}e^{-\frac{A}{\de^2}}.
\end{align*}
Analogously, for $\qDu$, one has that
\begin{align*}
\vabs{\qDu-(\wqDu-\wqDs)} \leq C\de \vabs{
	\paren{\wqUu-\wqUs,\wpUu-\wpUs,\wqDu-\wqDs,\wpDu-\wpDs}}
\leq
C \de^{\frac43}e^{-\frac{A}{\de^2}}
\end{align*}
and, by \eqref{proof:diferenciaGlobalCoordinates}, one obtains the expression of the statement for $\qDu$.
An analogous estimate holds for $\pDu$.
Lastly, by the expression of Hamiltonian $\HH$ in Proposition~\ref{proposition:normalForm}, one sees that $\HH(\qUu,\pUu,\qDu,\pDu)=\HH(0,0,0,0)=0$ and obtains the expression for $\qUu$.

\paragraph{Statement 2:}
Let us consider the symmetry axis $\SSS=\claus{\la=0,x=y}$ given in \eqref{def:symmetryAxisScaling}.
Notice that, by Proposition~\ref{proposition:existenceFixedPoint}, one has that $\wh{\phi}_{\de}(0)=\Ltres(\de) \in \SSS$.
Then, applying the affine transformation $\wh{\phi}_{\de}$ given in Lemma~\ref{lemma:globalChanges}, there exist functions $\g_3,\g_4:(0,\de_0) \to \reals$ satisfying $\g_3, \g_4 =\OO(1)$ such that
\[
\wh{\phi}_{\de} (\SSS) = \big\{
\wqU+\wpU + \de \sprod{\gamma_3(\de)}{(\wqU,\wpU,\wqD,\wpD)}=0,\,\,
\wpD + \de \sprod{\gamma_4(\de)}{(\wqU,\wpU,\wqD,\wpD)}=0\big\}.\]
%\begin{align*}
%\wh{\phi}_{\de} (\SSS) = \big\{
%\wqU+\wpU + \de \sprod{\gamma_3(\de)}{(\wqU,\wpU,\wqD,\wpD)}=&0, \\
%\wpD + \de \sprod{\gamma_4(\de)}{(\wqU,\wpU,\wqD,\wpD)}=&0\big\}.
%\end{align*}
%
Then, applying the change of coordinates $\wh{\FF}_{\de,\de}$, one has that
\begin{align*}
	\SSS_{\local} = \claus{\qU+\pU=\Psi_1(\qU,\pU,\qD,\pD;\de), \,
		\pD = \Psi_2(\qU,\pU,\qD,\pD;\de)},
\end{align*}
where
\begin{align*}
\Psi_1 =& 
\paren{\phiA_{1,\de}-\qU} +
\paren{\psi_{1,\de}-\pU} +
\de \sprod{\g_3(\de)}{\FF_{\de,\de}},
\\
\Psi_2 =& 
\paren{\psi_{2,\de}-\pD} +
\de \sprod{\g_4(\de)}{\FF_{\de,\de}}.
\end{align*}
Then, Proposition~\ref{proposition:normalForm} implies that for $(\qU,\pU,\qD,\pD)\in B(\rhoLocal)$ and $\de>0$ small enough, 
\begin{align*}
	\vabs{\Psi_1(\qU,\pU,\qD,\pD;\de)}
	&\leq
	C \de \vabs{(\qU,\pU,\qD,\pD)}
	+ C \vabs{(\qU,\pU)}^2,
	\\
	\vabs{\Psi_2(\qU,\pU,\qD,\pD;\de)}
	&\leq
	C \de \vabs{(\qU,\pU,\qD,\pD)}.
\end{align*}

\section{The invariant manifolds of the Lyapunov orbits: Proof of Theorems \ref{TheoremB} and \ref{TheoremC}}
\label{section:proofAB}

The goal of this section is to prove Item 1 of Theorem~\ref{TheoremB} and Theorem~\ref{TheoremC}. First we rephrase these results  referred to the Hamiltonian \eqref{def:hamiltonianScaling} (recall that $\de= \mu^{\frac{1}{4}}$, see \eqref{def:delta}).
We begin with the existence of the Lyapunov periodic orbits given by  Proposition \ref{PropositionB}. Note that the Lyapunov Center Theorem (see for instance \cite{MeyerHallOffin}) ensures the existence of a family of periodic orbits emanating from a saddle-center equilibrium point. 
In our setting, this family corresponds to perturbed orbits of the fast oscillator, centered at $\Ltres(\de)$, and therefore the existence of the periodic orbits given by  Proposition \ref{PropositionB} is just a consequence of this classical theorem. However, we need to ``reprove'' it to have estimates for the periodic orbits.

First, we introduce the following notation. For $d>0$, we denote 
\begin{align}\label{def:torusC}
	\torus = \reals/2\pi\integers, 
	\qquad
	\torusC = \claus{\tau \in \complexs/2\pi\integers \st %\Re \tau \in \torus := \reals/2\pi\integers, 
		\vabs{\Im \tau}<d}.
\end{align}

\begin{proposition}\label{proposition:periodicOrbit}
	Let $\banda, \cttTheoScalingDomainA, \cttTheoScalingDomainB>0$.
	There exist $\rhoPeriodicOrbit, \de_0>0$ such that, for $\de\in(0,\de_0)$, there exists a family of periodic orbits 
	$
	\claus{\po_{\rho}(\tau;\de)  \st  \tau \in \torusC}_{\rho \in [0,\rhoPeriodicOrbit]}, 
	$
	where $\po_{\rho}:\torusC \to \UComplexs(\cttTheoScalingDomainA,\cttTheoScalingDomainB)$ are real-analytic functions
	satisfying that 
	\[
	H({\po}_{\rho}(\tau; \de))=\frac{\rho^2}{\de^2} + H(\Ltres(\de)).
	\]
	
	Furthermore, there exist $\omegar>0$ and a constant $\cttPeriodicOrbit>0$, independent of $\rho$ and $\de$, such that the parametrization of the periodic orbit satisfies
	\[
	%	H({\po}_{\rho}(\tau; \de))=\frac{\rho^2}{\de^2},
	%	\qquad
	\dot{\tau} = \frac{\omegar}{\de^2}
	\qquad \text{with} \qquad
	\vabs{\omegar-1} \leq \cttPeriodicOrbit \de^4.
	\]
	In addition, the parametrization can be written as
	\begin{align}\label{def:poSplitting}
		{\po}_{\rho}(\tau;\de) = \Ltres(\de) + \rho\cdot\big(0,0, e^{-i\tau}, e^{i\tau}\big)^T + 
		\de\rho\cdot\big(\lapo,\Lapo,\xpo,\ypo\big)^T(\tau),
	\end{align}
	where
	$
	\vabs{\lapo(\tau)}, \vabs{\Lapo(\tau)}\leq \cttPeriodicOrbit,
	$
	and
	$
	\vabs{\xpo(\tau)},\vabs{\ypo(\tau)} \leq \cttPeriodicOrbit \de^3.
	\quad
	$
\end{proposition}
The proof of this proposition can be found in Appendix \ref{appendix:periodicOrbits}. 

%Next, we focus on the study of the intersections between the unstable and stable manifolds of the family of periodic orbits $\claus{\po_{\rho}(\cdot,\de)}_{\rho \in [0,\rhoPeriodicOrbit]}$.
%
Let us denote by $\WW^{\unstable}(\po_{\rho})$ and $\WW^{\stable}(\po_{\rho})$ the $2$-dimensional unstable and stable manifolds of the periodic orbit $\po_{\rho}(\cdot,\de)$. Analogously to the invariant manifolds of $\Ltres(\de)$, we denote each branch as $\WW^{\diamond,+}(\po_{\rho})$ and $\WW^{\diamond,-}(\po_{\rho})$  for $\diamond\in\claus{\unstable,\stable}$ (see Figure \ref{fig:perturbedInvariantManifolds2d}).
To prove Theorem \ref{TheoremB} and \ref{TheoremC}, we focus on the study of the ``$+$'' invariant manifolds.
By symmetry, there exist analogous results for the 
``$-$'' invariant manifolds.

We look for intersections between $\WW^{\unstable,+}(\po_{\rho})$ and $\WW^{\stable,+}(\po_{\rho})$ in the  $2$-dimensional section 
\begin{align}\label{def:section}
	\Si_{\rho} = \claus{\paren{\la,\La,x,y}\in \UReals(\cttTheoScalingDomainA, \cttTheoScalingDomainB) : 
	\La = \de^2\LtresLa(\de), \,
	 H(\la,\La,x,y)=\frac{\rho^2}{\de^2} + H(\Ltres(\de))},
\end{align}
where $\Ltres=(0,\de^2\LtresLa,\de^3\Ltresx,\de^3\Ltresy)^T$ as given in Proposition~\ref{proposition:existenceFixedPoint} and  $\UReals(\cttTheoScalingDomainA, \cttTheoScalingDomainB)$ is the domain introduced in  \eqref{def:dominiUReals}. Note that this definition is consistent with that of $\Sigma_0$ in  \eqref{def:section0} and  that, by Proposition~\ref{proposition:periodicOrbit}, the periodic orbit $\po_{\rho}$ belongs to the energy level $H=\frac{\rho^2}{\de^2}+H(\Ltres(\de))$ where $\Sigma_{\rho}$ is included.

In the next result, we see that the $2$-dimensional invariant manifolds $\WW^{\unstable,+}(\po_{\rho})$ and $\WW^{\stable,+}(\po_{\rho})$ intersect in the section $\Si_{\rho}$ for certain values of $\rho$
(see Figure \ref{fig:intersections}). Note that the intersection of the invariant manifolds for $\rho=0$ has been analyzed in Corollary \ref{mainTheoremDistCorollary}.
Both Item 1 of  Theorem \ref{TheoremB} and Theorem  \ref{TheoremC} are a  consequence of the following result.
%
% Its proof is postponed to Section \ref{section:proofAB}.

\begin{figure}
	\centering
	\begin{overpic}[scale=0.7]{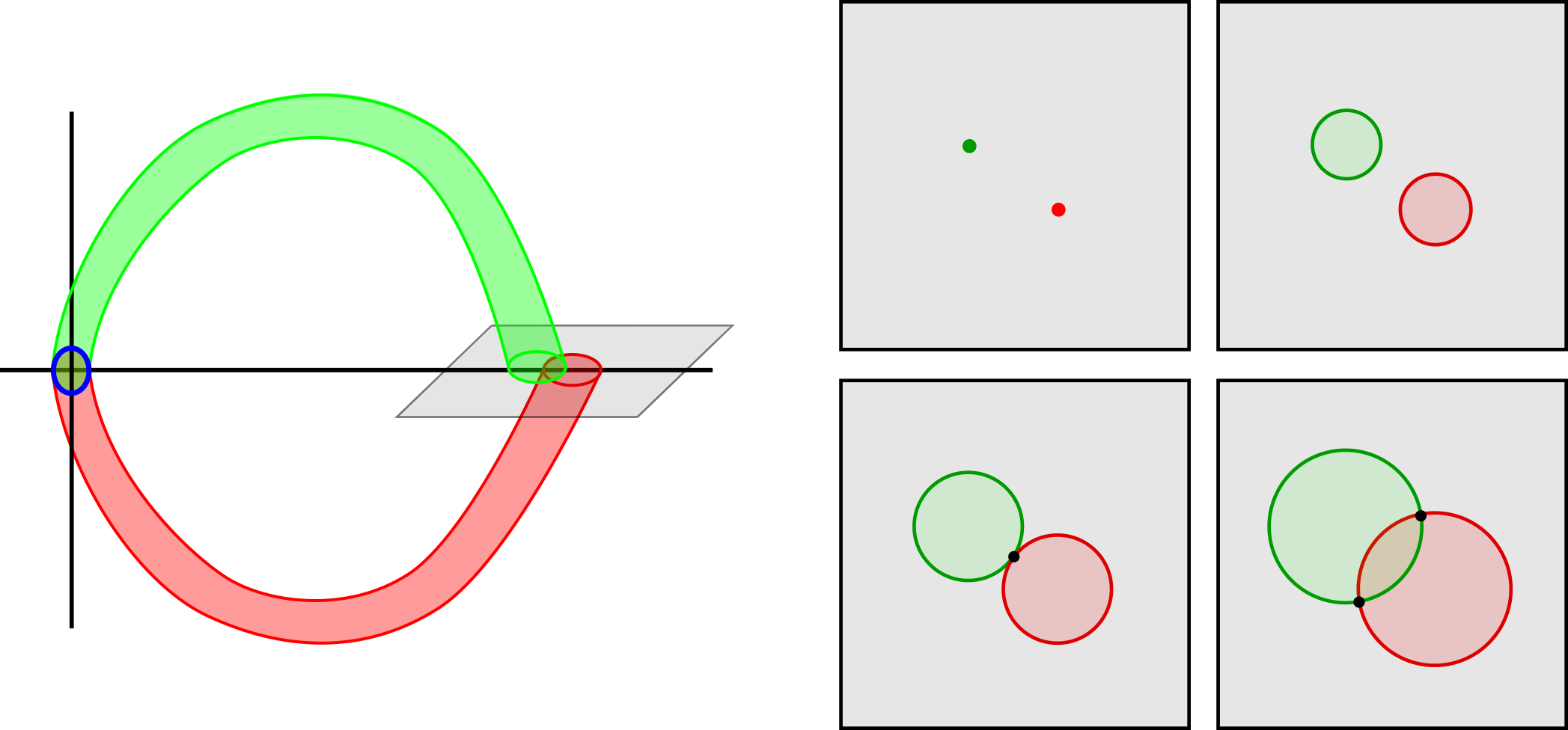}
		\put(0,20){\textcolor{blue}{$\po_{\rho}$}}
		\put(17,2){\textcolor{red}{$\WW^{\unstable,+}(\po_{\rho})$}}
		\put(17,41){\textcolor{myGreen}{$\WW^{\stable,+}(\po_{\rho})$}}
		\put(40,17){$\Sigma_{\rho}$}
		\put(38,23.8){\footnotesize \textcolor{red}{$\partial \DD^{\unstable}_{\rho}$}}
		\put(28,20.8){\footnotesize \textcolor{myGreen}{$\partial \DD^{\stable}_{\rho}$}}
		\put(46,22.5){\footnotesize $\lambda$}
		\put(4,40){\footnotesize $\Lambda$}
		\put(61,48){$\rho=0$}
		\put(67,30){\footnotesize\textcolor{red}{$(x^{\unstable},y^{\unstable})$}}
		\put(55,39){\footnotesize\textcolor{myGreen}{$(x^{\stable},y^{\stable})$}}
		\put(80,48){$\rho\in(0,\rho_{\min}(\de))$}
		\put(90,28){\footnotesize\textcolor{red}{$\partial \DD^{\unstable}_{\rho}$}}
		\put(83,41){\footnotesize\textcolor{myGreen}{$\partial \DD^{\stable}_{\rho}$}}
		\put(58,-4){$\rho=\rho_{\min}(\de)$}
		\put(67,2){\footnotesize\textcolor{red}{$\partial \DD^{\unstable}_{\rho}$}}
		\put(59,18.3){\footnotesize\textcolor{myGreen}{$\partial \DD^{\stable}_{\rho}$}}
		\put(76,-4){$\rho\in[\rho_{\min}(\de),\rho_{\max}(\de)]$}
		\put(93,1.5){\footnotesize\textcolor{red}{$\partial \DD^{\unstable}_{\rho}$}}
		\put(80,18.5){\footnotesize\textcolor{myGreen}{$\partial \DD^{\stable}_{\rho}$}}
	\end{overpic}
	\vspace{1cm}
	\caption{Intersection of the  manifolds $\WW^{\unstable,+}(\po_{\rho})$ and $\WW^{\stable,+}(\po_{\rho})$ with section $\Si_{\rho}$.
	The pictures in the right show the different possibilities given in Corollary \ref{mainTheoremDistCorollary} and Theorem \ref{mainTheoremA}.}
	\label{fig:intersections}
\end{figure}

\begin{theorem} \label{mainTheoremA}
% Assume Ansatz~\ref{ansatz}.
%
Let $\rhoPeriodicOrbit$ and $\po_{\rho}$, for $\rho \in [0,\rhoPeriodicOrbit]$, be as given in Proposition \ref{proposition:periodicOrbit}.
Then, the following is satisfied.
\begin{itemize}
\item There exists $\de_0>0$ such that, for every
$\rho \in [0,\rhoPeriodicOrbit]$ and $\de\in(0,\de_0)$, the invariant manifolds $\WW^{\unstable,+}(\po_{\rho})$ and $\WW^{\stable,+}(\po_{\rho})$ intersect the section $\Si_{\rho}$.
The first intersection is given by closed curves, which we denote by $\partial \DD^{\unstable}_{\rho}$ and $\partial \DD^{\stable}_{\rho}$.

\item Let $R>1$. There exists $\de_R>0$, satisfying $\lim_{R \to \infty} \de_R = 0$, and functions $\rho_{\min}, \rho_{\max}:(0,\de_R)\to [0,\rho_0]$ such that, for $\de\in(0,\de_R)$ and $\rho \in [\rho_{\min}(\de),\rho_{\max}(\de)]$, the curves $\partial \DD^{\unstable}_{\rho}$ and $\partial \DD^{\stable}_{\rho}$ intersect.
Moreover,
\begin{align*}
	\rho_{\min}(\de) &= \frac{\sqrt[6]{2}}2  \vabss{\CInn}\de^{\frac13} e^{-\frac{A}{\de^2}}\boxClaus{1+\OO\paren{\frac1{\vabs{\log \de}}}},
	\\
	\rho_{\max}(\de) &= \frac{\sqrt[6]{2}}2 \vabss{\CInn}\de^{\frac13} e^{-\frac{A}{\de^2}}\boxClaus{R+\OO\paren{\frac1{\vabs{\log \de}}}}.
\end{align*}

\item For $\rho \in (\rho_{\min}(\de),\rho_{\max}(\de)]$, the curves $\partial \DD^{\unstable}_{\rho}$ and $\partial \DD^{\stable}_{\rho}$ intersect transversally at least twice.
\item For $\rho = \rho_{\min}(\de)$, the curves $\partial \DD^{\unstable}_{\rho}$ and $\partial \DD^{\stable}_{\rho}$ have at least one quadratic tangency at a point $Q_0 \in \partial \DD^{\unstable}_{\rho} \cap \partial \DD^{\stable}_{\rho}$.
\item Fix $\de\in(0,\de_0)$ and let $\zeta$ be any smooth curve transverse to $\partial \DD^{\unstable}_{\rho_{\min}}$ and $\partial \DD^{\stable}_{\rho_{\min}}$ within $\Si_{\rho_{\min}}$ at $Q_0$.
Then, for $\rho$ close to $\rho_{\min}$, the local intersections of $\partial \DD^{\unstable}_{\rho}$ and $\partial \DD^{\stable}_{\rho}$ with the curve $\zeta$ cross each
other with relative non-zero velocity at $(Q_0,\rho_{\min})$.
\end{itemize}
\end{theorem}

Theorem \ref{mainTheoremA} implies in particular that, for small values of $\de$, there exist transverse intersections between some unstable and stable manifolds of Lyapunov periodic orbits of $\Ltres(\de)$.
%
% Thus, applying the Smale-Birkhoff homoclinic theorem (see \cite{Smale67, KatokHasselblatt}), it is possible to construct a Smale horseshoe map in a tubular neighborhood of the invariant manifolds $\WW^{\unstable,+}(\Ltres)$ and $\WW^{\stable,+}(\Ltres)$ of size $\OO(\de^{\frac13}e^{-\frac{A}{\de^2}})$.
%
By symmetry, an analogous result holds for $\WW^{\unstable,-}(\Ltres)$ and $\WW^{\stable,-}(\Ltres)$.
This proves Item 1 of Theorem~\ref{TheoremB}.

% \textcolor{blue}{
Moreover, the last two statements of Theorem~\ref{mainTheoremA} imply the existence of a generic unfolding of a quadratic tangency between $\WW^{\unstable,+}(\po_{\rho})$ and $\WW^{\stable,+}(\po_{\rho})$ (we follow the definition of generic unfolding given in \cite{Dua08}).
Indeed, denoting by $f_{\varrho}$ to the flow of $h$ in~\eqref{def:hamiltonianInitialNotSplit} restricted to the energy level $h=\varrho+h(L_3)$, for $\de\in(0,\de_0)$, one has that $f_{\varrho}$ unfolds generically an homoclinic quadratic tangency.
Finally, noticing that the energy level
$H(\la,\La,x,y;\de)=\frac{\rho^2}{\de^2} + H(\Ltres)$ corresponds to $h(q,p;\mu)=\sqrt{\mu}\rho^2 + h(L_3)$ (see \eqref{def:delta} and \eqref{def:changeScaling}), one proves Theorem~\ref{TheoremC}.
%
% In particular, from this result one should be able to prove the existence of a Newhouse domain (see Remark~\ref{remark:NewhouseDomain}).
% }

%\begin{proof}[Proof of Theorem \ref{TheoremB}]
%\textcolor{red}{Aix\`o probablement es podr\`a borrar tot.}
%Let us recall that 
%\begin{align*}
%H = \frac1{\de^4} h (\phi_{\Poi} \circ \phi_{\sca}), 
%\qquad
%\mu = \de^4.
%\end{align*}
%By Theorem \ref{mainTheoremB}, one has that
%$
%H(\po_{\rho}) = H(\Ltres) + \frac{\rho^2}{\de^2}.
%$
%Then
%\begin{align*}
% h(\phiA_{k}) = h(L_3) + \de^2 \rho^2.
%\end{align*}
%Therefore, since $k=h(\phiA_{k})$, if $\rho= C \de^{\frac13} e^{-\frac{A}{\de^2}} $ one has that
%\begin{align*}
%k - h(L_3) =  C^2 \mu^{\frac23} e^{-\frac{2A}{\sqrt\mu}}.
%\end{align*}
%\end{proof}

The rest of this section is devoted to prove Theorem \ref{mainTheoremA}.
First, in Section \ref{subsection:unperturbedSeparatrix}, we sum up the results concerning the unperturbed separatrix of the Hamiltonian $H_{\pend}$ in \eqref{def:HpendHosc} presented in~\cite{articleInner}.
Next, in Section \ref{subsection:existencePerturbed}, we obtain and analyze parametrizations of the unstable and stable manifolds of the Lyapunov periodic orbits given in Proposition \ref{proposition:periodicOrbit}. 
Last, in Section \ref{subsection:intersections}, we analyze the intersections between these manifolds to complete the proof of Theorem \ref{mainTheoremA}.

Throughout this section and the following ones, we denote the components of all the functions and operators by a numerical sub-index $f=(f_1,f_2,f_3,f_4)^T$, unless stated
otherwise.
In addition, we denote the canonical basis of $\complexs^4$ by $\claus{\mathbf{e_j}}_{j=1..4}$.

\subsection{The unperturbed separatrix}
\label{subsection:unperturbedSeparatrix}

Let us consider the unperturbed Hamiltonian $H_0$ as given in \eqref{def:hamiltonianScalingH0}.
Notice that the plane $\claus{x=y=0}$ is invariant for $H_0$ and the dynamics on it is described by 
\begin{align*} 
	{H}_{\pend}(\la,\La) &= -\frac{3}{2} \La^2  +
	V(\la),
	\qquad
	V(\la) =  1 - \cos \la - \frac{1}{\sqrt{2+2\cos \la}},
	% 
%	{H}_{\osc}(x,y; \de) = \frac{x y}{\de^2},
\end{align*}
(see \eqref{def:HpendHosc}).
%
%The two separatrices associated to $(0,0)$ (see Figure~\ref{fig:separatrix}) are found at energy level $H_{\pend}=-\frac12$
%%
%and its eigenvalues are $\claus{\pm \sqrt{\frac{21}8}}$.
%%
%We consider a real-analytic time parametrization of the 
%right separatrix (that is, with $\la\in (0,\pi)$),
%\begin{equation}\label{eq:separatrixParametrization}
%	\begin{split}
%		\reals &\to \torus \times \reals \\
%		t &\mapsto (\la_{\pend}(t),\La_{\pend}(t)),
%	\end{split}
%\end{equation}
%with initial condition $\la_{\pend}(0) =
%\arccos\paren{\frac12-\sqrt{2}} \in \left(\frac{2}{3}\pi,\pi\right)$ 
%and
%$\La_{\pend}(0)=0$.
%%
%Moreover, we denote
%\begin{align}\label{eq:separatrixParametrizationB}
%	\s_{\pend}(t)=(\lap(t),\Lap(t),0,0).
%\end{align}
%
The origin $(\la,\La)=(0,0)$ is a saddle  with two separatrices associated to it (see Figure~\ref{fig:separatrix}).
In \cite{articleInner}, we studied their real-analytic time-parametrizations.
The following result summarizes Theorem 2.2 and Corollary~2.4 in \cite{articleInner} and it establishes a suitable domain for these parametrizations, which we denote as
\begin{equation}\label{eq:separatrixParametrizationB}
	\s_{\pend}(u)=(\lap(u),\Lap(u),0,0)^T.
\end{equation}

\begin{figure} 
	\centering
	\begin{overpic}[scale=0.5]{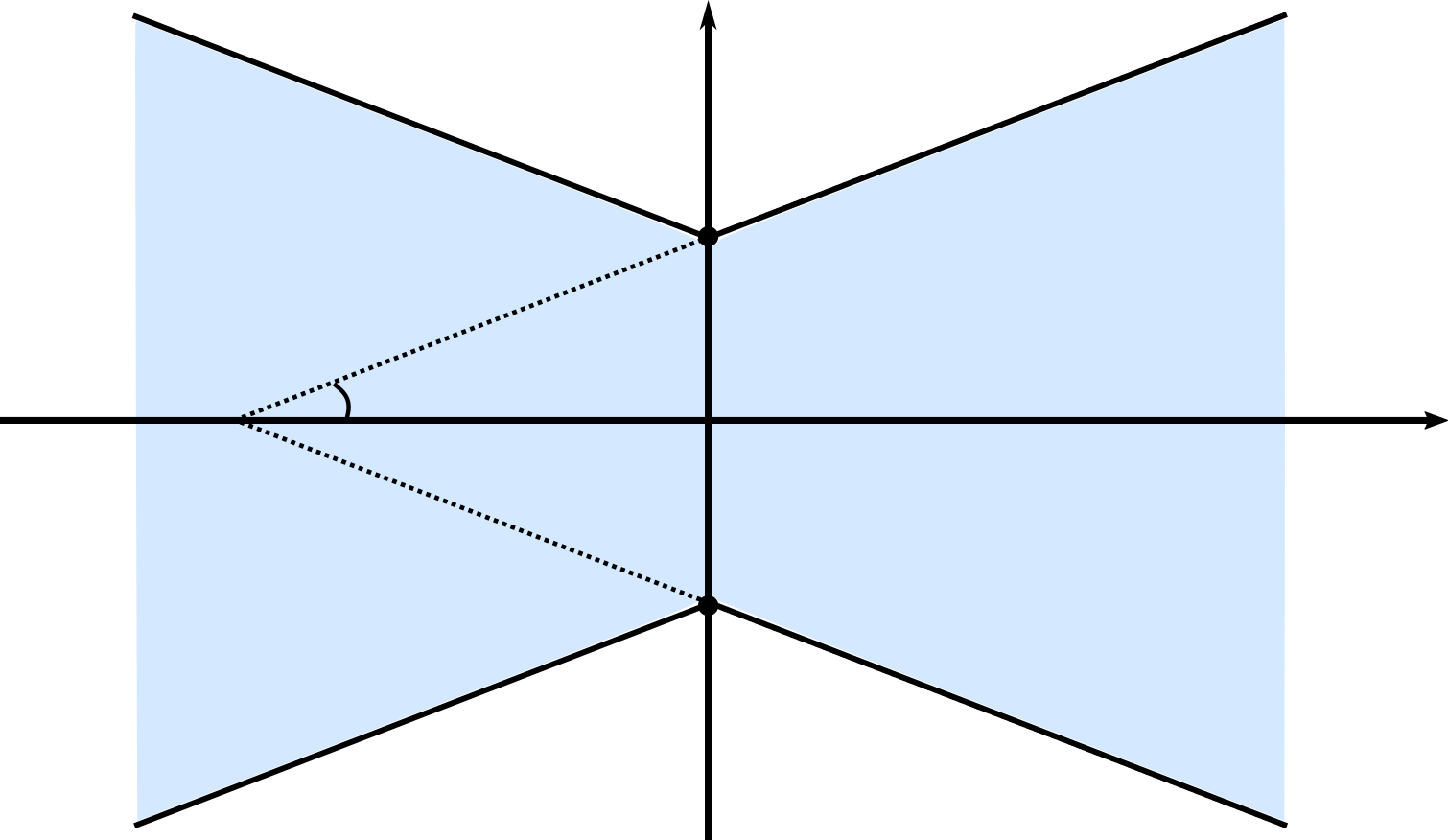}
		\put(101,27.5){\footnotesize$\Re u$}
		\put(45.5,59){\footnotesize$\Im u$}
		\put(28,30){\footnotesize $\beta$}
		\put(42,45){\footnotesize $iA$}
		%\put(40,10){{$-iA$}}
		\put(70,32){$\PiExtA$}
	\end{overpic}
	\bigskip
	\caption{Representation of the domain $\PiExtA$ in~\eqref{def:dominiBow}.}
	\label{fig:dominiBow}
\end{figure}

\begin{proposition}\label{proposition:singularities}
Let $\la_0>0$ be as given in \eqref{def:la0}.
There exists $0<\betaBow<\frac{\pi}{2}$ such that
the time-parametrization $(\lap(u),\Lap(u))$ of the right separatrix (i.e, $\lap(u) \in (0,\pi)$) of $H_{\pend}$ with $(\lap(0),\Lap(0))=(\la_0,0)$ extends analytically to
\begin{equation}\label{def:dominiBow}
	\begin{split}	
		\PiExtA =& \claus{
			u \in \complexs \st 
			\vabs{\Im u} < \tan \beta \, \Re u + A }
		\cup \\
		&\claus{
			u \in \complexs \st 
			\vabs{\Im u} < -\tan \beta \, \Re u + A},
	\end{split}
\end{equation} 
with $A>0$ as given in \eqref{def:integralA}, (see Figure \ref{fig:dominiBow}).
Moreover,
\begin{itemize}
	\item There exists $C>0$ such that, for $\vabss{\Re u}\gg 1$, $\vabs{\lap(u)}, \vabs{\Lap(u)} \leq C e^{-\sqrt{\frac{21}8}\vabs{\Re u}}$.
	\item For $u\in \ol{\PiExtA}$, $\lap(u) = \pi$  if and only if $u=\pm iA$. 
	\item For $u\in \ol{\PiExtA}$, $\Lap(u)=0$ if and only if $u=0$.
\end{itemize}
%
%Moreover, for $t\in \ol{\PiExt}$, 
%\begin{align*}
%	\bullet \,\, \lap(t) = \pi \text{ if and only if } t=\pm iA.
%	\qquad
%	\bullet \,\, \Lap(t) = 0 \text{ if and only if } t=0.
%\end{align*}
\end{proposition}

\subsection{Existence of the perturbed invariant manifolds}
\label{subsection:existencePerturbed}

We devote this section to obtain and analyze parametrizations of the $2$-dimensional branches of the manifolds $\WW^{\unstable,+}(\po_{\rho})$ and $\WW^{\stable,+}(\po_{\rho})$, where  $\claus{\po_{\rho}}_{\rho\in(0,\rhoPeriodicOrbit]}$ is the family of periodic orbits given in Proposition~\ref{proposition:periodicOrbit}.
We find these parametrizations through a Perron-like method.
In particular, following the ideas in \cite{BFGS12}, we write the perturbed manifolds as functions of  $(u,\tau)$, where $u$  parametrizes the unperturbed homoclinic orbit $\s_{\pend}(u)$ (see \eqref{eq:separatrixParametrizationB}) and $\tau$ parametrizes the Lyapunov periodic orbit $\po_{\rho}(\tau;\de)$.

Let us define the following complex domains (see Figure \ref{fig:dominiMig}),
\begin{align}\label{def:dominiMig}
	\DmigU = \claus{u \in \complexs \st
		\vabs{\Im u} < \textstyle \frac{A}2 - \tan \betaBow \, \Re u
	},
	\qquad
	\DmigS = \claus{u \in \complexs \st -u \in \DmigU}.
\end{align}
\begin{figure} 
	\centering
	\begin{overpic}[scale=0.7]{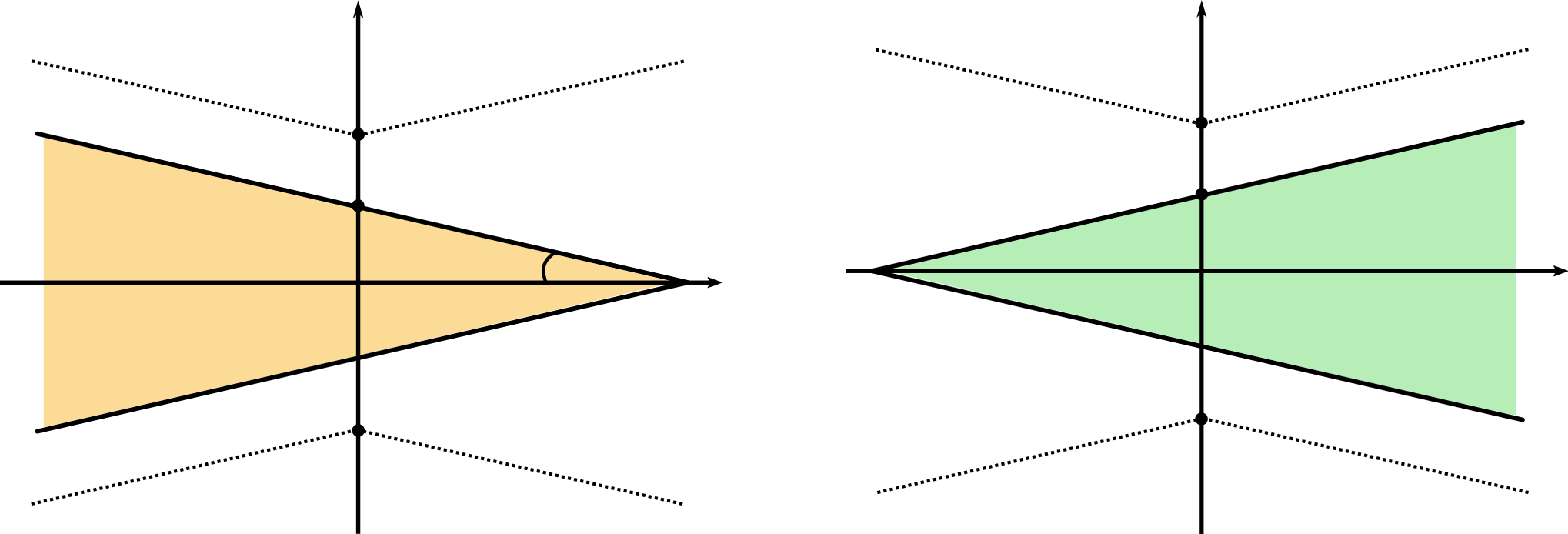}
		\put(100.5,16){\footnotesize$\Re u$}
		\put(20.5,35){\footnotesize$\Im u$}
		\put(74.5,35){\footnotesize$\Im u$}
		\put(30,16.5){\footnotesize $\betaBow$}
		\put(77.5,27.5){\footnotesize $iA$}
		\put(23.5,21.5){\footnotesize $i\frac{A}2$}
		\put(71,4){\footnotesize $-iA$}
		\put(35,23){$\PiExtA$}
		\put(10,18.5){$\DmigU$}
		\put(85,18.5){$\DmigS$}
	\end{overpic}
	\bigskip
	\caption{Representation of the domains $\DmigU$ and $\DmigS$ in~\eqref{def:dominiMig}.}
	\label{fig:dominiMig}
\end{figure}
Then, for $\diamond\in\claus{\unstable,\stable}$, we consider the parametrizations $\zdD(u,\tau)$
satisfying that
\begin{align*}
	\claus{\zdD(u,\tau) \st (u,\tau) \in \DmigD \times \torusC} \subseteq \WW^{\diamond,+}(\po_{\rho}).
\end{align*}
Notice that, for the unperturbed problem, since $\s_{\pend}(u)$ is a time-parametrization it satisfies $\dot{u}=1$. 
In addition, by Proposition \ref{proposition:periodicOrbit}, the dynamics in $\po_{\rho}(\tau;\de)$ satisfy $\dot{\tau} = \frac{\omegar}{\de^2}$.
Therefore, we impose that the dynamics on the perturbed parametrizations $\zdD$ are given by
\[
\dot{u}=1, \qquad \dot{\tau} = \frac{\omegar}{\de^2}.
\]
Hence, the parametrizations satisfy
% the system given by the Hamiltonian $H$ in \eqref{def:hamiltonianScaling}
\begin{equation}\label{eq:existence2D}
	\partial_u \zdD(u,\tau) +
	\frac{\omegar}{\de^2} \partial_{\tau} \zdD(u,\tau)
	= \begin{pmatrix}
		\mathbf{J} & 0 \\
		0 & i\mathbf{J}
	\end{pmatrix} DH(\zdD(u,\tau);\de)
	\quad
	\text{with }
	\mathbf{J} = \begin{pmatrix}
		0 & 1 \\
		-1 & 0
	\end{pmatrix}
\end{equation}
and the asymptotic conditions
\begin{align}\label{eq:assymptoticConditionExistence2d}
	\lim_{\Re u \to -\infty} \zuD(u,\tau) =
	\lim_{\Re u \to +\infty} \zsD(u,\tau) = \, \po_{\rho}(\tau;\de),
	\quad \text{for all} \quad
	\tau \in \torusC.
\end{align}

To prove their existence and behavior,
% certain properties of these parametrizations.
% $\zdD$, for $\diamond \in \claus{\unstable,\stable}$.
%
% In order to do so, 
we consider the decomposition
\begin{align}\label{def:parametrizationZusD}
	\zdD(u,\tau) =  \po_{\rho}(\tau;\de) + \s_{\pend}(u) + \zdDper(u,\tau),
%	\qquad
%	(u,\tau) \in \DmigD \times \torusC,
\end{align}
with $\s_{\pend}$ as given in \eqref{eq:separatrixParametrizationB}.
%
% Notice that, since $\zdD \subset \PiExtB$, $\s_{\pend}$ is well defined in $\zdD$.
%
The proof of the following result is deferred to Section \ref{section:proofExistence}.

\begin{proposition}\label{proposition:existencia2D}
	Fix $\banda>0$ and $\diamond\in\claus{\unstable,\stable}$. 
	Let $\rhoPeriodicOrbit>0$ be the constant given in Proposition~\ref{proposition:periodicOrbit}.  
	There exist $\cttTheoScalingDomainA, \cttTheoScalingDomainB, \de_0, \cttExistenciaD>0$	
	such that, for $\rho \in [0,\rhoPeriodicOrbit]$ and $\de \in (0,\de_0)$, equation \eqref{eq:existence2D} together with the condition \eqref{eq:assymptoticConditionExistence2d}  has a unique real-analytic solution 
	$\zdD:\DmigD \times \torusC \to \UComplexs(\cttTheoScalingDomainA,\cttTheoScalingDomainB)$ that can be decomposed as in \eqref{def:parametrizationZusD} and satisfies 
	\[
	\langle\zdDper(0,\tau),\mathbf{e_2}\rangle= 0,
	\qquad 	\text{for all }
	\tau \in \torusC.
	\]
	In addition, for $\nu=\frac12\sqrt{\frac{21}8}$, 
	\begin{align*}
		\vabs{\zdDper(u,\tau)}
		\leq \cttExistenciaD \de e^{-\nu \vabs{\Re u}},
		\qquad
		\text{for }
		(u,\tau) \in \DmigD \times \torusC.
	\end{align*}
%with $j=1,..,4$ and $\nu=\frac12\sqrt{\frac{21}8}$
%for $\diamond=\unstable$ and 
%	$\nu=-\frac12\sqrt{\frac{21}8}$ for $\diamond=\stable$.
\end{proposition}

%Proposition \ref{proposition:existencia2D} can be adapted to prove the existence of parametrizations of the $1$-dimensional invariant manifolds $W^{\unstable,+}(\Ltres)$ and $W^{\stable,+}(\Ltres)$.

Notice that, by Proposition \ref{proposition:periodicOrbit}, when $\rho=0$, $\po_{0}(\tau;\de)\equiv\Ltres(\de)$ is a fixed point and that, $\WW^{\unstable,+}(\Ltres)$ and $\WW^{\stable,+}(\Ltres)$ are $1$-dimensional invariant manifolds.
Then, for $\diamond \in \claus{\unstable,\stable}$, Proposition \ref{proposition:existencia2D} provides parametrizations $\zdUper$ independent of $\tau$ satisfying
\[
\claus{\zdU(u) \st u \in \DmigD} \subseteq \WW^{\diamond,+}(\Ltres),
\]
that can be decomposed as
\begin{align}\label{def:parametrizationZusU}
	\zdU(u) =  \Ltres + \s_{\pend}(u) + \zdUper(u).
\end{align}

\begin{corollary}\label{corollary:existencia1D}
Let $\diamond\in\claus{\unstable,\stable}$.  
There exist $\cttTheoScalingDomainA, \cttTheoScalingDomainB, \de_0,\cttExistenciaU>0$ such that, for $\de \in (0,\de_0)$ and $\rho=0$, equation \eqref{eq:existence2D}  together with the conditions \eqref{eq:assymptoticConditionExistence2d} has a unique real-analytic solution 
$\zdU:\DmigD \to \UComplexs(\cttTheoScalingDomainA,\cttTheoScalingDomainB)$ that can be decomposed as in \eqref{def:parametrizationZusU} and satisfies  
$
\langle\zdUper(0),\mathbf{e_2}\rangle=0.
$
% and the corresponding asymptotic condition in \eqref{eq:assymptoticConditionExistence2d}.
%
In addition, for $\nu=\frac12\sqrt{\frac{21}8}$,
\begin{align*}
	\vabs{\zdUper(u)}
	\leq \cttExistenciaD \de e^{-\nu \vabs{\Re u}},
	\qquad
	\text{for }
	u \in \DmigD.
\end{align*}
%with $j=1,..,4$ and $\nu=\frac12\sqrt{\frac{21}8}$.
\end{corollary}

Finally, for $\diamond \in \claus{\unstable,\stable}$, we can measure how accurately the $1$-dimensional manifolds $\WW^{\diamond,+}(\Ltres)$ approximate the $2$-dimensional manifolds $\WW^{\diamond,+}(\po_{\rho})$.

\begin{proposition}\label{proposition:comparison1d2d}
	Fix $\banda>0$ and $\diamond\in \claus{\unstable,\stable}$.
	Let $\rhoPeriodicOrbit$ be the constant in Proposition~\ref{proposition:periodicOrbit} and $\zdDper$ and $\zdUper$ be the parametrizations given in Proposition~\ref{proposition:existencia2D} and Corollary~\ref{corollary:existencia1D}, respectively.
	Then, there exists $\de_0>0$ and a constant $\cttExistenciaUD>0$ such that, for $\rho \in [0,\rhoPeriodicOrbit]$ and $\de \in (0,\de_0)$,
	\[
	\vabs{\zdDper(u,\tau)-\zdUper(u)} \leq \cttExistenciaUD \de \rho,
	\qquad \text{for} \quad
	(u,\tau) \in \DmigD \times \torusC.
	\]
\end{proposition}

The proof of this proposition is postponed to Section~\ref{section:proofComparison}.

\subsection{End of the proof of Theorem \ref{mainTheoremA}}
\label{subsection:intersections}

To prove the first statement of Theorem~\ref{mainTheoremA}, in the next lemma we study the intersections between the section $\Sigma_{\rho}$ (see~\eqref{def:section}) and the unstable and stable manifolds of $\po_{\rho}$ parametrized by $\zuD$ and $\zsD$, respectively.

\begin{lemma}\label{lemma:intersectionSectionManifolds}
Fix $\diamond\in\claus{\unstable,\stable}$.
Let $\rhoPeriodicOrbit$ and $\po_{\rho}$ be as given in Proposition~\ref{proposition:periodicOrbit}, $\zdD$ be the parametrization given in
\eqref{def:parametrizationZusD} and Proposition~\ref{proposition:existencia2D}
and $\Sigma_{\rho}$ the section given in \eqref{def:section}.
Then, there exists $\de_0>0$ and a real-analytic function $\UU^{\diamond}_{\rho}:\torusC \to \DmigD$ such that, for $\rho \in [0,\rhoPeriodicOrbit]$ and $\de\in(0,\de_0)$, 
\begin{align*}
\zdD(\UU^{\diamond}_{\rho}(\tau),\tau) \in \Sigma_{\rho},
\qquad
\text{for } \tau \in \torus= \reals/2\pi\integers.
\end{align*}
Moreover, there exists $C>0$ independent of $\rho$ and $\de$ such that, for $\tau \in \torusC$,
\begin{align*}
\UU^{\diamond}_{0}\equiv 0, \qquad	
\vabss{\UU^{\diamond}_{\rho}(\tau)}\leq C\de\rho.
\end{align*}
\end{lemma}

\begin{proof}
	Since the parametrization $\zdD$ is real-analytic (see Remark \ref{remark:realanalytic}), one has that
	\[
	\zdD(u,\tau) \in \UReals(\cttTheoScalingDomainA,\cttTheoScalingDomainB)
	\qquad \text{for  }
	(u,\tau) \in (\DmigD \cap \reals)\times\torus.
	\]
%
%	Next, we study the intersection the unstable and stable manifolds parameterized by $\zuD(u,\tau)$ and $\zsD(u,\tau)$ in \eqref{def:parametrizationZusD} with the section $\Si_{\rho} = \claus{\La = \de^2 \LtresLa, H = \frac{\rho^2}{\de^2}+H(\Ltres)}$ given in \eqref{def:section}.
	%
In addition, by Propositions~\ref{proposition:periodicOrbit} and~\ref{proposition:existencia2D}, one has that
\begin{align*}
	H(\zdD(u,\tau)) 
	&=
	H(\po_{\rho}(\tau;\de))
	= \frac{\rho^2}{\de^2} + H(\Ltres(\de)),
	\qquad \text{for  }
	(u,\tau) \in (\DmigD \cap \reals)\times\torus.
\end{align*}
Therefore, it is only necessary to find a function $\UU^{\diamond}_{\rho}(\tau)$ satisfying that $\sprod{\zdD(\UU^{\diamond}_{\rho}(\tau),\tau)}{\mathbf{e}_2} = \de^2 \LtresLa(\de)$ for all $\tau \in \torus$.
Then, by the decomposition~\eqref{def:parametrizationZusD} of $\zdD$ and Proposition~\ref{proposition:periodicOrbit},
\begin{align*} 
\de \rho \Lapo(\tau)
+ \Lap(\UU^{\diamond}_{\rho}(\tau)) + \sprod{\zdDper(\UU^{\diamond}_{\rho}(\tau),\tau)}{\mathbf{e}_2}
= 0.
\end{align*}
By Proposition~\ref{proposition:singularities}, one has that $\Lap(u)=\Ladp(0) u+ \OO(u^2)$
with $\Ladp(0)=-V'(\la_0)\neq 0$.
Then, $\UU^{\diamond}_{\rho}$ is a solution of the fixed point equation given by the operator
\begin{align*}
F[\UU^{\diamond}_{\rho}](\tau) 
=
- \frac1{\Ladp(0)}\boxClaus{\de \rho \Lapo(\tau) + 
\paren{\Lap(\UU^{\diamond}_{\rho}(\tau)) - \Ladp(0)\UU^{\diamond}_{\rho}}
+ \sprod{\zdDper(\UU^{\diamond}_{\rho}(\tau),\tau)}{\mathbf{e}_2}
}.
\end{align*}
Notice that, by Propositions \ref{proposition:periodicOrbit} and~\ref{proposition:existencia2D},
\begin{align*}
\vabs{F[0](\tau)} = \de\rho\frac{\vabs{\Lapo(\tau)}}{\vabss{\Ladp(0)}}
\leq C\de\rho.
\end{align*}
Moreover, for real-analytic functions $\UU,\VV:\torusC \to \DmigD$ satisfying that $\vabs{\UU},\vabs{\VV} \leq C\de\rho$ and applying the mean value theorem and Proposition~\ref{proposition:existencia2D}, one can see that the operator $F$ satisfies that, if $\de$ small enough,
\begin{align*}
\vabs{F[\UU]-F[\VV]} 
&\leq
C
\vabss{\UU^2-\VV^2} 
+ 
\vabs{\UU-\VV}
\sup_{s \in [0,1]}\vabs{
\sprod{\partial_u {\zdDper}(s\UU+(1-s)\VV,\tau)}{\mathbf{e}_2}} \\
&\leq
C\de\rho \vabs{\UU-\VV}
\leq \frac12 \vabs{\UU-\VV},
\end{align*}
where we have used that $\sprod{\zdDper(0,\tau)}{\mathbf{e_2}}=0$.
Hence, $F$ has a fixed point $\UU^{\diamond}_{\rho}$ satisfying that $\vabs{\UU^{\diamond}_{\rho}(\tau)}\leq C\de\rho$, for $\tau \in \torusC$. 
\end{proof}

The first statement of Theorem \ref{mainTheoremA} is a direct consequence of Lemma~\ref{lemma:intersectionSectionManifolds}.
We denote by $\partial \DD^{\unstable}_{\rho}$ and $\partial \DD^{\stable}_{\rho}$ the first intersection of the manifolds $\WW^{\unstable,+}(\po_{\rho})$ and $\WW^{\stable,+}(\po_{\rho})$ with the  section $\Sigma_{\rho}$, respectively, that can be parametrized as
\begin{equation}\label{def:curvesIntersections}
\begin{aligned}
\partial \DD^{\diamond}_{\rho} &= \claus{
\zdD(\UU^{\diamond}_{\rho}(\tau^{\diamond}),\tau^{\diamond}) \st \tau^{\diamond} \in \torus } \subset \Sigma_{\rho} \cap \WW^{\diamond,+}(\po_{\rho}),\quad \diamond\in\claus{\unstable,\stable}.
% \\
% \partial \DD^{\stable}_{\rho} &= \claus{
% 	\zsD(\UU^{\stable}_{\rho}(\tau),\tau) \st \tau \in \torus }
% \subset \Sigma_{\rho} \cap \WW^{\stable,+}(\po_{\rho}).
\end{aligned}
\end{equation}
In particular, the first intersection of the manifolds $\WW^{\unstable,+}(\Ltres)$ and $\WW^{\stable,+}(\Ltres)$ with the section $\Sigma_{0}$ corresponds to the points $\partial \DD^{\unstable}_{0} = \claus{\zuU(0)}$ and 
$\partial \DD^{\stable}_{0} = \claus{\zsU(0)}$.

To prove the rest of the statements, we study the difference between the parametrizations  of the curves considered in \eqref{def:curvesIntersections}.
Since $\Sigma_{\rho} \subset \UReals(\cttTheoScalingDomainA, \cttTheoScalingDomainB)$ (see \eqref{def:section}), 
% by the definition of the domain in \eqref{def:dominiUReals}, 
for $\tauU,\tauS \in \torus$ one has that
\begin{align*}
{\langle\zuD(\UU^{\unstable}_{\rho}(\tauU),\tauU)-\zsD(\UU^{\stable}_{\rho}(\tauS),\tauS),\mathbf{e_2}\rangle} &= 0,
\\
{\langle\zuD(\UU^{\unstable}_{\rho}(\tauU),\tauU)-\zsD(\UU^{\stable}_{\rho}(\tauS),\tauS),\mathbf{e_4}\rangle}
&= \conj{{\langle\zuD(\UU^{\unstable}_{\rho}(\tauU),\tauU)-\zsD(\UU^{\stable}_{\rho}(\tau),\tauS),\mathbf{e_3}\rangle}},
\end{align*}
and  ${\langle\zuD(\UU^{\unstable}_{\rho}(\tauU),\tauU)-\zsD(\UU^{\stable}_{\rho}(\tauS),\tauS),\mathbf{e_1}\rangle}$ can be recovered by the conservation of energy $H=\frac{\rho^2}{\de^2}+H(\Ltres)$.
Therefore, to analyze the intersections between $\partial \DD^{\stable}_{\rho}$ and $\partial \DD^{\unstable}_{\rho}$, it suffices to study the zeroes of the complex function
\begin{align*}
	\Delta(\tauU,\tauS,\rho,\delta) := \sprod{\zuD(\UU^{\unstable}_{\rho}(\tauU),\tauU)-\zsD(\UU^{\stable}_{\rho}(\tauS),\tauS)}{\mathbf{e_4}}.
\end{align*}

Let us recall that, by Proposition \ref{proposition:comparison1d2d}, the difference $\Delta(\tauU,\tauS)$ is given at first order, by the difference ${\zuU-\zsU}$.
Therefore, using the decompositions \eqref{def:parametrizationZusD} and \eqref{def:parametrizationZusU}, for $\diamond\in\claus{\unstable,\stable}$,  we write
\begin{align*}
	\zdD(\UU^{\diamond}(\tau),\tau) = \po_{\rho}(\tau) +\separatrix(\UU^{\diamond}(\tau))  + \zdUper(\UU^{\diamond}(\tau)) + \paren{\zdDper(\UU^{\diamond}(\tau),\tau)-\zdUper(\UU^{\diamond}(\tau))}, 
\end{align*}
where $\zdDper$ and $\zdUper$ are given in Proposition~\ref{proposition:existencia2D} and Corollary \ref{corollary:existencia1D}, respectively.
Recall that, $\s_{\pend}=(\lap,\Lap,0,0)$ (see \eqref{eq:separatrixParametrizationB}) and, by Proposition~\ref{proposition:periodicOrbit},
$
{\po}_{\rho} = \Ltres + \rho(0,0,e^{i\tau}, e^{-i\tau}) + 
\de\rho(\lapo,\Lapo,\xpo,\ypo).
$
Therefore, for $\de$ small enough, we look for $(\tauU,\tauS,\rho)$ such that
\begin{equation}\label{proof:eqImplicita1}
\Delta(\tauU,\tauS,\rho,\delta)=0,
\end{equation}
where
\[
 \Delta(\tauU,\tauS,\rho,\delta)= \,
\rho(e^{-i\tauU}-e^{-i\tauS}) +
 \sqrt[6]{2} \,
\de^{\frac{1}{3}}  e^{-\frac{A}{\de^2}}  \vabs{\CInn} e^{i\tht}  + M(\de) + R(\tauU,\tauS,\de,\rho),
\]
with $\theta =\arg\langle\zuUper(0)-\zsUper(0),\mathbf{e_4}\rangle$ and
\begin{align*}
	M(\de) =& \,
	{\langle\zuUper(0)-\zsUper(0),\mathbf{e_4}\rangle}
	-\sqrt[6]{2} \, \de^{\frac13}e^{-\frac{A}{\de^2}} \vabs{\CInn} e^{i\tht}, \\
%\end{align*}
%and
%\begin{align*}
	R(\tauU,\tauS,\de,\rho) =& \,
	{\langle\zuDper(\UU^{\unstable}_{\rho}(\tauU),\tauU)-\zuUper(\UU^{\unstable}_{\rho}(\tauU)),\mathbf{e_4}\rangle} - {\langle\zsDper(\UU^{\stable}_{\rho}(\tauS),\tauS)-\zsUper(\UU^{\stable}_{\rho}(\tauS)),\mathbf{e_4}\rangle} \\
	&+{\langle\zuUper(\UU^{\unstable}_{\rho}(\tauU))-\zuUper(0),\mathbf{e_4}\rangle}
	-{\langle\zsUper(\UU^{\stable}_{\rho}(\tauS))-\zsUper(0),\mathbf{e_4}\rangle} \\
	&+\de\rho(\ypo(\tauU)-\ypo(\tauS)).
\end{align*}
%
%Notice that, by Theorem \ref{mainTheoremDist},
%\begin{align*}
%	{\langle\zuU(0)-\zsU(0),\mathbf{e_3}\rangle} &= 
%	\sqrt[6]{2} \,
%	\de^{\frac{1}{3}}  e^{-\frac{A}{\de^2}} 
%	\boxClaus{ \ol{\CInn}
%		+ \OO\paren{\frac{1}{\vabs{\log \de}}}},
%	\\
%	{\langle\zuU(0)-\zsU(0),\mathbf{e_4}\rangle} 
%	&= 
%	\sqrt[6]{2} \,
%	\de^{\frac{1}{3}}  e^{-\frac{A}{\de^2}} 
%	\boxClaus{{\CInn}
%		+ \OO\paren{\frac{1}{\vabs{\log \de}}}},
%\end{align*}
%where $A>0$ is given in \eqref{def:integralA}
%and  $\CInn \in \complexs$ is the Stokes constant described in Ansatz~\ref{ansatz}.
%
Notice that, by Corollary \ref{mainTheoremDistCorollary}, Propositions
\ref{proposition:periodicOrbit} and \ref{proposition:comparison1d2d} and
Lemma~\ref{lemma:intersectionSectionManifolds},
\begin{align*}
 M(\de)=\OO\paren{\frac{\de^{\frac13}e^{-\frac{A}{\de^2}}}{\vabs{\log \de}}},
	\qquad
	R(\tauU,\tauS,\de,\rho) = \OO(\de\rho).
\end{align*}
Since, by Theorem~\ref{TheoremA}, $\CInn\neq 0$, we can consider the auxiliary parameter $r\in (0,r_0]$,
\begin{align}\label{proof:changerhor}
	r = \frac{2 e^{\frac{A}{\de^2}}}{\sqrt[6]{2} \,\de^{\frac13}\vabss{\CInn}}\rho, 
	\qquad \text{and} \qquad
	r_0 = \frac{2 e^{\frac{A}{\de^2}}}{\sqrt[6]{2} \,\de^{\frac13}\vabss{\CInn}}\rhoPeriodicOrbit.
\end{align}
Then, equation \eqref{proof:eqImplicita1} is equivalent to
\begin{equation}\label{proof:eqImplicita2}
\begin{split}
r(e^{-i(\tauU+\tht)}-e^{-i(\tauS+\tht)}) + 2 
+  g(\tauU,\tauS,r,\de)=0,	
\end{split}
\end{equation}
where
\begin{align*}
g(\tauU,\tauS,r,\de) &= \frac{2 e^{\frac{A}{\de^2}} e^{-i\tht} }{ \sqrt[6]{2} \,\de^{\frac13}\vabss{\CInn}}
\paren{
M(\de)
+ 
R\paren{\tauU,\tauS,\de,
\frac{\sqrt[6]{2}}2\de^{\frac13}\vabss{\CInn} e^{-\frac{A}{\de^2}}r}
} \\
&= 
\OO\paren{\frac1{\vabs{\log \de}}}
+ \OO\paren{\de r}.
\end{align*}
By introducing $G=(G_1,G_2):\torus^2\times [0,r_0] \times [0,\de_0) \to \reals^2$, as
\begin{equation}\label{def:functionG}
\begin{aligned}
G_1(\tauU,\tauS,r,\de) &= r \paren{\cos(\tauU+\tht)-\cos(\tauS+\tht)}
	+ 2 
	+ \Re  g(\tauU,\tauS,r,\de)
	, \\
G_2(\tauU,\tauS,r,\de) &= r\paren{\sin(\tauU+\tht)-\sin(\tauS+\tht)}+ 
\Im  g(\tauU,\tauS,r,\de),
\end{aligned}
\end{equation}
equation \eqref{proof:eqImplicita2} is  equivalent to $G(\tauU,\tauS,r,\de)=(0,0)$. 

Next result characterizes the solutions of this equation (see also Figure \ref{fig:implicitaInterseccions}). Note that it would be reasonable to look for the zeros of $G$ for a fixed $r$. This would give the intersections between the invariant manifolds of a given periodic orbit. Instead, we parameterize the zeros writing $(\tau^\unstable,r)$ as functions of $\tau^\stable$. This makes the application of the implicit function theorem easier and allows us to analyze at the same time transverse intersections and quadratic tangencies.

\begin{lemma}\label{lemma:intersectionsA}
Fix $\gamma \in (0,\frac{\pi}2)$ and consider $I_{\gamma}=[-\tht-{\gamma},-\tht+{\gamma}]$.
There exists $\de_{\gamma}$ satisfying
$\lim_{\gamma \to \pi/2} \de_{\gamma} = 0$ and functions	
%\begin{align*}
$(\tauU_{*},r_{*}): I_{\gamma} \times (0,\de_{\gamma}) \to \torus \times \reals$,
%\end{align*}
such that
	$
	G(\tauU_{*}(\tauS,\de),\tauS,r_{*}(\tauS,\de),\de) = (0,0)
	$ 
	and
	\begin{equation*}\label{proof:functionrStar}
		\begin{aligned}
			\tau_{*}^{\unstable}(\tauS,\de) &= 
			\pi -\tauS -2\tht + \OO\paren{\frac1{\vabs{\log \de}}},
			\\
			r_{*}(\tauS,\de)&=\frac1{\cos(\tauS+\tht)} + \OO\paren{\frac1{\vabs{\log \de}}}.
		\end{aligned}
	\end{equation*}
\end{lemma}

\begin{proof}
For $r\geq 1$ and $\delta=0$, the equation $G(\tauU,\tauS,r,0)=(0,0)$ has a family of solutions given by
\begin{align*}
% \label{def:Salpha}
S_{\al} = \claus{(\tauU,\tauS,r,0)=\paren{
		\pi-\al-\tht,
	\al-\tht,\frac1{\cos\al},0}},
\quad
\text{with }
\al \in [-\g,\g] \subset \paren{-\frac{\pi}2,\frac{\pi}2}.
\end{align*}
Therefore, for $\de>0$, it only remains to find zeroes of the function $G$ using the implicit function theorem around every solution of this family.
\end{proof}

\begin{figure}[t]
	\centering
	\begin{overpic}[height=3.8cm]{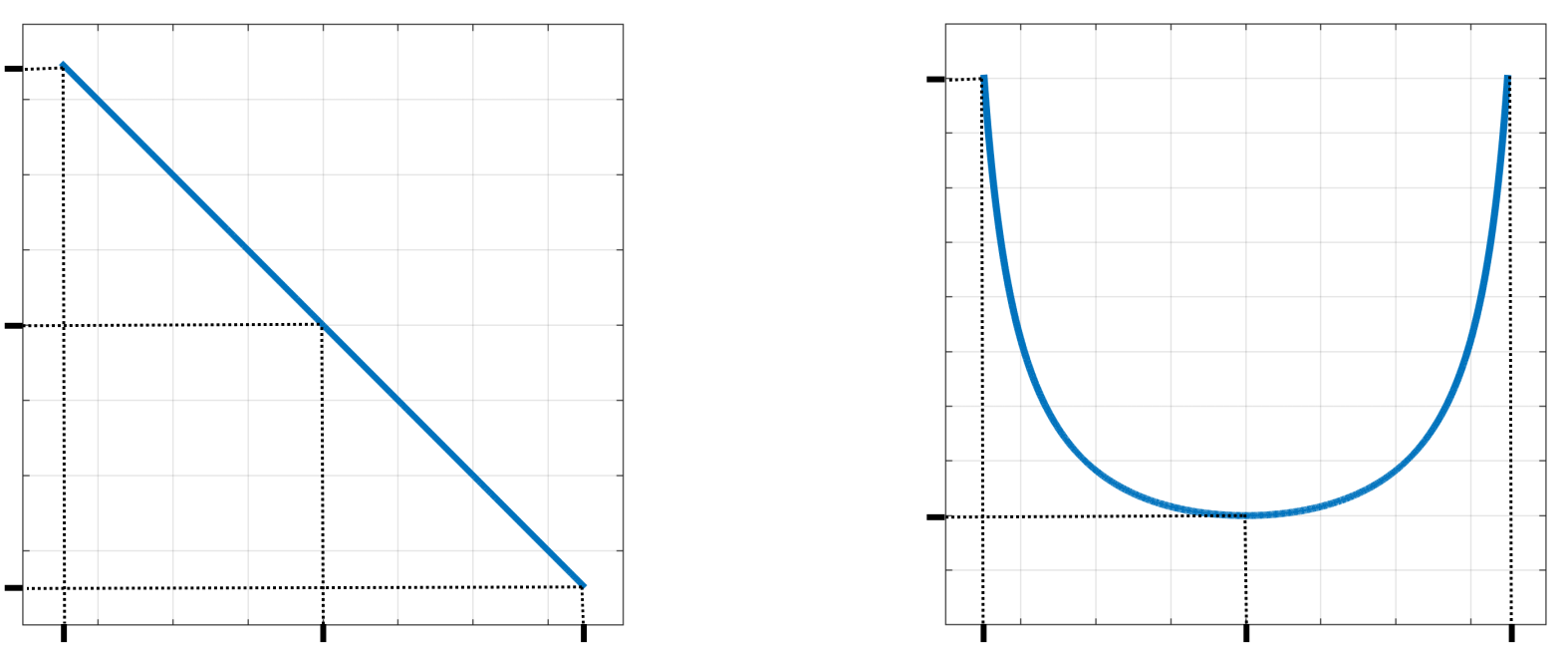}
		\put(21,24){\textcolor{blue}{$\tauU_{*}(\tauS,0)$}}
		\put(19,-2){\tiny $-\tht$}
		\put(0,-2){\tiny $-\g-\tht$}
		\put(34,-2){\tiny $\g-\tht$}
		\put(-12,3){\tiny $\pi-\g-\tht$}
		\put(-7,20){\tiny $\pi-\tht$}
		\put(-12,36){\tiny $\pi+\g-\tht$}
		\put(19,-6){\footnotesize $\tauS$}
		\put(73,12){\textcolor{blue}{$r_{*}(\tauS,0)$}}
		\put(78,-2){\tiny $-\tht$}
		\put(58,-2){\tiny $-\g-\tht$}
		\put(94,-2){\tiny $\g-\tht$}
		\put(57,7){\tiny $1$}
		\put(53,36){\tiny $\frac1{\cos\g}$}
		\put(78,-6){\footnotesize $\tauS$}
	\end{overpic}
	\vspace{0.8cm}
	\caption{Plot in $\tauS$ of functions $\tauU_{*}(\tauS,0)$ and $r_*(\tauS,0)$ as given in Lemma \ref{lemma:intersectionsA}.}
	\label{fig:implicitaInterseccions}
\end{figure}
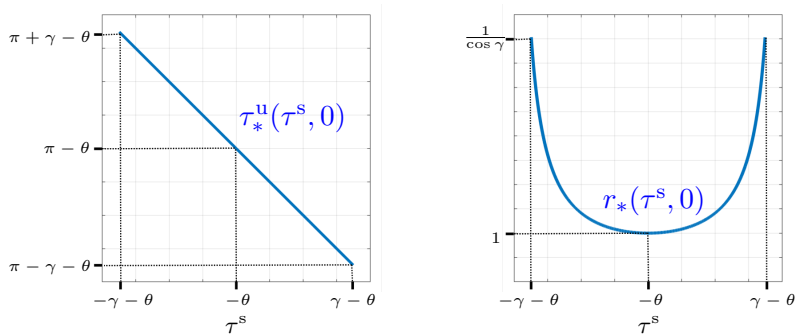

The second statement of Theorem \ref{mainTheoremA} is a consequence of this lemma.
% , we assume the setting given by Lemma \ref{lemma:intersectionsA}.
%
Indeed, take $R>1$ and $\g = \arccos(\frac1{R}) \in (0,\frac{\pi}2)$. Then, Lemma \ref{lemma:intersectionsA} implies that the equation $G(\tauU,\tauS,r,\de)=(0,0)$ has at least one solution  for $r\in[r_{\min}(\de),r_{\max}(\de)]$ and $\de \in (0,\de_{\gamma})$, with
\[
r_{\min}(\de) = 1 + \OO\paren{\frac1{\vabs{\log \de}}},
\qquad
r_{\max}(\de) = R + \OO\paren{\frac1{\vabs{\log \de}}}.
\]
% Then, notice that the equation $G(\tauU,\tauS,r,0)=(0,0)$  has solutions for $r\in[1,R]$ given by \eqref{def:Salpha},
% %
% see Figure \ref{fig:implicitaInterseccions}.
% %
% Moreover, the minimum and maximum points of $r_{*}(\tauS,0)$ are given at $r_*(-\tht,0)=1$ and $r_*(\pm\g-\tht,0)=R$.
% %
% Therefore, for $\de \in (0,\de_{\gamma})$, there exist $r_{\min}(\de)$ and $r_{\max}(\de)$ such that, for $r\in[r_{\min}(\de),r_{\max}(\de)]$, equation $G(\tauU,\tauS,r,\de)=(0,0)$ has at least one solution and
% \[
% r_{\min}(\de) = 1 + \OO\paren{\frac1{\vabs{\log \de}}},
% \qquad
% r_{\max}(\de) = R + \OO\paren{\frac1{\vabs{\log \de}}}.
% \]
% In addition, there exists $\tauS_{\min}(\de)$, such that 
% \begin{align}\label{proof:rmin}
% 	\tauS_{\min}(\de) = -\tht + \OO\paren{\frac1{\vabs{\log \de}}},
% 	\quad
% 	r_{\min}(\de) = r_*(\tauS_{\min}(\de),\de),
% 	\quad
% 	\partial_{\tauS} r_*(\tauS_{\min}(\de),\de) = 0.
% \end{align}
Taking into account \eqref{proof:changerhor}, we define
\[
\rho_{\min}(\de) = \frac{\sqrt[6]{2}}2
\de^{\frac13} e^{-\frac{A}{\de^2}}
\vabss{\CInn} r_{\min}(\de),
\qquad %\text{and} \qquad
\rho_{\max}(\de) = \frac{\sqrt[6]{2}}2
\de^{\frac13} e^{-\frac{A}{\de^2}}
\vabss{\CInn} r_{\max}(\de),
\]
and assume $\de>0$ small enough such that $\rho_{\max}(\de)<\rhoPeriodicOrbit$.
Then, for $\rho \in [\rho_{\min}(\de),\rho_{\max}(\de)]$, the closed curves $\partial \DD^{\unstable}_{\rho}$ and $\partial \DD^{\stable}_{\rho}$ (see \eqref{def:curvesIntersections}) intersect at least once.
See Figure \ref{fig:intersectionsTechnical} for a representation of the case $\de=0$.

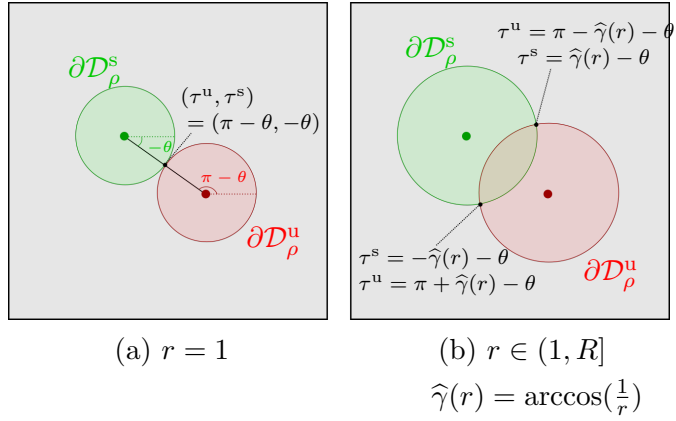
\begin{figure}[t]
	\centering
	\begin{overpic}[height=4.2cm]{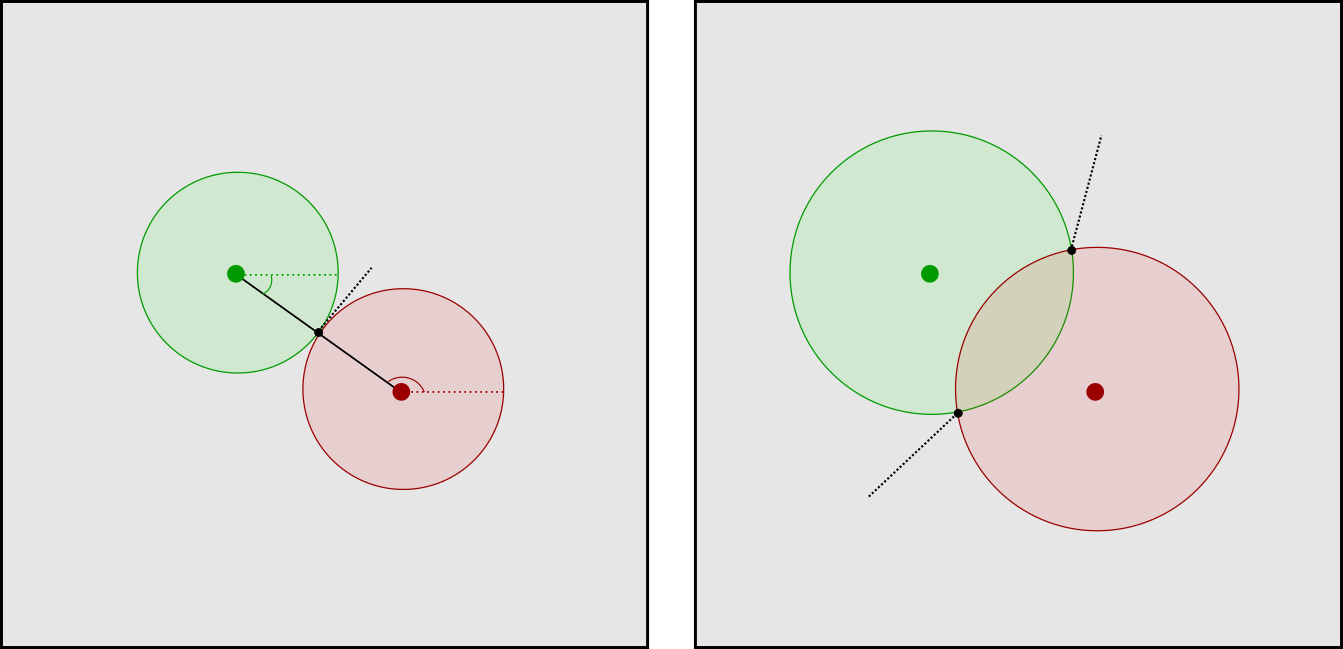}
		\put(16,-6){(a) $r=1$}
		\put(29,21){\tiny\textcolor{red}{$\pi-\tht$}}
		\put(21,25.5){\tiny\textcolor{myGreen}{$-\tht$}}
		\put(9,37){\textcolor{myGreen}{$\partial \DD^{\stable}_{\rho}$}}
		\put(36,11){\textcolor{red}{$\partial \DD^{\unstable}_{\rho}$}}
		\put(26,33){\scriptsize $(\tauU,\tauS)$}
		\put(27,29){\scriptsize $=(\pi-\tht,-\tht)$}
		\put(65,-6){(b) $r\in(1,R]$}
		\put(64,-13){$\wh{\g}(r)=\arccos(\frac1{r})$}
		\put(60,40){\textcolor{myGreen}{$\partial \DD^{\stable}_{\rho}$}}
		\put(87,6){\textcolor{red}{$\partial \DD^{\unstable}_{\rho}$}}
		\put(74,43){\scriptsize $\tauU=\pi-\wh{\g}(r)-\tht$}
		\put(77,39){\scriptsize $\tauS=\wh{\g}(r)-\tht$}
		\put(53,8.5){\scriptsize $\tauS=-\wh{\g}(r)-\tht$}
		\put(53,4.5){\scriptsize $\tauU=\pi+\wh{\g}(r)-\tht$}
	\end{overpic}
	\vspace{1.5cm}
	\caption{Representation of solutions of the equation \eqref{def:functionG} in function of the coordinate $r$.}
	\label{fig:intersectionsTechnical}
\end{figure}

Finally, we prove the third and fourth statement of Theorem~\ref{mainTheoremA}.
Let us denote the solutions of equation $G(\tauU,\tauS,r,\de)=(0,0)$ given in Lemma~\ref{lemma:intersectionsA} as
\[
P(\tauS,\de) = (\tauU_{*}(\tauS,\de),\tauS, r_*(\tauS,\de),\de)
\]
and consider the function
\[
\wt{G}(\tauS,\de) = \det \paren{\frac{\partial G}{\partial (\tauU,\tauS)}(P(\tauS,\de))},
\qquad (\tauS,\de) \in I_{\gamma} \times (0,\de_{\gamma}).
\]
Then, the values such that $\wt{G}\neq 0$ correspond to transverse intersections of the closed curves  $\partial \DD^{\unstable}_{\rho}$ and $\partial \DD^{\stable}_{\rho}$. 
Likewise, the values such that $\wt{G}=0$ and $\partial_{\tauS} \wt{G} \neq 0$ correspond to  quadratic tangencies.

To characterize the transverse intersections and the quadratic tangencies, we define $\tauS_{\min}(\de)$, the value of $\tauS$ where $r_*$ reaches its minimum value $r_{\min}$. Note that this corresponds to a critical point, 
% that is $	\partial_{\tauS} r_*(\tauS_{\min}(\de),\de) = 0$, 
which, by Lemma \ref{lemma:intersectionsA}, satisfies
%  In addition, there exists $\tauS_{\min}(\de)$, such that 
\begin{align}\label{proof:rmin}
\tauS_{\min}(\de) = -\tht + \OO\paren{\frac1{\vabs{\log \de}}},
\quad
r_{\min}(\de) = r_*(\tauS_{\min}(\de),\de),
\quad
\partial_{\tauS} r_*(\tauS_{\min}(\de),\de) = 0.
\end{align}
Now we prove that 
% Therefore, by \eqref{proof:rmin}, we need to see 
that 
$(\tauS_{\min}(\de),\de)$ is a simple zero of $\wt{G}$ and otherwise $\wt{G}\neq 0$, for $\tauS\neq \tauS_{\min}(\de)$.
By the definition of function $G$ in \eqref{def:functionG} and Lemma \ref{lemma:intersectionsA},  for $(\tauS,\de) \in I_{\gamma}\times (0,\de_{\g})$, one has that
%\begin{align*}%\label{proof:functionGwt}
%	\wt{G}(\tauS,\de)
%	= r_*^2(\tauS,\de) \sin(\tauU_*(\tauS,\de)-\tauS) + \wt{g}(\tauS,\de),
%\end{align*}	
%where
%\begin{align*}
%	\wt{g}(\tauS,\de) =& \,
%	\det \paren{\frac{\partial (\Re g, \Im g)}{\partial (\tauU,\tauS)} (P(\tauS,\de)) }
%	- r_*\sin(\tauU_{*}+\tht) 
%	\partial_{\tauS} \Im g (P(\tauS,\de))
%	\\
%	&-r_* \cos(\tauS+\tht) 
%	\partial_{\tauU} \Re g (P(\tauS,\de))
%	-r_*\sin(\tauS+\tht) 
%	\partial_{\tauU} \Im g (P(\tauS,\de))
%	\\
%	&-r_* \cos(\tauU_*+\tht) 
%	\partial_{\tauS} \Re g (P(\tauS,\de))
%	= \OO \paren{\frac1{\vabs{\log \de}}},
%\end{align*}
%with $\tauU_*=\tauU_*(\tauS,\de)$ and $r_*=r_*(\tauS,\de)$.
%%
%Then, applying the results in Lemma \ref{lemma:intersectionsA}, 
%
\begin{align*}
	\wt{G}(\tauS,\de) = 2 \tan(\tauS+\tht)
	+ \OO\paren{\frac1{\vabs{\log \de}}},
	\quad
	\partial_{\tauS} \wt{G}(\tauS,\de)
	= \frac2{\cos^2(\tauS+\tht)}
	+ \OO\paren{\frac1{\vabs{\log \de}}},
\end{align*}
(see Figure \ref{fig:interseccionsWtG}).
Notice that, for $\de$ small enough, 
\begin{align*}
	\partial_{\tauS}\wt{G}(\tauS,\de) \geq 
	2 + \OO\paren{\frac1{\vabs{\log \de}}} > 0.
\end{align*}
Therefore, $\wt{G}$ is a strictly increasing function in $\tauS$ and can only have one simple zero.
Moreover, this zero corresponds to $\tauS=\tauS_{\min}(\de)$.
Indeed, since $G(P(\tauS_{\min}(\de),\de))=(0,0)$ and $\partial_{\tauS} r_*(\tauS_{\min}(\de),\de) = 0$ (see \eqref{proof:rmin}), taking the derivatives one has that
\begin{align*}
	\partial_{\tauS} G(P(\tauS_{\min}(\de),\de))
	+
	\partial_{\tauU} G(P(\tauS_{\min}(\de),\de)) 
	\partial_{\tauS} \tauU_*(\tauS_{\min}(\de),\de) = (0,0),
\end{align*} 
and, as a result, the vectors $\partial_{\tauS} G$ and $\partial_{\tauU} G$ at $P(\tauS_{\min}(\de),\de)$ are linearly dependent and, therefore, $\wt{G}(\tauS_{\min},\de)=0$.
Hence, there exists at least one quadratic tangency at $r=r_{\min}(\de)$ and at least two transverse intersection for each $r\in(r_{\min}(\de),r_{\max}(\de)]$.

\begin{figure}[t]
	\centering
	\begin{overpic}[height=4.2cm]{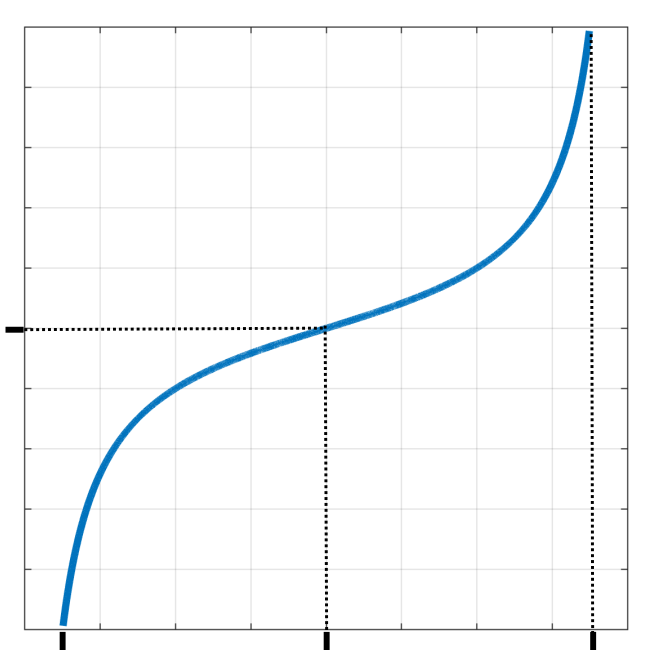}
		\put(40,60){\textcolor{blue}{$\wt{G}(\tauS,0)$}}
		\put(48,-3){\tiny $-\tht$}
		\put(0,-3){\tiny $-\g-\tht$}
		\put(86,-3){\tiny $\g-\tht$}
		\put(-5,48){\tiny $0$}
		\put(49,-10){\footnotesize $\tauS$}
	\end{overpic}
	\vspace{0.5cm}
	\caption{Plot in $\tauS$ of function $\wt{G}(\tauS,0)$ as given in \eqref{def:functionG}.}
	\label{fig:interseccionsWtG}
\end{figure}

\subsection{Proof of Proposition \ref{proposition:existencia2D}}
\label{section:proofExistence}

% In this section, we prove Propositions~\ref{proposition:existencia2D} and~\ref{proposition:comparison1d2d}. 
%

From now on, we consider a fixed $\banda>0$ and the corresponding complex torus $\torusC$ (see \eqref{def:torusC}).
We also set $\rhoPeriodicOrbit$ satisfying the conditions in Proposition \ref{proposition:periodicOrbit} and $\rho \in [0,\rhoPeriodicOrbit]$.
To avoid cumbersome notations, throughout the rest of the section, we omit the dependence on the parameter $\de$ unless necessary and denote by $C$ any positive constant independent of $\de$ and $\rho$ to state estimates.
%To avoid cumbersome notations, in the following, we omit the dependence in the parameters $\de$ and $\rho$, unless necessary.
We only prove the results for the unstable manifold, the proof for the stable manifold is analogous.

We look for parametrizations of the invariant manifold $\WW^{u,+}(\po_{\rho})$ of the form 
\begin{align*}
	\zuD(u,\tau) =  \po_{\rho}(\tau) + \s_{\pend}(u) + \zuDper(u,\tau),
	\qquad
	(u,\tau) \in \DmigU \times \torusC,
\end{align*}
(see \eqref{def:parametrizationZusD}) satisfying the equation~\eqref{eq:existence2D} and the asymptotic condition given in~\eqref{eq:assymptoticConditionExistence2d}.

Let us recall that we split the Hamiltonian $H$ as $H=H_{\pend}+H_{\osc}+H_1$ (see \eqref{def:hamiltonianScaling}). 
Since $\s_{\pend}=(\lap,\Lap,0,0)$ is a solution of the unperturbed system $H_{\pend}+H_{\osc}$, it satisfies the invariance equation~\eqref{eq:existence2D} for the unperturbed Hamiltonian (see Proposition \ref{proposition:singularities}).
By Proposition~\ref{proposition:periodicOrbit},  $\po_{\rho}$ also satisfies~\eqref{eq:existence2D} (for the full Hamiltonian $H$).
Then, the parametrization $\zdDper$ satisfies
\begin{align}\label{eq:existence2Dtechincal}
	\LL_{\rho} \zuDper = \RRR_{\rho}[\zuDper], 
\end{align}
where
\begin{align}\label{def:operatorLL}
	\LL_{\rho} \zeta = \paren{\partial_u + \frac{\omegar}{\de^2}\partial_{\tau}-\AAA(u) }\zeta,
	\quad
	\AAA = \begin{pmatrix}
		0 & -3 & 0 & 0 \\
		-V''(\lap(u)) & 0 & 0 & 0 \\
		0 & 0 & \frac{i}{\de^2} & 0 \\
		0 & 0 & 0 &-\frac{i}{\de^2}
	\end{pmatrix}
\end{align}
and
\begin{align}\label{def:operatorRRR}
	\RRR_{\rho}[\zeta] = \begin{pmatrix}
		\partial_{\La} H_1(\po_{\rho}+\s_{\pend}+\zeta)
		- \partial_{\La} H_1(\po_{\rho})
		\\
		- T_{\rho}[\zeta_1]
		-\partial_{\la} H_1(\po_{\rho}+\s_{\pend}+\zeta) 
		+ \partial_{\la} H_1(\po_{\rho})
		\\
		i\partial_{y} H_1(\po_{\rho}+\s_{\pend}+\zeta)
		-i\partial_{y} H_1(\po_{\rho})
		\\
		-i\partial_{x} H_1(\po_{\rho}+\s_{\pend}+\zeta)
		+i\partial_{x} H_1(\po_{\rho})
	\end{pmatrix},
\end{align}
with
\begin{align}\label{def:operatorT}
	T_{\rho}[\zeta_1] = V'(\lap+\po_{\rho,1}+\zeta_1)
	-V'(\lap) - V'(\po_{\rho,1})
	- V''(\lap)\zeta_1.
\end{align}
%
%In addition, the asymptotic condition~\eqref{eq:assymptoticConditionExistence2d} and Theorem \ref{theorem:singularities}, $\zuDper$ must satisfy
%\begin{align*}\label{eq:assymptoticConditionExistence2dTechnical}
%	\lim_{\Re u \to -\infty} \zuDper(u,\tau) = 0,
%	\qquad
%	\lim_{\Re u \to +\infty} \zsDper(u,\tau) = 0. 
%\end{align*}

We solve equation \eqref{eq:existence2Dtechincal} by means of a fixed point scheme on a suitable Banach space. 
For $\al\geq 0$, we consider the Banach space
\begin{align*}%\label{def:BanachSpaceXY}
	\Ycal{\al} =& \claus{\zeta:\DmigU \times \torusC \to \complexs \st
		\zeta \text{ real-analytic}, \, 
		\normY{\zeta}{\al}:=\sup_{(u,\tau)\in \DmigU\times\torusC} \vabs{e^{-\al u}\zeta(u,\tau)}<+\infty	},
\end{align*}
where $\DmigU$ is the domain introduced in \eqref{def:dominiMig}.
We also consider the product Banach space ${\Ycal{\al}^4=\Ycal{\al}\times...\times \Ycal{\al}}$ endowed with the norm
\begin{align*}  
	\normYprod{\zeta}{\al}=\sum_{j=1}^4
	\normY{\zeta_j}{\al}.
\end{align*}
In the next lemma, we state some properties of these Banach spaces.
We will use them throughout the section.

\begin{lemma}\label{lemma:sumnorms}
The following statements hold.
	\begin{enumerate}
		\item If $\al\geq\beta \geq 0$, then
		$\Ycal{\al} \subseteq \Ycal{\beta}$.
		Moreover, for $\zeta \in \Ycal{\al}$,
		$\normY{\zeta}{\beta} \leq C
		\normY{\zeta}{\al}$.
		\item If $\zeta \in \Ycal{\al}$ and
		$\eta \in \Ycal{\beta}$, then
		$\zeta\eta \in \Ycal{\al+\beta}$
		and
		$\normY{\zeta\eta}{\al+\beta} \leq \normY{\zeta}{\al}
		\normY{\eta}{\beta}$.
		%
		%		\item If $\zeta \in \Ycal{\al}$, then
		%		$\normY{\zeta}{\al} \leq \sup_{\tau \in \Pi_{\s}} \norm{\zeta(\cdot,\tau)}_{\al}$.
	\end{enumerate}

\end{lemma}

Next, we obtain and analyze a suitable right-inverse of the operator $\LL_{\rho}$ introduced in~\eqref{def:operatorLL}. The first step is to construct a fundamental matrix for $\dot{\zeta} = \AAA(u) \zeta$.

\begin{lemma}\label{lemma:fundamentalMatrixFlow}
Fix $u_0 \in \reals\setminus\claus{0}$ and consider the linear differential equation $\dot{\zeta} = \AAA(u) \zeta$, with $\AAA$ as given in \eqref{def:operatorLL}.
Then, a real-analytic fundamental matrix of this equation is
\begin{align*}
	%	\MoveEqLeft 	
	\Phi(u) = \begin{pmatrix}
		3f_{\Phi}(u) & 3g_{\Phi}(u) & 0 & 0  \\
		-\dot{f}_{\Phi}(u) & -\dot{g}_{\Phi}(u) & 0 & 0 \\
		0 & 0 & e^{\frac{i}{\de^2} u} & 0 \\
		0 & 0 & 0 & e^{-\frac{i}{\de^2} u}
	\end{pmatrix},
	%	\, \text{with } \,
	%	f_h(u) = \La_h(u) \int_{-\infty}^{u} \frac{1}{\La_h^2(v)} dv.
\end{align*}
with
\begin{align*}
f_{\Phi}(u) =
\frac1{3\xi(0)}
\paren{\xi(u)-
\frac{\dot{\xi}(0)}{\Lap(0)}
		\Ladp(u)},
\quad
g_{\Phi}(u) = -\frac{\Lap(u)}{\Lap(0)},
\quad
\xi(u) = \Lap(u) \int_{u_0}^{u} \frac{dv}{\Lap^2(v)},
\end{align*}
where, in the last integral, we consider an integration path in $\DmigU$ given by the straight line if $u \in \complexs\setminus\reals$ and by a path avoiding $u=0$ when $u \in \reals$.
%\begin{align*}
%	f_{\Phi}(u) =
%	\frac1{3f_{\phi}(0)}
%	\paren{f_{\phi}(u)-
%		\frac{\dot{f}_{\phi}(0)}{\dot{g}_{\phi}(0)}
%		g_{\phi}(u)},
%	\qquad
%	g_{\Phi}(u) = -\frac{g_{\phi}(u)}{\dot{g}_{\phi}(0)},
%\end{align*}
%where $f_{\phi}$ and $g_{\phi}$ are as given in \eqref{def:fundamentalMatrixphi}.
%

Moreover, $\Phi(u)$ satisfies that $\det\Phi(u) = 1$, $\Phi(0) = \mathbf{Id}$ and that there exists a constant $C>0$ such that, denoting $\nu = \frac12\sqrt{\frac{21}8}$,
\begin{align*}%\label{eq:normsFundamentalFunctions}
	\normX{g_{\Phi}}{2\nu} \leq C, 
	\quad
	\normX{\dot{g}_{\Phi}}{2\nu} \leq C,
	\quad
	\normX{f_{\Phi}}{-2\nu} \leq C, 
	\quad
	\normX{\dot{f}_{\Phi}}{-2\nu} \leq C.
\end{align*}
\end{lemma}

\begin{proof}
Let us recall that, by Proposition \ref{proposition:singularities}, the time-parametrization of the separatrix satisfies that
$
\ladp(u)=-3\Lap(u)
$
and
$
\Ladp(u) = - V'(\lap(u)),
$
for $u \in \PiExtA$.
%
%\[
%\dot{\zeta} = \begin{pmatrix}
%	0 & -3 \\
%	-V''(\lap(u)) & 0 
%\end{pmatrix} \zeta.
%\]
%
Then, a fundamental matrix of the equation $\dot{\zeta}=\AAA(u)\zeta$ is given by
\begin{align*}%\label{def:fundamentalMatrixphi}
	\phi(u) = \begin{pmatrix}
		3\xi(u) & 3 \Lap(u) & 0 & 0\\
		-\dot{\xi}(u) & -\Ladp(u) & 0 & 0\\
		0 & 0 & e^{\frac{i}{\de^2} u} & 0 \\
		0 & 0 & 0 & e^{-\frac{i}{\de^2} u}
	\end{pmatrix}.
%	\qquad
%	f_{\phi}(u) = \Lap(u) \int_{u_0}^{u} \frac{1}{\Lap^2(v)} dv,
%	\quad
%	g_{\phi}(u) = \Lap(u),
\end{align*}
We stress that $\xi$ is real-analytic in $\DmigU \subset \PiExtA$.
Indeed, one has that $u=0$ is the only zero of $\Lap(u)$ (see Proposition~\ref{proposition:singularities}), that $\Ladp(0) =-V'(\lap(0)) \neq 0$ and $\ddot{\La}_{\pend}(0)=0$.
Thus, $\Lap(u)=\Ladp(0) u + \OO(u^3)$.
That implies that the integral appearing on $\xi$ does not depend on the path of integration since its residue is zero.
As a consequence, $\xi(u) \in \reals$ for $u \in \reals$.
In addition, since
$\xi(0) = -{\Ladp}^{-1}(0) \neq 0$,
%\begin{align*}
%	\xi(0) = -\frac1{\Ladp(0)} \neq 0,
%	\qquad
%	\Lap(0) = 0,
%	\qquad
%	\Ladp(0) \neq 0,
%\end{align*}	
we can perform a linear transformation to $\phi(u)$ to obtain the fundamental matrix $\Phi(u)$ satisfying $\Phi(0)=\Id$ and $\det \Phi(u) = 1$.
Lastly, recalling that, by Proposition \ref{proposition:singularities}, 
$\normX{\lap}{2\nu} \leq C$	
and
$\normX{\Lap}{2\nu} \leq C$,
%\begin{equation}\label{eq:normsSeparatrix}
%\normX{\lap}{2\nu} \leq C, 
%\qquad
%\normX{\Lap}{2\nu} \leq C.
%\end{equation}
we obtain the corresponding estimates for $f_{\Phi}$ and $g_{\Phi}$.
\end{proof}

%\begin{proof}
%%
%\begin{align*}
%	f_{\phi}(0) &= 
%	\lim_{\vabs{u} \to 0} 
%	\Ladp(0)u(1+\OO(u^2))
%	\int_{u_0}^u
%	\frac1{\Lad^2_{\pend}(0) v^2}(1+\OO(v^2))dv
%	\\
%	&= \frac1{\Ladp(0)}\lim_{\vabs{u} \to 0} 
%	u\int_{u_0}^u \paren{\frac{1}{v^2} + \OO(1)} dv
%	\\
%	&= 
%	\frac1{\Ladp(0)}
%	\lim_{\vabs{u} \to 0} \paren{
%	-1+\frac{u}{u_0^2} + \OO(u(u-u_0))}
%	= -\frac1{\Ladp(0)}.
%\end{align*}
%\end{proof}
%
%However, $\dot{f}_{\phi}(0)$ does not have a simple expression. 
%
%We have that
%\begin{align*}
%	\dot{f}_{\phi}(0) &= 
%	\lim_{\vabs{u} \to 0}
%	\paren{
%		\Ladp(u)
%		\int_{u_0}^u \frac1{\Lap^2(v)} dv
%		+
%		\frac1{\Lap(u)} } 
%	\\
%	&= \Ladp(0)
%	\paren{
%		-\int_{u_0}^0
%		\frac{\ddot{\La}_{\pend}(v)}{
%			\Lap(v)\Ladp^2(v)} dv
%		+ \frac1{\Lap(u_0)\Ladp(u_0)}
%	}.
%\end{align*}

%
%\begin{align*}	
%	\Phi^{-1}(u) = \begin{pmatrix}
%		-\dot{g}_{\Phi}(u) & -3g_{\Phi}(u) & 0 & 0  \\
%		\dot{f}_{\Phi}(u) & 3f_{\Phi}(u) & 0 & 0 \\
%		0 & 0 & e^{-\frac{i}{\de^2} u} & 0 \\
%		0 & 0 & 0 & e^{\frac{i}{\de^2} u}
%	\end{pmatrix}.
%\end{align*}
%

We use this matrix $\Phi$ to construct a right-inverse of the operator $\LL_{\rho}$  in~\eqref{def:operatorLL}.
For $\zeta \in \Ycal{\nu}^4$, we consider the operator 
\begin{equation}\label{def:operatorGG}
	\GG_{\rho}[\zeta](u,\tau)=\sum_{j=1}^4 \GG_{\rho,j}[\zeta](u,\tau) \mathbf{e_j},
\end{equation}
given by
\begin{align*}
	\left(
	\begin{array}{l}
		\displaystyle
		\GG_{\rho, 1}[\zeta](u, \tau) \\[0.4em]
		\GG_{\rho, 2}[\zeta](u, \tau)
	\end{array}  
	\right) 
	&=
	\begin{pmatrix}
		3 f_{\Phi}(u) & 3 g_{\Phi}(u) \\
		- \dot{f}_{\Phi}(u) & - \dot{g}_{\Phi}(u)
	\end{pmatrix}
	\left(
	\begin{array}{c}
		\displaystyle 
		\int_{-\infty}^{0} 
		\II_1[\zeta_1,\zeta_2]
		\paren{u+t,\tau+
			\frac{\omegar}{\de^2}t} dt
		\\[1em]
		\displaystyle
		\int_{-u}^0 
		\II_2[\zeta_1,\zeta_2]
		\paren{u+t,\tau+
			\frac{\omegar}{\de^2}t} dt
	\end{array}	\right)
\end{align*}
and
\begin{align*}
	\GG_{\rho,3}[\zeta](u,\tau) &=
	\int_{-\infty}^0
	e^{-\frac{i}{\de^2}t}
	\zeta_3
	\paren{u+t,\tau+\frac{\omegar}{\de^2}t} dt, 
	\\
	\GG_{\rho,4}[\zeta](u,\tau) &=
	\int_{-\infty}^0
	e^{\frac{i}{\de^2}t}
	\zeta_4
	\paren{u+t,\tau+\frac{\omegar}{\de^2}t} dt,
\end{align*}
where
\begin{align*}
	\II_1[\zeta_1,\zeta_2](u,\tau) &=
	-\dot{g}_{\Phi}(u) \zeta_1(u,\tau)
	- 3 g_{\Phi}(u) \zeta_2(u,\tau), 
	\\
	\II_2[\zeta_1,\zeta_2](u,\tau) &=
	\dot{f}_{\Phi}(u) \zeta_1(u,\tau)
	+ 3 f_{\Phi}(u) \zeta_2(u,\tau).
\end{align*}

\begin{lemma}\label{lemma:boundsGG}
%	Let $\rhoPeriodicOrbit>0$ be as given in Proposition \ref{proposition:periodicOrbit}.
%	% 	
%	Fix $\banda>0$ and consider the operator $\GG_{\rho}$ as defined in~\eqref{def:operatorGG}.
%
For $\rho \in [0,\rhoPeriodicOrbit]$ and $\de \in (0,1)$, the operator
	$
	\GG_{\rho}: \Ycal{\nu}^4 \to \Ycal{\nu}^4
	$
	is well defined and is a right-inverse of the operator $\LL_{\rho}$ given in \eqref{def:operatorLL}.
	Moreover, $\GG_{\rho, 2}[\zeta](0,\cdot)\equiv0$ and there exists a constant $C>0$ independent of $\rho$ and $\de$ such that 
	\[
	\normYprod{\GG_{\rho}[\zeta]}{\nu} \leq C \normYprod{\zeta}{\nu}.
	\]
In addition, if $\partial_{\tau}\zeta \equiv 0$, one has that $\GG_{\rho}[\zeta]=\GG_{\tl{\rho}}[\zeta]$ for $\rho, \tl{\rho} \in [0,\rhoPeriodicOrbit]$.
\end{lemma}

\begin{proof}
The fact that $\GG_{\rho}$ is a right inverse of $\LL_{\rho}$ is straightforward.
We show how to obtain estimates for $\GG_{\rho,1}$.
The estimates for $\GG_{\rho,2}$, $\GG_{\rho,3}$ and $\GG_{\rho,4}$ are analogous.

Let $\zeta_1, \zeta_2 \in \Ycal{\nu}$.
By the estimates in Lemma \ref{lemma:fundamentalMatrixFlow}, for $(u,\tau) \in \DmigU \times \torusC$ one has
\begin{align*}
	\vabs{\II_1[\zeta_1,\zeta_2](u,\tau)} 
	&\leq C 
	\vabss{e^{3\nu u}}\paren{\normY{\zeta_1}{\nu} +
		\normY{\zeta_2}{\nu} },
	\\
	\vabs{\II_2[\zeta_1,\zeta_2](u,\tau)} &\leq C
	\vabss{e^{-\nu u}} \paren{\normY{\zeta_1}{\nu} +
		\normY{\zeta_2}{\nu} }.
\end{align*}
Then,
\begin{align*}
	\vabs{\GG_{\rho, 1}(u,\tau)e^{-\nu u}}
	\leq& \,
	C \vabs{e^{-3\nu u}}\vabs{\int_{-\infty}^0 
		\II_1[\zeta_1,\zeta_2]\paren{u+t,\tau+\frac{\omegar}{\de^2}t} dt} \\
	&+ 
	C \vabs{e^{\nu u}}\vabs{\int_{-u}^0 
		\II_2[\zeta_1,\zeta_2]\paren{u+t,\tau+\frac{\omegar}{\de^2}t} dt} \\
	\leq& \,
%	C \paren{1+\vabs{1-e^{\nu u}}} \paren{\normY{\zeta_1}{\nu} +
%		\normY{\zeta_2}{\nu} }
%	\\
%	\leq& \,
%\leq
	C \paren{\normY{\zeta_1}{\nu} +
		\normY{\zeta_2}{\nu} }.
\end{align*}
\end{proof}

We introduce the fixed point operator
\begin{align}\label{def:operatorFF}
	\FF_{\rho} = \GG_{\rho} \circ \RRR_{\rho},
\end{align}
with $\RRR_{\rho}$ and $\GG_{\rho}$ as given in \eqref{def:operatorRRR} and \eqref{def:operatorGG}, respectively.
Then, equation~\eqref{eq:existence2Dtechincal} can be expressed as $\zuDper=\FF_{\rho}[\zuDper]$.

Proving Proposition \ref{proposition:existencia2D} is equivalent to prove the following result.

\begin{proposition}\label{proposition:existencia2Dtechnical}
Let $\rhoPeriodicOrbit>0$ be the constant given in Proposition~\ref{proposition:periodicOrbit}.  
There exist $\de_0>0$ and $\cttExistenciaDtech>0$ such that, for $\rho \in [0,\rhoPeriodicOrbit]$ and $\de \in (0,\de_0)$, the equation $\zuDper = \FF[\zuDper]$ has a unique solution $\zuDper \in \Ycal{\nu}^4$ satisfying
\begin{align*}
	\normYprod{\zuDper}{\nu} \leq \cttExistenciaDtech \de.
\end{align*}
\end{proposition}

\begin{proof}
For $\varsigma>0$, let us consider 
$
	B(\varsigma) = \claus{\zeta \in \Ycal{\nu}^4 \st
		\normYprod{\zeta}{\nu} \leq \varsigma}.
$
We will check that $\FF_{\rho}:B(\varsigma)\to B(\varsigma)$ is a contraction for a suitable $\varsigma$.

We first claim that there exist $\de_0>0$ such that, for $\rho \in [0,\rhoPeriodicOrbit]$ and
$\de \in (0,\de_0)$,
\begin{align}\label{proof:lemmaBoundsRRR}
	\normYprod{\RRR_{\rho}[\zeta]}{\nu} \leq C\de,
	\qquad
	\normYprod{\partial_j \RRR_{\rho}[\zeta]}{0}
	\leq C\de,
\end{align}
for $\zeta \in B(\varsigma \de)$ and $j=1,..,4$.
Indeed, we obtain the estimates for $\RRR_{\rho,2}[\zeta]$, the other cases are proven analogously.
For the derivatives it is enough to apply Cauchy estimates.

We recall the definitions
\begin{align}\label{proof:R1Def}
\begin{aligned}
\separatrix &= (\lap,\Lap,0,0)^T, 
\\
{\po}_{\rho} &=  (0,\de^2\LtresLa,\de^3\Ltresx,\de^3\Ltresy)^T + 
	\rho(0,0,e^{i\tau},e^{-i\tau})^T + 
	\de\rho\big(\lapo,\Lapo,\xpo,\ypo\big)^T,
\\
\RRR_{\rho,2}[\zeta] &=
-\partial_{\la} H_1(\po_{\rho}+\s_{\pend}+\zeta) 
+ \partial_{\la} H_1(\po_{\rho})
- T_{\rho}[\zeta_1],
\\
T_{\rho}[\zeta_1] &= V'(\lap+\de\rho\lapo+\zeta_1)
-V'(\lap) - V'(\de\rho\lapo)
- V''(\lap)\zeta_1,
\end{aligned}
\end{align}
where $V$ is the potential given in \eqref{def:potentialV}.
Then, by the mean value theorem, 
\begin{equation*}%\label{proof:R1meanValueTh}
\begin{split}
\RRR_{\rho,2}[\zeta](u,\tau) 
=&\,
-
\int_0^1 D \partial_{\la} 
H_1(s \s_{\pend}(u)+ s \zeta(u,\tau) +
\po_{\rho}(\tau) ) ds
\paren{\s_{\pend}(u)+\zeta(u,\tau)}
\\
&- \zeta_1(u,\tau) \boxClaus{V''(\lap(u)+\de\rho\lapo(\tau)) 
	- V'''(\lap(u))}
+ \OO\paren{\zeta_1(u,\tau)}^2
\\
&-
\de\rho\lapo(\tau)\lap(u) V'''(0) + \OO\paren{\de\rho\lapo(\tau)\lap(u)}^2.
\end{split}
\end{equation*}
%
%\begin{equation*}%\label{proof:R1meanValueTh}
%	\begin{split}
%		\RRR_{\rho,2}&[\zeta](u,\tau) 
%		=
%		-\int_0^1 D \partial_{\la} 
%		H_1(s \s_{\pend}(u)+ s \zeta(u,\tau) +
%		\po_{\rho}(\tau) ) ds
%		\paren{\s_{\pend}(u)+\zeta(u,\tau)}
%		\\
%		&-\lap(u)\int_0^1
%		(\de\rho\lapo(\tau)+s\zeta_1(u,\tau)) 
%		\boxClaus{
%			\int_0^1
%			V'''(s\lap(u)+r\de\rho\lapo(\tau)+rs\zeta_1(u,\tau))dr} ds
%		\\
%		&-\zeta_1(u,\tau)
%		\int_0^1 V''(s\lap+\de\rho\lapo+s\zeta_1)ds 
%		+ \zeta_1(u,\tau) V''(\lap(u)).
%	\end{split}
%\end{equation*}
%
From Proposition \ref{proposition:periodicOrbit} and Proposition \ref{proposition:singularities}, one easily checks that
\begin{align*}
	\normYprod{s \s_{\pend} + s\zeta+\po_{\rho}}{0} \leq C, \qquad \text{for } s\in [0,1].
\end{align*}
%Next, we obtain estimates for $s\separatrix+s\zeta+\po_{\rho}$.
%
%For $(u,\tau) \in \DmigU \times \torusC$, $s \in [0,1]$, $\rho\in[0,\rhoPeriodicOrbit]$ and $\de$ small enough, by Proposition \ref{proposition:periodicOrbit} and Theorem \ref{theorem:singularities}, one sees that
%\begin{align*}
%	\vabs{s \lap(u) + s\zeta_1(u,\tau) + \po_{\rho,1}(\tau)} 
%	&\leq 
%	\vabs{\lap(u)} + 
%	C\normY{\zeta_1}{\nu} +
%	\rho\de\vabs{\lapo(\tau)} < \pi, 
%	%
%	\\
%	\vabs{s \Lap(u) + s\zeta_2(u,\tau) +\po_{\rho,2}(\tau)} &\leq 
%	\vabs{\Lap(u)} + C\normY{\zeta_2}{\nu} + \rho\de\vabs{\Lapo(\tau)} + 
%	\de^2 \vabs{\LtresLa}
%	\leq C,
%	\\
%	\vabs{s\zeta_3(u,\tau)+\po_{\rho,3}(\tau)} &\leq
%	C\normY{\zeta_3}{\nu}
%	+ \de^3 \vabs{\Ltresx} +
%	\rho + \rho \de \vabs{\xpo(\tau)}
%	+ \de^3 \vabs{\Ltresx}
%	\leq C,
%	\\
%	\vabs{s\zeta_4(u,\tau)+\po_{\rho,4}(\tau)} &\leq
%	C\normY{\zeta_4}{\nu} 
%	+ \de^3 \vabs{\Ltresy}
%	+ \rho
%	+ \rho \de \vabs{\ypo(\tau)}
%	\leq C.
%\end{align*}
Thus, applying the estimates in Proposition \ref{proposition:HamiltonianScaling} and using that $\lap,\Lap \in \Ycal{2\nu}$,
\begin{align*}
	\normY{\RRR_{\rho,2}[\zeta]}{\nu} \leq& \,
	C \de \normY{\lap+\zeta_1}{\nu}+
	C \de^2 \normY{\Lap+\zeta_2}{\nu}+
	C \de \normY{\zeta_3}{\nu}+
	C \de \normY{\zeta_4}{\nu}
	\\
	&+C\normY{\zeta_1}{\nu}
	+C\de\rho\normY{\lap}{\nu}
	\leq C\de,
\end{align*}
which proves \eqref{proof:lemmaBoundsRRR}.

%To compute estimates for $\partial_1 \RRR_{2,\rho}[\zeta]$, by \eqref{proof:R1Def} and the Mean Value Theorem,
%\begin{align*}
%	\partial_1 \RRR_{2,\rho}[\zeta] =
%	-\partial^2_{\la}H_1(\po_{\rho}+\s_{\pend}+\zeta)
%	- (\de\rho\lapo+\zeta_1)
%	\int_0^1 V'''(\lap+ s\de\rho\lapo+ s\zeta_1) ds.
%\end{align*} 
%Then, by Proposition \ref{proposition:HamiltonianScaling},
%\begin{align*}
%	\normY{\partial_1 \RRR_{2,\rho}[\zeta]}{0}
%	\leq C\de + C\de\rho\normY{\lapo}{0}
%	+ C\normY{\zeta_1}{\nu}
%	\leq C\de.
%\end{align*}

As a consequence of \eqref{proof:lemmaBoundsRRR} and using Lemma \ref{lemma:boundsGG}, there exists a constant $\cttExistenciaDtech>0$ such that
\begin{align}\label{proof:firstIterationExistence}
	\normYprod{\FF_{\rho}[0]}{\nu} \leq
	C \normYprod{\RRR_{\rho}[0]}{\nu} \leq \frac12 \cttExistenciaDtech \de.
\end{align} 
In addition, for $\zeta,\zetaB \in B(\cttExistenciaDtech\de)$ and by the mean value theorem,
\begin{align*}
	\RRR_{\rho}[\zeta]- \RRR_{\rho}[\zetaB]
	=
	\boxClaus{\int_0^1 D\RRR_{\rho}[s \zeta +(1-s)\zetaB] ds} 
	(\zeta-\zetaB).
\end{align*}
Then, from Lemma~\ref{lemma:boundsGG} and the estimates in \eqref{proof:lemmaBoundsRRR}, we deduce that
% \begin{equation}\label{proof:FFFLipschitz}
% \begin{split}	
% 	\normYprod{\FF_{\rho}&[\zeta]- \FF_{\rho}[\zetaB]}{\nu}
% 	\leq C
% 	\normYprod{\RRR_{\rho}[\zeta]- \RRR_{\rho}[\zetaB]}{\nu}
% 	\\
% 	&\leq
% 	\sup_{s\in[0,1]}
% 	\sum_{k=1}^4 \normY{\partial_k \RRR_{\rho}[s \zeta +(1-s)\zetaB]}{0}
% 	\normY{\zeta_k-\zetaB_k}{\nu} 
% 	\leq C \de \normYprod{\zeta-\zetaB}{\nu}. 
% \end{split}
% \end{equation}
\begin{equation}\label{proof:FFFLipschitz}
\begin{split}	
	\normYprod{\FF_{\rho}[\zeta]- \FF_{\rho}[\zetaB]}{\nu}
	&\leq C
	\normYprod{\RRR_{\rho}[\zeta]- \RRR_{\rho}[\zetaB]}{\nu}
	\\
	&\leq
	\sup_{s\in[0,1]}
	\sum_{k=1}^4 \normY{\partial_k \RRR_{\rho}[s \zeta +(1-s)\zetaB]}{0}
	\normY{\zeta_k-\zetaB_k}{\nu} 
	\leq C \de \normYprod{\zeta-\zetaB}{\nu}. 
\end{split}
\end{equation}
This implies that, taking $\de$ small enough,
$
	\normYprod{\FF_{\rho}[\zeta]- \FF_{\rho}[\zetaB]}{\nu}
	\leq \frac12 \normYprod{\zeta-\zetaB}{\nu}
$
and, therefore, $\FF_{\rho}:B(\cttExistenciaDtech\de) \to B(\cttExistenciaDtech\de)$ is well defined and contractive.
Hence, $\FF_{\rho}$ has a fixed point $\zuDper \in B(\cttExistenciaDtech\de)$.
\end{proof}

Proposition \ref{proposition:existencia2Dtechnical} completes the proof of  Proposition \ref{proposition:existencia2D}. Note that, since $\mathcal{G}_{\rho,2}[\zeta](0,\cdot)\equiv 0$ (see \eqref{def:operatorGG}) and $\mathcal{F}_\rho=\mathcal{G}_{\rho}\circ\mathcal{R}_\rho$, the solution obtained in Proposition \ref{proposition:existencia2Dtechnical} satisfies
	\[
	\langle\zdDper(0,\tau),\mathbf{e_2}\rangle= 0,
	\qquad 	\text{for all }
	\tau \in \torusC.
	\]

\subsection{Proof of Proposition \ref{proposition:comparison1d2d}}
\label{section:proofComparison}

To prove Proposition \ref{proposition:comparison1d2d}, let  us consider the parametrizations $\zuDper(u,\tau)$ and $\zuUper(u)$ given in Proposition \ref{proposition:existencia2D} and Corollary \ref{corollary:existencia1D}, respectively.

Let us recall that, by Proposition \ref{proposition:existencia2Dtechnical}, $\zuDper$ satisfies $\zuDper = (\GG_{\rho}\circ\RRR_{\rho})[\zuDper]$ 
and, as a result, $\zuUper = (\GG_{0}\circ\RRR_{0})[\zuUper]$.
By Lemma \ref{lemma:boundsGG}, since $\zuUper$ does not depend on $\tau$, one has that
\begin{align}\label{eq:existence1D}
%	\zuUper = \FF_{0}[\zuUper], 
%	\qquad \text{and} \qquad
	\zuUper = \GG_{\rho}\circ\RRR_{0}[\zuUper],
	\qquad \text{for any }
	\rho \in [0,\rhoPeriodicOrbit].
\end{align}
Then, by Proposition \ref{proposition:existencia2Dtechnical},
\begin{equation}\label{proof:splittingZu}
	\begin{aligned}
		\zuDper-\zuUper
		&= \FF_{\rho}[\zuDper] - \GG_{\rho} \circ \RRR_{0}[\zuUper] \\
		&= \FF_{\rho}[\zuDper]
		- \FF_{\rho}[\zuUper] +
		\GG_{\rho} \paren{\RRR_{\rho}[\zuUper] - \RRR_0[\zuUper]},
	\end{aligned}	
\end{equation}	
where we recall that $\FF_{\rho} = \GG_{\rho} \circ \RRR_{\rho}$ (see \eqref{def:operatorFF}).	

Let us consider the constant $\cttExistenciaD$ as given in Proposition~\ref{proposition:existencia2D}.
It is clear that,
\begin{align*}
	\zuDper, \zuUper \in 
	B(\cttExistenciaD\de) := \claus{\zeta \in \Ycal{\nu}^4 \st
		\normYprod{\zeta}{\nu} \leq \cttExistenciaD \de}.
\end{align*} 
Since $\FF_{\rho}$ is contractive with Lipschitz constant $\Lip(\FF_{\rho})\leq C\de$ (see \eqref{proof:FFFLipschitz}), for $\de$ small enough, one has that
\begin{align*}
	\normYprod{\FF_{\rho}[\zuDper]
		- \FF_{\rho}[\zuUper]}{\nu} 
	\leq C\de \normYprod{\zuDper-\zuUper}{\nu}
	\leq \frac12 \normYprod{\zuDper-\zuUper}{\nu}.
\end{align*}
Thus, by \eqref{proof:splittingZu} and Lemma \ref{lemma:boundsGG},
\begin{align}\label{proof:estimateszuDzuU}
	\normYprod{\zuDper-\zuUper}{\nu} \leq \frac12 
	\normYprod{\zuDper-\zuUper}{\nu} + C\normYprod{\RRR_{\rho}[\zuUper]- \RRR_{0}[\zuUper]}{\nu},
\end{align}

We claim that, for $\rho \in [0,\rhoPeriodicOrbit]$ and $\de>0$ small enough,	
\begin{align}\label{proof:lemmaBoundsRRRcomparison}
\normYprod{\RRR_{\rho}[\zuUper]-\RRR_{0}[\zuUper]}{\nu} \leq C\de\rho.
\end{align}
Indeed, first we consider estimates for $\RRR_{\rho,1}$ as given in \eqref{def:operatorRRR}.
One has that
\begin{align*}
	\RRR_{\rho,1}[\zuUper]-\RRR_{0,1}[\zuUper] =& \, \big(\partial_{\La} H_1(\separatrix+\po_{\rho}+\zuUper)-\partial_{\La} H_1(\po_{\rho})\big)
	\\
	&-\big(\partial_{\La} H_1(\separatrix+\po_{0}+\zuUper)-\partial_{\La} H_1(\po_{0})\big).
\end{align*}
Denoting $\po^s = (1-s)\po_0+s\po_{\rho}$, by the mean value theorem,
\begin{align*}
	\RRR_{\rho,1}[\zuUper]-\RRR_{0,1}[\zuUper] = 
%	\int_0^1 \boxClaus{ D\partial_{\La} H_1(\separatrix+\zuUper+\po^s) -D\partial_{\La} H_1(\po^s)} ds 
%	\paren{\po_{\rho}-\po_0} 
%	\\
%	=& \,
	\paren{\po_{\rho}-\po_0}^T
	\boxClaus{ \int_{[0,1]^2}
	D^2\partial_{\La} H_1(r(\separatrix+\zuUper)+\po^s)  dr ds }
	\paren{\separatrix + \zuUper}.
\end{align*}
Then, using Lemma \ref{lemma:sumnorms} and for $(\al_1,\al_2,\al_3,\al_4)=(\la,\La,x,y)$, one sees that
\begin{equation*}%\label{proof:RRRrho0}
	\begin{split}	
		\normY{\RRR_{\rho,1}[\zuUper]-\RRR_{0,1}[\zuUper]}{\nu}
		\leq
		\sum_{j=1}^4 \sum_{k=1}^4 
		\sup_{s\in[0,1]} 
		\sup_{r\in[0,1]} &
		\normY{\partial_{\al_j \al_k\La} H_1(r(\separatrix+\zuUper)+\po^s)}{0} \\
		&\cdot \normYprod{\separatrix+\zuUper}{\nu}
		\normYprod{\po_{\rho}-\po_{0}}{0}.
	\end{split}
\end{equation*}
Notice that, Proposition \ref{proposition:periodicOrbit} implies that $\normYprod{\po_{\rho}-\po_{0}}{0} \leq C\rho$
and Proposition \ref{proposition:singularities} and  Corollary~\ref{corollary:existencia1D} imply that 
$\normYprod{\separatrix+\zuUper}{\nu}\leq C$.
These estimates and those of Proposition \ref{proposition:HamiltonianScaling}, which bound $\normY{\partial_{\al_j \al_k \La} H_1}{0}$, imply that
\begin{equation*}%\label{proof:estimatesRRRrho1}
	\normY{\RRR_{\rho,1}[\zuUper]-\RRR_{0,1}[\zuUper]}{\nu}
	\leq C\de\rho.
\end{equation*}
Analogously, it can be seen that
\begin{equation*}%\label{proof:estimatesRRRrho234}
	\begin{split}
		\normY{\RRR_{\rho,2}[\zuUper]-\RRR_{0,2}[\zuUper]}{\nu}
		&\leq C\de\rho
		+ \normY{T_{\rho}[\zuUper]- T_0[\zuUper]}{\nu}, 
		\\
		\normY{\RRR_{\rho,3}[\zuUper]-\RRR_{0,3}[\zuUper]}{\nu}
		&\leq C\de\rho,
		\\
		\normY{\RRR_{\rho,4}[\zuUper]-\RRR_{0,4}[\zuUper]}{\nu}
		&\leq C\de\rho,
	\end{split}
\end{equation*}
with $T_{\rho}$ defined in~\eqref{def:operatorT}.
Therefore, it only remains to analyze $T_{\rho}[\zuUper]- T_0[\zuUper]$.
Indeed, applying the mean value theorem one sees that
\begin{align*}
	T_{\rho}[\zuUper]- T_0[\zuUper] &=
	V'(\lap+\po_{\rho,1}+\zuUper)-
	V'(\po_{\rho,1})-
	V'(\lap+\zuUper) +
	V'(0) \\
	&=
	\po_{\rho,1}
	\paren{\lap + \zuUper}
	\int_{[0,1]^2}
	V'''(s\lap+r\po_{\rho,1}+s\zuUper) dr ds.
\end{align*}
Then, since $\lap\in \Ycal{2\nu}$ and taking into account that $\po_{\rho,1}(\tau)=\de\rho\lapo(\tau)$ with ${\normY{\lapo}{0}\leq C}$ (see Proposition~\ref{proposition:periodicOrbit}),
one has that
$
\normY{T_{\rho}[\zuUper]- T_0[\zuUper]}{\nu} \leq C\de\rho.
$
This proves \eqref{proof:lemmaBoundsRRRcomparison} and, by \eqref{proof:estimateszuDzuU}, Proposition \ref{proposition:comparison1d2d} holds.

\appendix

\section{Lyapunov periodic orbits}
\label{appendix:periodicOrbits}

In this appendix we prove Proposition \ref{proposition:periodicOrbit}. 
Let us recall that,  by Proposition \ref{proposition:existenceFixedPoint}, the equilibrium point $L_3$ in the  coordinates $(\la,\La,x,y)$ (see \eqref{def:changeScaling}), is given by
\[
\Ltres(\de) = \paren{0,\de^2\LtresLa(\de),\de^3\Ltresx(\de),\de^3\Ltresy(\de)}^T,
\]
with
$\vabs{\LtresLa(\de)}, 
\vabs{\Ltresx(\de)}
\vabs{\Ltresy(\de)} \leq \cttTheoLtres$ for $\de>0$ small enough.
Using that one can write $H$ as  $H=H_0+H_1$, we have that 
\begin{equation}\label{eq:equationsL3}
\begin{aligned}
\partial_{\la}{H_1}(\Ltres(\de);\de)&=0,
\quad	&
\partial_{\La}{H_1}(\Ltres(\de);\de)&=	3\de^2\LtresLa(\de),
\\
\partial_x {H_1}(\Ltres(\de);\de)&=-\de\Ltresy(\de),
\quad	&
\partial_y {H_1}(\Ltres(\de);\de)&=-\de\Ltresx(\de).
\end{aligned}
\end{equation}
In addition, one can easily check that
\begin{align}\label{eq:equationsL3Hamiltonian}
	H(\Ltres(\de);\de) = -\frac12 -\frac32 \de^4 \LtresLa^2(\de) + \de^4 \Ltresx(\de) \Ltresy(\de) + H_1(\Ltres(\de);\de).
\end{align}
For $\rho>0$, we consider a polar symplectic change of coordinates $\phi_{\Lyap}:(\la,\Law,\phiA,\Iw)\to(\la,\La,x,y)$ given by
\begin{align}\label{def:changeLyap}
\La = \Law + \de^2\LtresLa(\de), 
\quad 
x = \sqrt{\rho^2+\Iw} e^{-i\phiA} + \de^3\Ltresx(\de),
\quad
y = \sqrt{\rho^2+\Iw} e^{i\phiA} + \de^3\Ltresy(\de).
\end{align}
%Therefore, to compute the Lyapunov periodic orbits surrounding $\Ltres(\de)$ in this set of coordinates is equivalent to compute $2\pi$-periodic solutions with respect to $\phiA$.
%
The Hamiltonian $H$ expressed in the coordinates $(\la,\Law,\phiA,\Iw)$ becomes  $H^{\Lyap} = H \circ \phi_{\Lyap}$, given by
\begin{align*}
	H^{\Lyap}(\la,\Law,\phiA,\Iw;\rho,\de) =& \,
	-\frac32 \Law^2
	+ V(\la) + \frac{\rho^2+\Iw}{\de^2}
	+ H_1(\phi_{\Lyap}(\la,\Law,\phiA,\Iw);\de)
	-3\de^2 \Law \LtresLa
	\\
	&+\de\sqrt{\rho^2 + \Iw}\paren{
		e^{-i\phiA}\Ltresy + e^{i\phiA}\Ltresx}
	-\frac32 \de^4 \LtresLa
	+ \de^4 \Ltresx \Ltresy,
\end{align*}
which, using \eqref{eq:equationsL3} and \eqref{eq:equationsL3Hamiltonian}, can be rewritten as
\begin{equation}\label{def:HLyap}
\begin{aligned}
	H^{\Lyap}(\la,\Law,\phiA,\Iw;\rho,\de) =& -\frac32 \Law^2
	+ V(\la) +\frac12 + \frac{\Iw}{\de^2}   
	+ H_1(\phi_{\Lyap}(\la,\Law,\phiA,\Iw);\de)
	\\
	&- H_1(\Ltres;\de)
	- D H_1(\Ltres;\de) \cdot
	\paren{\phi_{\Lyap}(\la,\Law,\phiA,I) - \Ltres}^T \\
	&+ \frac{\rho^2}{\de^2}
	+ H(\Ltres;\de).
\end{aligned}
\end{equation}

We are interested in proving the existence of a periodic orbit in the energy level $H^{\Lyap}=\frac{\rho^2}{\de^2}+H(\Ltres;\de)$.
To this end, in the following lemma, we first obtain an expression of $\Iw$ in terms of the other coordinates.
Let us denote by
$
B(\varsigma) = \claus{z \in \complexs \st \vabs{z}<\varsigma}
$, 
the open ball of radius $\varsigma$.

\begin{lemma}\label{lemma:hamiltonianLevelPO}
Fix $\banda, \varsigma_\la, \varsigma_\Law, \rhoPeriodicOrbit>0$.
There exists $\de_0>0$ such that, for all $\rho \in (0,\rhoPeriodicOrbit]$ and $\de \in (0,\de_0)$, there exists a function 
\[
\Iwh:B(\de \rho \varsigma_\la)\times B(\de\rho\varsigma_\Law) \times \torusC \to \complexs,
\]
such that
$
H^{\Lyap}(\la,\Law,\phiA,\Iwh(\la,\Law,\phiA);\rho,\de) = \frac{\rho^2}{\de^2}+H(\Ltres;\de).
$

Moreover, there exists a constant $C>0$ independent of $\rho$ and $\de$ such that
\begin{align*}
	\vabss{\Iwh(\la,J,\phiA;\de)} 
	&\leq C \de^4 \rho^2,
	\qquad &
	\vabss{\partial_{\la} \Iwh(\la,J,\phiA;\de)} &\leq C \de^3 \rho,
	\\
	\vabss{\partial_{J} \Iwh(\la,J,\phiA;\de)} &\leq C \de^3 \rho,
	\qquad &
	\vabss{\partial_{\phiA}  \Iwh(\la,J,\phiA;\de)} &\leq C\de^4\rho^2.
\end{align*}
\end{lemma}

\begin{proof}
One has that the function $\Iwh$ must satisfy the equation $\Iwh=F[\Iwh]$ with
\begin{align*}
F[\Iw](\la,\Law,\phiA) 
=& 
\, \de^2 H^{\Lyap}(\la,\Law,\phiA,\Iw;\rho,\de) - I
- \rho^2 - \de^2 H(\Ltres;\de)
\\
=& \, \de^2 \bigg[ \frac32 \Law^2
+ V(\la)  +\frac12 
+ H_1(\phi_{\Lyap}(\la,\Law,\phiA,\Iw);\de)
\\
&- H_1(\Ltres;\de)
- D H_1(\Ltres;\de) \cdot
\paren{\phi_{\Lyap}(\la,\Law,\phiA,I) - \Ltres}^T 
\bigg].
\end{align*}
Let $(\la, \Law, \phiA) \in B(\de \rho \varsigma_\la)\times B(\de\rho\varsigma_\Law) \times \torusC$.
Then, using the estimates of $D^2 H_1$ in Proposition~\ref{proposition:HamiltonianScaling}, one has that
\[
\vabs{F[0](\la,\Law,\phiA)} \leq C \de^2 \rho^2.
\]
In addition, for functions $\iota_1, \iota_2: B(\de \rho \varsigma_\la)\times B(\de\rho\varsigma_\Law) \times \torusC \to \complexs$ such that $\vabs{\iota_1(\la,\Law,\phiA)},\vabs{\iota_2(\la,\Law,\phiA)} \leq C\de^2\rho^2$, by the estimates of the third derivatives of $H_1$ in Proposition \ref{proposition:HamiltonianScaling} and the mean value theorem, one has that
\[
\vabs{F[\iota_1](\la,\Law,\phiA)-F[\iota_2](\la,\Law,\phiA)} 
\leq C\de^3\rho \vabs{\iota_1(\la,\Law,\phiA)-\iota_2(\la,\Law,\phiA)}
\leq C\de_0^3\rhoPeriodicOrbit \vabs{\iota_1(\la,\Law,\phiA)-\iota_2(\la,\Law,\phiA)}.
\]
Then, taking $\de_0$ small enough and applying the fixed point theorem, one obtains the existence of the function $\Iwh$ and its corresponding bounds.
The bounds for the derivatives of $\Iwh$ are a direct consequence of Cauchy estimates.
\end{proof}

By Lemma \ref{lemma:hamiltonianLevelPO}, the Hamiltonian system on the energy level
$H^{\Lyap}=\frac{\rho^2}{\de^2}+H(\Ltres;\de)$ is of the form
%
%Since we are interested in periodic orbits in the energy level $H^{\Lyap}=\frac{\rho^2}{\de^2}+H(\Ltres;\de)$, by Lemma \ref{lemma:hamiltonianLevelPO}, we can skip the $I$ variable in the analysis, and deal with the ``reduced'' system for $(\la,\Law,\phiA)$.
%Taking into account the expression of $H^{\Lyap}$ in \eqref{def:HLyap}, this reduced system is of the form
%
\begin{align}\label{eq:systemEquationsHalfWayPO}
\dot{\la} = -3 \Law + f_1(\la,\Law,\phiA),
\qquad
\dot{\Law} = -\frac78 \la + f_2(\la,\Law,\phiA),
\qquad
\dot{\phiA} = \frac1{\de^2} + g(\la,\Law,\phiA),
\end{align}
where, denoting $\Iwh=\Iwh(\la,\Law,\phiA)$ and using the expression of $H^{\Lyap}$ in \eqref{def:HLyap} and that $V''(0)=-\frac78$,
\begin{equation}\label{def:auxFunctionsPO}
\begin{aligned}
f_1(\la,\Law,\phiA) &=
\partial_{\La}{H_1}\paren{ \phi_{\Lyap}(\la,\Law,\phiA,\Iwh);\de} - \partial_{\La}{H_1}(\Ltres(\de);\de),
\\
f_2(\la,\Law,\phiA) &=
- V'(\la) + V''(0)\la
- \partial_{\la}{H_1}\paren{ \phi_{\Lyap}(\la,\Law,\phiA,\Iwh);\de}
+ \partial_{\la}{H_1}(\Ltres(\de);\de), 
\\
g(\la,\Law,\phiA) &= 
\frac{e^{-i\phiA}}{2\sqrt{\rho^2+\Iwh}} 
\paren{\partial_x H_1\paren{ \phi_{\Lyap}(\la,\Law,\phiA,\Iwh);\de}
	- \partial_x{H_1}(\Ltres(\de);\de)} \\
&+ \frac{e^{i\phiA}}{2\sqrt{\rho^2+\Iwh}} 
\paren{\partial_y{H_1}\paren{ \phi_{\Lyap}(\la,\Law,\phiA,\Iwh);\de}
	- \partial_y{H_1}(\Ltres(\de);\de)}.
\end{aligned}
\end{equation}
We look for the periodic orbit of the system \eqref{eq:systemEquationsHalfWayPO} as a graph over $\phiA$ provided $\dot{\phiA}\neq 0$ (which will be true on the periodic orbits).
In other words, we look for periodic functions
\[
w = (w_{\la},w_{\Law}): \torusC \to \complexs^2,
\qquad
w=w(\phiA),
\]
satisfying the invariance equation
$
	\LL w = \RRR[w],
$
with
\begin{equation}\label{def:operatorsPO}
\begin{aligned}	
	\LL w &= (\partial_{\phiA}-\de^2\AAA)w,
	\qquad
	\AAA = \begin{pmatrix}
	0 & -3 \\
	-\frac78 & 0
	\end{pmatrix},
	\\[0.5em]
	\RRR[w](\phiA) &= 
	\de^2 \paren{
	\frac{
	\AAA w + 	
	f(w_{\la}(\phiA),w_{\Law}(\phiA),\phiA)}{1+
	\de^2g(w_{\la}(\phiA),w_{\Law}(\phiA),\phiA)}
	-\AAA w} \quad \text{where}\quad f=(f_1,f_2).
\end{aligned}
\end{equation}
%where $f=(f_1,f_2)$.
% and $g$ are given in \eqref{def:auxFunctionsPO}.
Let us consider the Banach space
\begin{align*}
	\ZcalR =& \claus{h: \torusC \to \complexs\st
		h \text{ analytic} , \, 
		\normZR{h} := 
		\sup_{\phiA \in \torusC} \vabs{h(\phiA)}<+\infty	
	},
\end{align*}
and the space $\ZcalR^2$ endowed with the product norm
$\normZRprod{h}= \normZR{h_1} + \normZR{h_2}$.

\begin{proposition}\label{proposition:periodicOrbitTechnical}
There exist $\rhoPeriodicOrbit, \de_0,\cttPeriodicOrbitTech>0$ such that, for $\rho \in (0,\rhoPeriodicOrbit]$ and $\de \in (0,\de_0)$, there exists a solution of $\LL w = \RRR[w]$ belonging to $\ZcalR^2$ and satisfying
\begin{align*}
	\normZRprod{w} \leq \cttPeriodicOrbitTech \de \rho.
\end{align*}
\end{proposition}

To prove Proposition \ref{proposition:periodicOrbitTechnical} we first study the right-inverse of the operator $\LL=\partial_{\phiA}-\de^2\AAA$ in $\ZcalR^2$.
%
%We proceed by computing a right-inverse of the operator $\LL=\partial_{\phiA}-\AAA$. 
%
First, notice that
\begin{align*}
	\AAA = \PP \DD \PP^{-1}
	\quad	\text{where} \quad
	\DD = \begin{pmatrix}
		\nu & 0 \\
		0 & -\nu
	\end{pmatrix},
	\quad
	\PP = \begin{pmatrix}
		3 & 3 \\
		-\nu & \nu
	\end{pmatrix},
	\quad
	\nu=\sqrt{\frac{21}8}.
\end{align*}

\begin{lemma}\label{lemma:boundsGGPO}
% We consider the linear operator 
The operator
$
{\GG: \ZcalR^2 \to \ZcalR^2}
$
defined as
\begin{equation}\label{def:operatorGGPO}
	\begin{aligned}
		\GG[h](\phiA) =& \,
		\PP e^{\phiA\de^2\DD} (e^{-2\pi\de^2\DD}-\Id)^{-1} \int_0^{2\pi} e^{-\tht\de^2\DD} \PP^{-1} h(\tht)d\tht \\
		&+ \PP e^{\phiA\de^2\DD}\int_0^{\phiA} e^{-\tht \de^2\DD}\PP^{-1} h(\tht)d\tht,
	\end{aligned}
\end{equation}	
is a right-inverse of the operator $\LL$ given in \eqref{def:operatorsPO}.
In addition, there exists $C>0$ such that, for $\de \in (0,1)$,
\[
\normZRprod{\GG[h]} \leq \frac{C}{\de^2} \normZRprod{h},
\qquad \text{for }
h \in \ZcalR^2.
\]
\end{lemma}

\begin{proof}	
If $w$ is a solution of $\LL[w]=h$, it must exist $K_0 \in \reals^2$ such that
\[
w(\phiA) = \PP e^{\phiA \de^2\DD} \boxClaus{K_0 + \int_0^{\phiA} e^{-\tht\de^2\DD} \PP^{-1}h(\tht) d\tht}.
\]	
Then, imposing that $w$ is $2\pi$-periodic, one obtains~\eqref{def:operatorGGPO}.
The estimates for the operator are straightforward taking into account that	
\begin{align*}
	\norms{(e^{-2\pi\de^2\DD}-\Id)^{-1}}
	\leq
	(1-
	\norms{e^{-2\pi\de^2\DD}-\Id})^{-1}
	\leq
	\frac{C}{\de^2}.
\end{align*}
\end{proof}

For $\varsigma>0$, we denote 
$
\BB(\varsigma) = \claus{h \in \ZcalR^2 \st \normZRprod{h} \leq \varsigma}.
$

\begin{lemma}\label{lemma:boundsRRPO}
Fix constants $\rhoPeriodicOrbit, \varsigma>0$. 
Then, there exist $\de_0, C>0$ such that,
for $\rho \in (0,\rho_0]$, $\de\in(0,\de_0)$ and $h\in \BB(\varsigma\de\rho)$, the function $\RRR$  in \eqref{def:operatorsPO} satisfies
	\begin{align*}
		\normZR{\RRR_1[h]} \leq C \de^5 \rho,
		\qquad
		\normZR{\RRR_2[h]} \leq C \de^3 \rho
	\end{align*}
	and
	\begin{align*}
		\normZR{\partial_1 \RRR_1[h]} \leq C \de^4,
		\quad
		\normZR{\partial_2 \RRR_1[h]} \leq C \de^4,
		\quad
		\normZR{\partial_1 \RRR_2[h]} \leq C \de^3,
		\quad
		\normZR{\partial_2 \RRR_2[h]} \leq C \de^4.
	\end{align*}
\end{lemma}

\begin{proof}
Let $h=(h_1, h_2) \in  \BB(\varsigma\de\rho)$ and $\phiA \in \torusC$.
For $s\in[0,1]$, we denote
\begin{align*}
z^s(\phiA) = s \, \phi_{\Lyap}\paren{h_{1}(\phiA),h_{2}(\phiA),\phiA, \Iwh(h(\phiA)) }
+ (1-s) \Ltres(\de).
\end{align*}
We notice that, by the definition in \eqref{def:changeLyap} of $\phi_{\Lyap}$,
\begin{align*}
	z^1(\phiA)-z^0(\phiA) = \paren{h_{1}(\phiA),
	h_{2}(\phiA), \sqrt{\rho^2+\Iwh(h(\phiA))} e^{-i\phiA},
	\sqrt{\rho^2+\Iwh(h(\phiA))} e^{i\phiA}
	}^T.
\end{align*}
We recall that $f_1=\partial_{\La} H_1(\phi_{\Lyap})-\partial_{\La}H_1(\Ltres)$ (see \eqref{def:auxFunctionsPO}) 
and then, by the mean value theorem and the estimates in Proposition \ref{proposition:HamiltonianScaling} and Lemma \ref{lemma:hamiltonianLevelPO}, 
\begin{equation}\label{proof:estimatef1PO}
\begin{split}
	\vabs{f_1(h(\phiA),\phiA)}
	\leq \sup_{s \in [0,1]} \Big\{
	\vabs{\partial_{\La \la} H_1\paren{z^s(\phiA)}}
	\vabs{h_{1}(\phiA)}
	+
	\vabs{\partial^2_{\La} H_1\paren{z^s(\phiA)}}
	\vabs{h_{2}(\phiA)} \\
	+ 
	\big(
	\vabs{\partial_{\La x} 
		H_1\paren{z^s(\phiA)}} +
	\vabs{\partial_{\La y} 
		H_1\paren{z^s(\phiA)}}\big) \,
	\vabss{\rho^2+\Iwh(h(\phiA))}^{\frac12}
	\Big\} \leq C\de^3 \rho.
\end{split}
\end{equation}
	Analogously,
	\begin{align}\label{proof:estimatef2gPO}
		\vabs{f_2(h(\phiA),\phiA)} 
		\leq C \de \rho,
		\qquad
		\vabs{g(h(\phiA),\phiA)} 
		\leq C \de^2.
	\end{align}
To obtain estimates for the derivatives of $f_1$, $f_2$ and $g$, note that
\begin{align*}
\partial_{\la}{f_1}(h(\phiA),\phiA) 
=& \,
\partial_{\la \La}{H_1}\paren{z^1(\phiA)}
+ 
\frac{\partial_{\la} \Iwh }{2\sqrt{\rho^2+\Iwh}}
\boxClaus{
	e^{-i\phiA} \partial_{\La x}{H_1}\paren{z^1(\phiA)}
	+e^{i\phiA}\partial_{\La y}{H_1}\paren{z^1(\phiA)}},
\\
\partial_J{f_1}(h(\phiA),\phiA) 
=& \,
\partial^2_{\La}{H_1}\paren{z^1(\phiA)}
+ 
\frac{\partial_{J} \Iwh }{2\sqrt{\rho^2+\Iwh}}
\boxClaus{
e^{-i\phiA} \partial_{\La x}{H_1}\paren{z^1(\phiA)}
+e^{i\phiA}\partial_{\La y}{H_1} \paren{z^1(\phiA)}},
\end{align*}
where $\Iwh=\Iwh(h(\phiA))$.
Then, using the estimates in Proposition \ref{proposition:HamiltonianScaling} and Lemma \ref{lemma:hamiltonianLevelPO}, 
	\begin{align}\label{proof:estimatef1DerPO}
		\vabs{\partial_{\la}{f_1}(h(\phiA),\phiA)}
		\leq C \de^2, 
		\qquad
		\vabs{\partial_{J}{f_1}(h(\phiA),\phiA)}
		\leq C \de^2.
	\end{align}
	Analogously,
	\begin{equation}\label{proof:estimatef2gDerPO}
	\begin{aligned}
		\vabs{\partial_{\la}{f_2}(h(\phiA),\phiA)}
		&\leq C \de, &
		\quad
		\vabs{\partial_{J}{f_2}(h(\phiA),\phiA)}
		&\leq C \de^2,
		\\
		\vabs{\partial_{\la}{g}(h(\phiA),\phiA)}
		&\leq \frac{C \de}{\rho}, &
		\quad
		\vabs{\partial_{J}{g}(h(\phiA),\phiA)}
		&\leq \frac{C \de^3}{\rho}.
	\end{aligned}
	\end{equation}
	Finally, joining the just obtained bounds with the definition of the operator $\RRR$ in \eqref{def:operatorsPO}, we obtain the statement of the lemma.
\end{proof}

\begin{proof}[Proof of Proposition \ref{proposition:periodicOrbitTechnical}]
A fixed point of $w=\FF[w]$ with $\FF=\GG \circ \RRR$ is a periodic solution of $\LL w = \RRR[w]$.
By Lemmas \ref{lemma:boundsGGPO} and \ref{lemma:boundsRRPO}, there exists  $\cttPeriodicOrbitTech>0$ such that
 \begin{align}\label{proof:firstIterationPO}
 	\normZRprod{\FF[0]} \leq 
 	\frac{C}{\de^2}\paren{\normZR{\RRR_1[0]} + \normZR{\RRR_2[0]}}
 	\leq
 	\frac{\cttPeriodicOrbitTech}2  \de \rho.
 \end{align}
Moreover, for $h,\hat{h} \in \BB(\cttPeriodicOrbitTech\de\rho)$, by the mean value theorem,
\begin{align*}
	\normZRprod{\RRR[h]-\RRR[\hat{h}]} 
	\leq
	\sup_{s \in [0,1]} \bigg[&
		\normZRprod{D\RRR[(1-s)h + s\hat{h}]
		(h-\hat{h}) }\bigg].
\end{align*}
Thus, by Lemmas \ref{lemma:boundsGGPO} and \ref{lemma:boundsRRPO},
\begin{equation}\label{proof:FcontractivePO}
\begin{aligned}	
	\normZRprod{\FF[h]-\FF[\hat{h}]} &\leq
	\frac{C}{\de^2} 
	\normZRprod{\RRR[h]-\RRR[\hat{h}]} 
	\leq C\de \normZRprod{h-\hat{h}}.
\end{aligned}
\end{equation}
Then, if $\de$ is small enough, the operator $\FF:\BB(\cttPeriodicOrbitTech\de\rho) \to \BB(\cttPeriodicOrbitTech\de\rho)$ is well defined and contractive and, as a consequence, it has a fixed point $w \in \BB(\cttPeriodicOrbitTech\de\rho)$.
\end{proof}

\begin{proof}[End of the proof of Proposition \ref{proposition:periodicOrbit}]
Let $w(\phiA)=(w_{\la}(\phiA),w_J(\phiA))$ be the solution of $\LL w=\RRR[w]$ given by Proposition \ref{proposition:periodicOrbitTechnical} and introduce $w_{\Iw}(\phiA)=\Iwh(w(\phiA),\phiA)$ as given in Lemma \ref{lemma:hamiltonianLevelPO}.
Then, the curve $(w_{\la}(\phiA),w_J(\phiA),\phiA,w_{\Iw}(\phiA))$ is a graph parametrization of the Lyapunov periodic solution
 in the energy level $H^{\Lyap}=\frac{\rho^2}{\de^2} + H(\Ltres)$.
However, $\dot{\phiA}=\partial_t {\phiA} = \frac1{\de^2} + g(w({\phiA}),{\phiA})$. Then, to complete the proof of Proposition \ref{proposition:periodicOrbit}, we look for a reparametrization $\phiA=\wh{\phiA}(\tau)$ and a constant $\omegar$ such that $\dot{\tau}=\frac{\omegar}{\de^2}$.
Moreover, we impose $\phiA(t)|_{t=0}=0$ and therefore $\wh{\phiA}(2\pi)=2\pi$.
Then, $\wh{\phiA}$  must satisfy that
\begin{align*}%\label{proof:phiAtau}
	\partial_{\tau} \wh{\phiA} = \frac{1+\de^2g(w(\wh\phiA),\wh\phiA)}{\omegar}
	\qquad \text{and} \qquad
	\wh\phiA(2\pi) = 2\pi.
\end{align*}
Notice that, by \eqref{proof:estimatef2gPO} and for $\de$ small enough, one has that $\partial_{\tau} \wh\phiA \neq 0$. 
Then, its inverse $\tau \equiv \wh\tau(\phiA)$ satisfies that
\begin{align*}
\partial_{\phiA} \wh\tau = \frac{\omegar}{1+\de^2 g(w(\phiA),\phiA)}
\qquad \text{and} \qquad
\wh\tau(2\pi) = 2\pi.
\end{align*}
These conditions give definitions for the function $\wh{\tau}(\phiA)$ and the constant $\omegar$,
\begin{align*}
\wh\tau(\phiA) = \omegar \int_0^{\phiA}
 \frac{d\eta}{1+\de^2 g(w(\eta),\eta)}
\qquad \text{and} \qquad
\omegar =  \frac{2\pi}{\int_0^{2\pi}
	\frac{d\phiA}{1+\de^2 g(w(\phiA),\phiA)}}.
\end{align*}
We notice that $\wh{\tau}(\phiA+2\pi)=2\pi+\wh{\tau}(\phiA)$.
By the estimate for $g$ in \eqref{proof:estimatef2gPO}, we obtain
\begin{align}\label{proof:estimatesPhiA}
\vabs{\omegar-1} \leq C\de^4, 
\qquad
\vabs{\wh\tau(\phiA)-\phiA} \leq C\de^4,
\qquad
\vabs{\wh\phiA(\tau)-\tau} \leq C\de^4.
\end{align}
Then, for $\tau\in \torusC$, the curve
\begin{align*}
&\po_{\rho}(\tau;\de) 
= 
\phi_{\Lyap}\big(w_{\la}(\wh\phiA(\tau)),
w_{\Law}(\wh\phiA(\tau)),\wh\phiA(\tau),
w_{\Iw}(\wh\phiA(\tau))\big) 
%\\
%	&=
%\Ltres(\de) + \paren{w_{\la}(\wh\phiA(\tau)), w_{J}(\wh\phiA(\tau)),
%		\sqrt{\rho^2+w_{\Iw}(\wh\phiA(\tau))} e^{-i\wh\phiA(\tau)},
%		\sqrt{\rho^2+w_{\Iw}(\wh\phiA(\tau))} e^{i\wh\phiA(\tau)}},
\end{align*} 
is a real-analytic and $2\pi$-periodic solution of the Hamiltonian system given by the Hamiltonian $H$ in \eqref{def:hamiltonianScaling} and it belongs to the energy level $H=\frac{\rho^2}{\de^2}+H(\Ltres)$.
In addition, the functions in \eqref{def:poSplitting} are given by
\begin{align*}
	\lapo(\tau) &= \frac{w_{\la}(\wh{\phiA}(\tau))}{\de\rho},
	&
	\xpo(\tau) &= \frac{\sqrt{\rho^2+w_{\Iw}(\wh{\phiA}(\tau))}e^{-i\wh{\phiA}(\tau)}-\rho e^{-i\tau}}{\de\rho},
	\\
	\Lapo(\tau) &=  
	\frac{w_{J}(\wh{\phiA}(\tau))}{\de\rho},
	&
	\ypo(\tau) &= \frac{\sqrt{\rho^2+w_{\Iw}(\wh{\phiA}(\tau))}e^{i\wh{\phiA}(\tau)}-\rho e^{i\tau}}{\de\rho},
\end{align*}
and, by Lemma \ref{lemma:hamiltonianLevelPO}, Proposition \ref{proposition:periodicOrbitTechnical} and \eqref{proof:estimatesPhiA}, satisfy that
$\vabs{\lapo(\tau)}, \vabs{\Lapo(\tau)}\leq C$
and
$\vabs{\xpo(\tau)}, \vabs{\ypo(\tau)}\leq C\de^3$.
\end{proof}

\section{Difference between the invariant manifolds of \texorpdfstring{$L_3$}{L3} on \texorpdfstring{$\Sigma_0$}{Sigma0}}
\label{appendix:changeSectionL3}

In this appendix we prove Corollary \ref{mainTheoremDistCorollary}, relying on the results  in Sections \ref{subsection:existencePerturbed} and
\ref{section:proofExistence}.
Let us consider the real-analytic time parametrizations $\zuU$ and $\zsU$ of the unstable and stable manifolds $\WW^{\unstable,+}(\Ltres)$ and $\WW^{\stable,+}(\Ltres)$ defined in  Corollary~\ref{corollary:existencia1D}.
Notice that, for $u\in\DmigU\cap\DmigS$ (see \eqref{def:dominiMig}), they satisfy
\begin{align}\label{proof:zusMenysSeparatrix}
\vabs{\zuU(u)-\s_{\pend}(u)-\de^2\LtresLa}
\leq C \de,
\qquad
\vabs{\zsU(u)-\s_{\pend}(u)-\de^2 \LtresLa}
\leq C  \de ,
\end{align}
where $\s_{\pend}=(\lap,\Lap,0,0)^T$ is given in \eqref{eq:separatrixParametrizationB}.
Moreover, $\zuU(0), \zsU(0) \in \claus{\La=\de^2\LtresLa}$ and, since $\zuU$ and $\zsU$ satisfy equation~\eqref{eq:existence2D} and are independent of $\tau$, for ${\zdU=(\la^{\diamond},\La^{\diamond},x^{\diamond},y^{\diamond})}$, $\diamond=\unstable,\stable$, one has that
\begin{equation}\label{proof:canviSeccioEquacions}
\begin{aligned}
	\Dt{\la^{\diamond}}{u}  &= -3\La^{\diamond} + \partial_{\La} H_1(\zuU;\de),
&
\Dt{x^{\diamond}}{u}  &= \frac{i}{\de^2}x^{\diamond} + i\partial_{y} H_1(\zuU;\de)
\\
\Dt{\La^{\diamond}}{u}  &= -V'(\la^{\diamond}) - \partial_{\la} H_1(\zuU;\de),
&
\Dt{y^{\diamond}}{u}  &= -\frac{i}{\de^2}y^{\diamond} -i\partial_{x} H_1(\zuU;\de).
\end{aligned}
\end{equation}

%\begin{figure}
%	\centering
%	\begin{overpic}[scale=1]{canviSeccio.png}
%		\put(25,23){\color{blue} $\s_{\pend}$}
%		\put(9,20){\color{red} $\zuU$}
%		\put(13,53){\color{green} $\zsU$}
%		\put(-1,76){$\La$}
%		\put(95,30){$\la$}
%		\put(59,74){\color{myOrange} $S_{\la}(\la_*)$}
%		\put(78,30){\color{myOrange} $S_{\La}(0)$}
%	\end{overpic}
%	\caption{Omplir}
%	\label{fig:canviSeccio}
%\end{figure}

Fix $\la_* \in \paren{\frac{2\pi}3,\la_0}$, (see \eqref{def:la0}).
By Proposition \ref{proposition:singularities}, there exists $u_*>0$ such that $\la_* = \lap(u_*)$. 
Therefore, by \eqref{proof:zusMenysSeparatrix} and for $\de>0$ small enough, there exist $T^{\unstable}, T^{\stable}=u_* + \OO(\de)$ such that
$
\zuU(T^{\unstable}), \zsU(T^{\stable})  \in \claus{\la=\la_*, \La>0}.
$
Moreover, by Theorem~\ref{mainTheoremDist},
\begin{align}\label{proof:differenceSectionla}
	\zuU(T^{\unstable}) - \zsU(T^{\stable})
	= \sqrt[6]{2} \de^{\frac{1}{3}} e^{-\frac{A}{\de^2}}
	\boxClaus{ 
	\paren{0,0,\conj{\CInn},\CInn}^T
	+ \OO_{\de}},
\end{align}
where $\OO_{\de} = (0,\OO(\de),\OO(\vabs{\log \de}^{-1}),\OO(\vabs{\log \de}^{-1}))^T$.

To prove Corollary~\ref{mainTheoremDistCorollary}, we  deduce the difference $\zuU(0)-\zsU(0)$ from \eqref{proof:differenceSectionla}.
To this end, we  define $\D(u)=\zuU(u)-\zsU(u)$, for $u \in [0,T^{\unstable}]$.
It is clear that, by \eqref{proof:canviSeccioEquacions}, the function $\D(u)$ satisfies the linear equation
\begin{align*}
	\Dt{}{u}\D(u) = (M_0(u) + M_1(u))\D(u),
\end{align*}
with
\begin{align*}
	M_0(u) &=
	\begin{pmatrix}
		0 & -3 & 0 & 0 \\
		-V''(\lap(u)) & 0 & 0 & 0 \\
		0 & 0 & \frac{i}{\de^2} & 0 \\
		0 & 0 & 0 & -\frac{i}{\de^2}
	\end{pmatrix},
	\\
	M_1(u) &= 
	\begin{pmatrix}
		0 & 0 & 0 & 0 \\
		m(u) & 0 & 0 & 0 \\
		0 & 0 & 0 & 0 \\
		0 & 0 & 0 & 0
	\end{pmatrix}
	+
	\int_0^1 \mathbf{J} D^2 H_1\paren{\varsigma\zuU(u)+
		(1-\varsigma)\zsU(u)} d\varsigma,
	\\
	m(u) &= V''(\lap(u)) - \int_0^1 V''\paren{\varsigma\la^{\unstable}(u)+
		(1-\varsigma)\la^{\stable}(u)} d\varsigma,
\end{align*}
where $\mathbf{J}$ is the symplectic matrix associated with the form $d\la\wedge d\la + idx \wedge dy$.
%
%From the definition of the potential $V$ in \eqref{def:potentialV} we deduce that $m_*>0$.
%
Moreover, from Proposition~\ref{proposition:HamiltonianScaling} and Corollary~\ref{corollary:existencia1D}, we deduce that $\vabs{M_1(u)} \leq C\de$,  for $u \in [0,T^{\unstable}]$.
Let $\Phi(u)$ be the fundamental matrix of the differential equation $\frac{d}{du}\Phi(u) = M_0(u) \Phi(u)$ given in Lemma~\ref{lemma:fundamentalMatrixFlow}, which satisfies $\Phi(0)=\Id$.
Then,
\begin{align*}
\D(u) = \Phi(u) \boxClaus{
\Phi^{-1}(T^{\unstable}) \D(T^{\unstable})
+
\int_{T^{\unstable}}^u \Phi^{-1}(\sigma) M_1(\sigma) \D(\sigma) d\sigma}.
\end{align*}
On one hand, using Gronwall's Lemma, one has that $\vabs{\D(u)} \leq C\vabs{\D(T^{\unstable})}$ for $u \in [0,T^{\unstable}]$ and, on the other hand
\begin{align}\label{proof:NoSeQuinaEtiquetaPosar1}
\vabs{\D(0)-
\Phi^{-1}(T^{\unstable}) \D(T^{\unstable})}
\leq 
C\de T^{\unstable}\vabs{\D(T^{\unstable})}.
\end{align}
Thus, to obtain an asymptotic formula for $\D(0)$, we need to compute $\D(T^{\unstable})$.
We write
\begin{align}\label{proof:NoSeQuinaEtiquetaPosar2}
\D(T^{\unstable}) 
=
\zuU(T^{\unstable}) - \zsU(T^{\stable}) + \zsU(T^{\stable}) - \zsU(T^{\unstable}).
\end{align}
Since the difference $\zuU(T^{\unstable}) - \zsU(T^{\stable})$ is given by \eqref{proof:differenceSectionla}, we only need to analyze the term $\zsU(T^{\stable}) - \zsU(T^{\unstable})$. To do so, we first bound $T^u-T^s$.
%
%We denote $\D(s)=\D^{a}(s)+\D^b(s)$ where
%\[
%\D^a(s) = \zuU((1-s)T^{\unstable})-\zsU((1-s)T^{\unstable}),
%\qquad
%\D^b(s) =
%\zsU((1-s)T^{\unstable})-\zsU((1-s)T^{\stable}).
%\]
%
Since $\zsU=(\la^{\stable},\La^{\stable},x^{\stable},y^{\stable})$ satisfies equation~\eqref{proof:canviSeccioEquacions} and using the mean value theorem, we obtain that
\begin{align*}
	T^{\unstable}-T^{\stable} = \frac{\La^{\unstable}(T^{\stable})-\La^{\unstable}(T^{\unstable})}{V'(\lap(u_*)) + \beta(T^{\unstable},T^{\stable})},
\end{align*}
where, denoting $T (r) = r T^{\unstable} + (1-r)T^{\stable}$, the function $\beta$ is given by
\begin{align*}
\beta(T^{\unstable},T^{\stable}) = 	\int_0^1 
\boxClaus{V'(\la^{\unstable}(T(r)))
- V'(\lap(u_*))} dr
+
\int_0^1 \partial_{\la} H_1(z^{\unstable}(T(r))) dr.
\end{align*}
Notice that $V'(\lap(u_*)) = V'(\la_*) \neq 0$ (see \eqref{def:potentialV}).
%\begin{align*}
%	\Ladp(u) = -V'(\lap(u)) 
%	= \sin\la\paren{1-\paren{2+2\cos\la}^{-\frac32}}
%	\neq 0.
%\end{align*}
%
Moreover, since ${T^{\unstable}, T^{\stable} = u_* + \OO(\de)}$, by \eqref{proof:zusMenysSeparatrix} and the estimates in Proposition \ref{proposition:HamiltonianScaling}, one can see that
$\vabs{\beta(T^{\unstable},T^{\stable})} \leq C \de$.
Therefore, one has that 
$
\vabs{T^{\unstable}-T^{\stable}} \leq C \de^{\frac43} e^{-\frac{A}{\de^2}}
$. 
Then, since
\begin{align*}
	\zsU(T^{\stable}) - \zsU(T^{\unstable}) = (T^{\stable}-T^{\unstable})
	\int_0^1 \partial_u \zsU(r T^{\unstable} + (1-r)T^{\stable}) d r,
\end{align*}
we have
%\begin{align*}%\label{proof:Db}
$\zsU(T^{\stable}) - \zsU(T^{\unstable}) = \OO \paren{\de^{\frac43}e^{-\frac{A}{\de^2}}}$.
%\end{align*}
Therefore, by \eqref{proof:differenceSectionla} and \eqref{proof:NoSeQuinaEtiquetaPosar2}
\begin{align}\label{proof:NoSeQuinaEtiquetaPosar3}
\D(T^{\unstable}) 
=
\sqrt[6]{2} \de^{\frac{1}{3}} e^{-\frac{A}{\de^2}}
\boxClaus{ 
	\paren{0,0,\conj{\CInn},\CInn}^T
+ \wt{\OO_{\de}}},
\end{align}
where $\wt{\OO_{\de}}=(\OO(\de),\OO(\de),\OO(\vabs{\log \de}^{-1}),\OO(\vabs{\log \de}^{-1}))^T$.
Lastly, joining the results in \eqref{proof:NoSeQuinaEtiquetaPosar1} and \eqref{proof:NoSeQuinaEtiquetaPosar3}, we obtain
\begin{align*}
\vabs{\D(0)} = 
\sqrt[6]{2} \de^{\frac{1}{3}} e^{-\frac{A}{\de^2}}
\boxClaus{ 
	\vabss{\Phi^{-1}(T^{\unstable})
	\paren{0,0,\conj{\CInn},\CInn}^T}
	+ \wt{\OO_{\de}}}.
\end{align*}
Then, applying the expression of the fundamental matrix $\Phi$ given in Lemma~\ref{lemma:fundamentalMatrixFlow}, we obtain the statement of the corollary.

\bibliographystyle{alpha}
\addcontentsline{toc}{section}{Bibliography}
{\small\bibliography{biblio}}

\newcommand{\etalchar}[1]{$^{#1}$}
\begin{thebibliography}{BCGG23}

\bibitem[Ale76]{Alekseev1976}
V.~M. Alekseev.
\newblock Quasi-random oscillations and qualitative problems of celestial
  mechanics.
\newblock {\em STIA}, 77:212--341, 1976.

\bibitem[Ari02]{Ari02}
G.~Arioli.
\newblock Periodic orbits, symbolic dynamics and topological entropy for the
  restricted 3-body problem.
\newblock {\em Comm. Math. Phys.}, 231(1):1--24, 2002.

\bibitem[Arn63]{Arnold63}
V.~I. Arnol'd.
\newblock Small denominators and problems of stability of motion in classical
  and {Celestial} {Mechanics}.
\newblock {\em Uspehi Mat. Nauk}, 18(6 (114)):91--192, 1963.

\bibitem[BCGG23]{BCGG23}
I.~Baldom\'a, M.~Capi\'nski, M.~Giralt, and M.~Guardia.
\newblock Breakdown of homoclinic orbits to ${L}_3$: {N}onvanishing of the
  {S}tokes constant.
\newblock {\em Preprint arXiv:2312.13138}, 2023.

\bibitem[BFGS12]{BFGS12}
I.~Baldom\'{a}, E.~Fontich, M.~Guardia, and T.~M. Seara.
\newblock Exponentially small splitting of separatrices beyond {M}elnikov
  analysis: rigorous results.
\newblock {\em J. Differential Equations}, 253(12):3304--3439, 2012.

\bibitem[BFPC13]{BFPC13}
A.~Bengochea, M.~Falconi, and E.~P{\'e}rez-Chavela.
\newblock Horseshoe periodic orbits with one symmetry in the general planar
  three-body problem.
\newblock {\em Discrete \& Continuous Dynamical Systems}, 33(3):987, 2013.

\bibitem[BFPS22]{BergerFP22}
P.~Berger, A.~Florio, and D.~Peralta-Salas.
\newblock Steady euler flows on $\mathbb{R}^3$ with wild and universal
  dynamics.
\newblock {\em Preprint arXiv:2202.02848}, 2022.

\bibitem[BGG22]{articleInner}
I.~Baldom{\'{a}}, M.~Giralt, and M.~Guardia.
\newblock {Breakdown} of homoclinic orbits to {L3} in the {RPC3BP} {(I)}.
  {Complex} singularities and the inner equation.
\newblock {\em {Advances} in {Mathematics}}, 408:108562 (64), 2022.

\bibitem[BGG23]{articleOuter}
I.~Baldom\'{a}, M.~Giralt, and M.~Guardia.
\newblock Breakdown of homoclinic orbits to {$L_3$} in the {RPC}3{BP} ({II}).
  {A}n asymptotic formula.
\newblock {\em Adv. Math.}, 430:109218(72), 2023.

\bibitem[BIS20]{BIS20}
I.~Baldom{\'{a}}, S.~Ib{\'{a}}{\~{n}}ez, and T.~M. Seara.
\newblock Hopf-{Zero} singularities truly unfold chaos.
\newblock {\em Communications in Nonlinear Science and Numerical Simulation},
  84:105162, 2020.

\bibitem[BM05]{BarrabesMikkola05}
E.~Barrab{\'{e}}s and S.~Mikkola.
\newblock {Families of periodic horseshoe orbits in the restricted three-body
  problem}.
\newblock {\em Astronomy {\&} Astrophysics}, 432(3):1115--1129, 2005.

\bibitem[BM06]{BolMac06}
S.~V. Bolotin and R.~S. MacKay.
\newblock Nonplanar {Second} {Species} {Periodic} and {Chaotic} {Trajectories}
  for the {Circular} {Restricted} {Three}-{Body} {Problem}.
\newblock {\em Celestial Mech. Dynam. Astronom.}, 94(4):433--449, 2006.

\bibitem[BMO09]{BMO09}
E.~Barrab{\'{e}}s, J.~M. Mondelo, and M.~Oll{\'{e}}.
\newblock {Dynamical aspects of multi-round horseshoe-shaped homoclinic orbits
  in the RTBP}.
\newblock {\em Celestial Mech. Dynam. Astronom.}, 105(1-3):197--210, 2009.

\bibitem[BO06]{BarrabesOlle2006}
E.~Barrab{\'{e}}s and M.~Oll{\'{e}}.
\newblock {Invariant manifolds of L3 and horseshoe motion in the restricted
  three-body problem}.
\newblock {\em Nonlinearity}, 19:2065--2089, 2006.

\bibitem[Bol06]{Bol06}
S.~Bolotin.
\newblock Symbolic dynamics of almost collision orbits and skew products of
  symplectic maps.
\newblock {\em Nonlinearity}, 19(9):2041--2063, 2006.

\bibitem[Bro11]{Brown1911}
E.~W. Brown.
\newblock {Orbits, Periodic, On a new family of periodic orbits in the problem
  of three bodies}.
\newblock {\em Monthly Notices of the Royal Astronomical Society}, 71:438--454,
  1911.

\bibitem[BRS03]{BRS03}
P.~Bernard, C.~G. Ragazzo, and P.~Salamao.
\newblock Homoclinic orbits near saddle-center fixed points of {Hamiltonian}
  systems with two degrees of freedom.
\newblock {\em Ast{\'e}risque}, 286:151--165, 2003.

\bibitem[Cap12]{Cap12}
M.~J. Capi{\'{n}}ski.
\newblock {Computer Assisted Existence Proofs of Lyapunov Orbits at L2 and
  Transversal Intersections of Invariant Manifolds in the Jupiter--Sun PCR3BP}.
\newblock {\em SIAM Journal on Applied Dynamical Systems}, 11(4):1723--1753,
  2012.

\bibitem[CH03]{CorsHall03}
J.~M. Cors and G.~R. Hall.
\newblock Coorbital periodic orbits in the three body problem.
\newblock {\em SIAM Journal on Applied Dynamical Systems}, 2(2):219--237, 2003.

\bibitem[CP11]{ChPi11}
L.~Chierchia and G.~Pinzari.
\newblock The planetary {N}-body problem: symplectic foliation, reductions and
  invariant tori.
\newblock {\em Inventiones mathematicae}, 186(1):1--77, 2011.

\bibitem[CPY19]{CPY19}
J.~M. Cors, J.~F. Palaci\'{a}n, and P.~Yanguas.
\newblock On co-orbital quasi-periodic motion in the three-body problem.
\newblock {\em SIAM J. Appl. Dyn. Syst.}, 18(1):334--353, 2019.

\bibitem[DM81a]{DerMur81a}
S.~F. Dermott and C.~D. Murray.
\newblock The dynamics of tadpole and horseshoe orbits. {I} - {Theory}.
\newblock {\em Icarus}, 48(1):1--11, 1981.

\bibitem[DM81b]{DerMur81b}
S.~F. Dermott and C.~D. Murray.
\newblock The dynamics of tadpole and horseshoe orbits. {II} - {The} coorbital
  satellites of {Saturn}.
\newblock {\em Icarus}, 48(1):12--22, 1981.

\bibitem[Dua08]{Dua08}
P.~Duarte.
\newblock Elliptic isles in families of area-preserving maps.
\newblock {\em Ergodic Theory and Dynamical Systems}, 28(6):1781--1813, 2008.

\bibitem[F\'04]{Fejoz04}
J.~F\'{e}joz.
\newblock D\'{e}monstration du `th\'{e}or\`eme d'{A}rnold' sur la stabilit\'{e}
  du syst\`eme plan\'{e}taire (d'apr\`es {H}erman).
\newblock {\em Ergodic Theory Dynam. Systems}, 24(5):1521--1582, 2004.

\bibitem[FGKR16]{FGKR16}
J.~F{\'{e}}joz, M.~Guardia, V.~Kaloshin, and P.~Roldan.
\newblock {Kirkwood gaps and diffusion along mean motion resonances in the
  restricted planar three-body problem}.
\newblock {\em Journal of the European Mathematical Society},
  18(10):2313--2401, 2016.

\bibitem[GGSZ21]{GGSZ21}
O.~M.~L. Gomide, M.~Guardia, T.~M. Seara, and C.~Zeng.
\newblock On small breathers of nonlinear {Klein}-{Gordon} equations via
  exponentially small homoclinic splitting.
\newblock {\em arXiv preprint arXiv:2107.14566}, 2021.

\bibitem[GJMS01]{GJMS01v4}
G.~G{\'o}mez, A.~Jorba, J.~J. Masdemont, and C.~Sim{\'o}.
\newblock {\em {Dynamics And Mission Design Near Libration Points - Volume 4:
  Advanced Methods For Triangular Points}}, volume~5.
\newblock World Scientific, 2001.

\bibitem[GK12]{GorodetskiK12}
A.~Gorodetski and V.~Kaloshin.
\newblock Hausdorff dimension of oscillatory motions for restricted three body
  problems.
\newblock Preprint, 2012.

\bibitem[GMPS22]{GuardiaMPS22}
M.~Guardia, P.~Martín, J.~Paradela, and T.~M. Seara.
\newblock Hyperbolic dynamics and oscillatory motions in the 3 body problem.
\newblock Preprint arXiv:2207.14351, 2022.

\bibitem[GMS16]{GMS16}
M.~Guardia, P.~Mart{\'{i}}n, and T.~M. Seara.
\newblock Oscillatory motions for the restricted planar circular three body
  problem.
\newblock {\em Inventiones mathematicae}, 203(2):417--492, 2016.

\bibitem[Gor12]{Gorodetski12}
A.~Gorodetski.
\newblock On stochastic sea of the standard map.
\newblock {\em Comm. Math. Phys.}, 309(1):155--192, 2012.

\bibitem[GPSV21]{GPSV21}
M.~Guardia, J.~Paradela, T.~M. Seara, and C.~Vidal.
\newblock Symbolic dynamics in the restricted elliptic isosceles three body
  problem.
\newblock {\em Journal of Differential Equations}, 294:143--177, 2021.

\bibitem[GSMS17]{GMSS17}
M.~Guardia, T.~M. Seara, Pau Mart{\'{i}}n, and L.~Sabbagh.
\newblock Oscillatory orbits in the restricted elliptic planar three body
  problem.
\newblock {\em Discrete \& Continuous Dynamical Systems - A}, 37(1):229, 2017.

\bibitem[GZ19]{GierZgli19}
A.~Gierzkiewicz and P.~Zgliczy{\'{n}}ski.
\newblock A computer-assisted proof of symbolic dynamics in {Hyperion’s}
  rotation.
\newblock {\em Celestial Mech. Dynam. Astronom.}, 131(7):1--17, 2019.

\bibitem[Her98]{Her98}
M.~Herman.
\newblock Some open problems in dynamical systems.
\newblock In {\em Proceedings of the International Congress of Mathematicians},
  volume~2, pages 797--808. Berlin, 1998.

\bibitem[HTL07]{HTL07}
X.~Hou, J.~Tang, and L.~Liu.
\newblock Transfer to the collinear libration point {L3} in the
  {Sun}--{Earth}+{Moon} system.
\newblock {\em Nasa Technical Report}, 2007.

\bibitem[JBL16]{JBL16}
T.~J\'ez\'equel, P.~Bernard, and E.~Lombardi.
\newblock Homoclinic connections with many loops near a {$0^2 iw$} resonant
  fixed point for {Hamiltonian} systems.
\newblock {\em Discrete and Continuous Dynamical Systems}, 36(6):3153--3225,
  2016.

\bibitem[JN20]{JorNic20}
A.~Jorba and B.~Nicol{\'a}s.
\newblock {Transport and invariant manifolds near L3 in the Earth-Moon
  Bicircular model}.
\newblock {\em Commun. Nonlinear Sci. Numer. Simul.}, 89:19pp, 2020.

\bibitem[JN21]{JorNic21}
{\`A}.~Jorba and B.~Nicol{\'a}s.
\newblock Using invariant manifolds to capture an asteroid near the {L3} point
  of the {Earth}-{Moon} {Bicircular} model.
\newblock {\em Commun. Nonlinear Sci. Numer. Simul.}, 102:23pp, 2021.

\bibitem[KH95]{KatokHasselblatt}
A.~Katok and B.~Hasselblatt.
\newblock {\em {Introduction to the Modern Theory of Dynamical Systems}}.
\newblock Cambridge University Press, 1995.

\bibitem[Ler91]{Ler91}
L.~M. Lerman.
\newblock Hamiltonian systems with loops of a separatrix of a saddle-center.
\newblock {\em Selecta Math. Sov}, 10(1991):297--306, 1991.

\bibitem[LO01]{LlibreOlle01}
J.~Llibre and M.~Oll{\'{e}}.
\newblock {The motion of Saturn coorbital satellites in the restricted
  three-body problem}.
\newblock {\em Astronomy \& astrophysics}, 378(3):1087--1099, 2001.

\bibitem[Lom99]{Lom99}
E~Lombardi.
\newblock Non-persistence of homoclinic connections for perturbed integrable
  reversible systems.
\newblock {\em J. Dynam. Differential Equations}, 11(1):129--208, 1999.

\bibitem[Lom00]{Lom00}
E.~Lombardi.
\newblock {\em Oscillatory {Integrals} and {Phenomena} {Beyond} all {Algebraic}
  {Orders}: with {Applications} to {Homoclinic} {Orbits} in {Reversible}
  {Systems}}.
\newblock Springer, 2000.

\bibitem[LS80]{LlibSim80}
J.~Llibre and C.~Sim{\'o}.
\newblock Oscillatory solutions in the planar restricted three-body problem.
\newblock {\em Mathematische Annalen}, 248(2):153--184, 1980.

\bibitem[MHO92]{MHO92}
A.~Mielke, P.~Holmes, and O.~O'Reilly.
\newblock Cascades of homoclinic orbits to, and chaos near, a {H}amiltonian
  saddle-center.
\newblock {\em J. Dynam. Differential Equations}, 4(1):95--126, 1992.

\bibitem[MO17]{MeyerHallOffin}
K.~R. Meyer and D.~C. Offin.
\newblock {\em {Introduction to Hamiltonian Dynamical Systems and the N-Body
  Problem}}, volume~90.
\newblock Springer Science+Business Media, 2017.

\bibitem[Moe89]{Moe89}
R.~Moeckel.
\newblock Chaotic dynamics near triple collision.
\newblock {\em Archive for Rational Mechanics and Analysis}, 107(1):37--69,
  1989.

\bibitem[Moe07]{Moe07}
R.~Moeckel.
\newblock Symbolic dynamics in the planar three-body problem.
\newblock {\em Regular and Chaotic Dynamics}, 12(5):449--475, 2007.

\bibitem[Mor02]{morbidelli2002}
A.~Morbidelli.
\newblock {\em Modern celestial mechanics. Aspects of solar system dynamics}.
\newblock CRC Press, 2002.

\bibitem[Mos58]{Moser58}
J.~Moser.
\newblock On the generalization of a theorem of {A}. {Liapounoff}.
\newblock {\em Communications on Pure and Applied Mathematics}, 11(2):257--271,
  1958.

\bibitem[Mos01]{Moser2001}
J.~Moser.
\newblock {\em {Stable and random motions in dynamical systems : with special
  emphasis on celestial mechanics}}.
\newblock Princeton University Press, 2001.

\bibitem[New70]{Newhouse70}
S.~E. Newhouse.
\newblock Nondensity of axiom {${\rm A}({\rm a})$} on {$S\sp{2}$}.
\newblock In {\em Global {A}nalysis ({P}roc. {S}ympos. {P}ure {M}ath., {V}ols.
  {XIV}, {XV}, {XVI}, {B}erkeley, {C}alif., 1968)}, volume XIV-XVI of {\em
  Proc. Sympos. Pure Math.}, pages 191--202. Amer. Math. Soc., Providence, RI,
  1970.

\bibitem[NPR20]{NPR20}
L.~Niederman, A.~Pousse, and P.~Robutel.
\newblock On the {Co}-orbital {Motion} in the {Three}-{Body} {Problem}:
  {Existence} of {Quasi}-periodic {Horseshoe}-{Shaped} {Orbits}.
\newblock {\em Communications in Mathematical Physics}, 377(1):551--612, 2020.

\bibitem[Rag97a]{Rag97a}
C.~G. Ragazzo.
\newblock Irregular dynamics and homoclinic orbits to {Hamiltonian} saddle
  centers.
\newblock {\em Communications on Pure and Applied Mathematics}, 50(2):105--147,
  1997.

\bibitem[Rag97b]{Rag97b}
C.~G. Ragazzo.
\newblock On the stability of double homoclinic loops.
\newblock {\em Communications in mathematical physics}, 184(2):251--272, 1997.

\bibitem[Rob95]{Rob95}
P.~Robutel.
\newblock Stability of the planetary three-body problem. {II}. {KAM} theory and
  existence of quasiperiodic motions.
\newblock {\em Celestial Mech. Dynam. Astronom.}, 62(3):219--261, 1995.

\bibitem[Sit60]{Sitnikov1960}
K.~Sitnikov.
\newblock The existence of oscillatory motions in the three-body problem.
\newblock In {\em {Dokl. Akad. Nauk SSSR}}, volume 133, pages 303--306, 1960.

\bibitem[Sma67]{Smale67}
S.~Smale.
\newblock {Differentiable Dynamical Systems}.
\newblock {\em Bull. of the AMS}, 73(6):747--817, 1967.

\bibitem[SSST13]{SSST13}
C.~Sim{\'{o}}, P.~Sousa-Silva, and M.~Terra.
\newblock Practical {Stability} {Domains} {Near} {L4},5 in the {Restricted}
  {Three}-{Body} {Problem}: {Some} {Preliminary} {Facts}.
\newblock In {\em Progress and {Challenges} in {Dynamical} {Systems}}, pages
  367--382. Springer, 2013.

\bibitem[Sze67]{Szebehely}
V.~G. Szebehely.
\newblock {\em {Theory of orbits : the restricted problem of three bodies}}.
\newblock Academic Press, New York [etc.], 1967.

\bibitem[TFR{\etalchar{+}}10]{TFRPGM10}
M.~Tantardini, E.~Fantino, Y.~Ren, P.~Pergola, G.~G{\'o}mez, and J.~J.
  Masdemont.
\newblock Spacecraft trajectories to the {L3} point of the {Sun}--{Earth}
  three-body problem.
\newblock {\em Celestial Mech. Dynam. Astronom.}, 108(3):215--232, 2010.

\bibitem[WZ03]{WilcZgli03}
D.~Wilczak and P.~Zgliczy{\'n}ski.
\newblock {Heteroclinic connections between periodic orbits in planar
  restricted circular three-body problem -- a computer assisted proof}.
\newblock {\em Comm. Math. Phys.}, 234(1):37--75, 2003.

\end{thebibliography}

\end{document}